\documentclass[draft]{ws-m3as}
\usepackage{amssymb,amsmath,amsfonts}
\usepackage{amscd} 
\usepackage[utf8]{inputenc}
\usepackage{enumerate}
\usepackage{graphicx}
\usepackage{tabularx}
\usepackage{color}
\renewcommand{\paragraph}[1]{\smallskip \noindent\textbf{#1}}
\newcommand{\MyReference}[1]{Ref.\,\refcite{#1}}
\newcommand{\colred}[1]{{\color{red}#1}}

\newcommand{\set}[1]{\mathbb{#1}}
\newcommand{\NN}{\set{N}}
\newcommand{\RR}{\set{R}}
\newcommand{\ee}{\text{e}}
\newcommand{\dd}{{\rm{d}}}
\newcommand{\abs}[1]{|#1|}
\newcommand{\absbig}[1]{\big|#1\big|}
\newcommand{\absBig}[1]{\Big|#1\Big|}
\newcommand{\scpbig}[3]{\big(#2\big\vert#3\big)_{#1}}
\newcommand{\scpBig}[3]{\Big(#2\Big\vert#3\Big)_{#1}}
\newcommand{\norm}[2]{\|#2\|_{#1}}
\newcommand{\normbig}[2]{\big\|#2\big\|_{#1}}
\newcommand{\normBig}[2]{\Big\|#2\Big\|_{#1}}
\newcommand{\nice}[1]{\mathcal{#1}}

\newcommand{\nM}{\nice{M}}
\newcommand{\nN}{\nice{N}}
\newcommand{\nO}{\nice{O}}
\newcommand{\nT}{\nice{T}}
\newcommand{\nLa}{\nice{L}^{(a)}}
\newcommand{\nLak}{\nice{L}^{(a_k)}}
\newcommand{\nLzero}{\nice{L}^{(0)}}
\newcommand{\Lp}{L_p}
\newcommand{\Lq}{L_q}

\newcommand{\Lone}{L_1}
\newcommand{\Ltwo}{L_2}
\newcommand{\Lfour}{L_4}
\newcommand{\Lsix}{L_6}
\newcommand{\Linfty}{L_{\infty}}
\newcommand{\Wpk}{W_p^{k}}
\newcommand{\Wtwok}{W_2^{k}}
\newcommand{\Wfourone}{W_4^{1}}
\newcommand{\Winftyone}{W_{\infty}^1}
\newcommand{\Winftytwo}{W_{\infty}^2}
\newcommand{\Winftythree}{W_{\infty}^3}

\newcommand{\Hminus}{H^{-1}}
\newcommand{\Hk}{H^k}
\newcommand{\Hone}{H^1}
\newcommand{\Honezero}{H^1_0}
\newcommand{\Htwo}{H^2}
\newcommand{\Hthree}{H^3}
\newcommand{\Hfour}{H^4}
\newcommand{\Htwodiamond}{\Htwo_{\diamond}}
\newcommand{\Hthreediamond}{\Hthree_{\diamond}}
\newcommand{\wX}{\widetilde{X}}
\newcommand{\vecx}{\vec{x}}
\newcommand{\vecv}{\vec{v}}
\newcommand{\vecA}{\vec{A}}
\renewcommand{\Pr}{\text{Pr}}
\newcommand{\bb}{\beta}
\newcommand{\bba}{\beta^{(a)}}
\newcommand{\bbak}{\beta^{(a_k)}}
\newcommand{\bbzero}{\beta^{(0)}}
\newcommand{\psia}{\psi^{(a)}}
\newcommand{\psiak}{\psi^{(a_k)}}
\newcommand{\psizero}{\psi^{(0)}}
\newcommand{\phik}{\phi^{(k)}}
\newcommand{\psik}{\psi^{(k)}}

\newcommand{\phistar}{\phi_{*}}
\newcommand{\psistar}{\psi_{*}}
\newcommand{\alphaphi}{\alpha^{(\phi)}}
\newcommand{\rphi}{r^{(\phi)}}
\newcommand{\wE}{\widetilde{E}}
\newcommand{\MywEphipsi}[2]{\widetilde{E}_{#1}\big(\phi(#2), \psi(#2)\big)}
\newcommand{\MywEphipsiAtzero}[1]{\widetilde{E}_{#1}(\psi_0, \psi_0)}
\newcommand{\MyEphipsi}[2]{E_{#1}\big(\phi(#2), \psi(#2)\big)}
\newcommand{\MyEphiphi}[2]{E_{#1}\big(\phi(#2), \phi(#2)\big)}
\newcommand{\MyEphipsiAtzero}[1]{E_{#1}(\psi_0, \psi_0)}
\newcommand{\BoundE}[1]{\overline{E}_{#1}}
\newcommand{\Bounde}[1]{\overline{e}_{#1}}
\newcommand{\EE}{\nice{E}}
\newcommand{\BEMERKUNGBARBARA}[1]{\colred{}}
\begin{document}
\markboth{Barbara Kaltenbacher, Mechthild Thalhammer}{Fundamental models in nonlinear acoustics: Part~I. Analytical comparison}
\title{FUNDAMENTAL MODELS IN NONLINEAR ACOUSTICS \\ PART~I. ANALYTICAL COMPARISON}
\author{BARBARA KALTENBACHER}
\address{Alpen-Adria-Universität Klagenfurt, Institut für Mathematik, \\ Universitätsstraße 65–67, 9020~Klagenfurt, Austria. \\
Barbara.Kaltenbacher@aau.at}
\author{MECHTHILD THALHAMMER}
\address{Leopold–Franzens-Universität Innsbruck, Institut für Mathematik, \\ Technikerstraße 13/VII, 6020~Innsbruck, Austria. \\
Mechthild.Thalhammer@uibk.ac.at}
\maketitle
\begin{abstract}
This work is concerned with the study of fundamental models from nonlinear acoustics. 
In Part~I, a hierarchy of nonlinear damped wave equations arising in the description of sound propagation in thermoviscous fluids is deduced. 
In particular, a rigorous justification of two classical models, the Kuznetsov and Westervelt equations, retained as limiting systems for consistent initial data, is given. 
Numerical comparisons that confirm and complement the theoretical results are provided in Part~II.
\end{abstract}
\section{Introduction}
\paragraph{Nonlinear damped wave equations.}
In the present work, we study a hierarchy of higher-order nonlinear damped wave equations that arise in the modelling of sound propagation in thermoviscous fluids, see Table~\ref{Table1}.
Employing a reformulation as abstract evolution equation for the acoustic velocity potential, our most fundamental model takes the compact form 
\begin{equation}
\label{eq:Equation1}
\begin{split}
&\partial_{ttt} \psi(t) 
- \bb_1 \, \Delta \partial_{tt} \psi(t) 
+ \bb_2 \, \Delta^2 \partial_{t} \psi(t) 
- \bb_3 \, \Delta \partial_{t} \psi(t) 
+ \bb_4 \, \Delta^2 \psi(t) \\
&\quad + \partial_{tt} \Big(\tfrac{1}{2} \, \bb_5 \, \big(\partial_{t} \psi(t)\big)^2 + \bb_6 \, \abs{\nabla \psi(t)}^2\Big) = 0\,;
\end{split}
\end{equation}
the positive coefficients $\bb_1, \dots, \bb_6 > 0$ are defined by decisive physical quantities such as the mean value of the mass density, the speed of sound, the viscosity, the thermal conductivity, and the parameter of nonlinearity, see also Table~\ref{Table2}.
As this equation only marginally extends a nonlinear damped wave equation deduced in \MyReference{Brunnhuber2015}~(Eq.\,(1.19)) and \MyReference{BrunnhuberJordan2016}~(Eq.\,(4)), we refer to it as \textsc{Blackstock--Crighton--Brunnhuber--Jordan--Kuznetsov equation} or briefly as \textsc{Brunnhuber--Jordan--Kuznetsov equation}. 
Various models known from the literature are embedded in our most fundamental model as reduced models; 
a central question of this work is to rigorously justify that two classical models, the Kuznetsov and Westervelt equations, are retained as limiting systems for vanishing thermal conductivity, provided that the initial data satisfy suitable consistency conditions.

\begin{table}[t!]
{\scriptsize
\begin{equation*}
\begin{CD}
\text{Brunnhuber--Jordan--Kuznetsov (BJK) } @>\sigma = 0>> \text{ Brunnhuber--Jordan--Westervelt (BJW)} \\ 
@VV \sigma_0 = 0 V @VV \sigma_0 = 0 V \\
\text{Blackstock--Crighton--Kuznetsov (BCK) } @>\sigma = 0>> \text{ Blackstock--Crighton--Westervelt (BCW)} \\ 
@VV a = 0 V @VV a = 0 V \\
\text{Kuznetsov (K) } @>\sigma = 0>> \text{ Westervelt (W)} \\ 
\end{CD}
\end{equation*}}
\caption{Hierachy of nonlinear damped wave equations. The Kuznetsov and Westervelt equations are retained as limiting systems for consistent initial data.}
\label{Table1}
\end{table}

\begin{table}[t!]
{\footnotesize 
Decisive physical quantities
\begin{quote}
\text{Mass density $\varrho = \varrho_0 + \varrho_{\sim}$}  \\
\text{Acoustic particle velocity $\vecv = \vecv_{\sim}$}  \\
\text{Associated acoustic velocity and vector potentials $\vecv_{\sim} = \nabla \psi + \nabla \times \vecA$}  \\
\text{Acoustic pressure $p = p_0 + p_{\sim}$}  \\
\text{Temperature $T = T_0 + T_{\sim}$}  \\
\text{Shear (or dynamic) viscosity~$\mu$}  \\
\text{Bulk viscosity~$\mu_B$}  \\
\text{Kinematic viscosity $\nu = \tfrac{\mu}{\varrho_0}$}  \\
\text{Prandtl number Pr} \\ 
\text{Thermal conductivity $a = \tfrac{\nu}{\Pr}$}  \\
\text{Specific heat at constant volume~$c_V$} \\
\text{Specific heat at constant pressure~$c_p$} \\
\text{Thermal expansion coefficient~$\alpha_V$} \\
\text{Speed of sound $c_0 = \sqrt{\tfrac{c_p p_0}{c_V \varrho_0}}$} \\
Parameter of nonlinearity~$\tfrac{B}{A}$ 
\end{quote}
Auxiliary abbreviations and relations 
\begin{quote}
$A = c_0^2 \, \varrho_0$ \\
$\tfrac{a}{c_V \varrho_0} = a \, (1 + \tfrac{B}{A})$ \\
$\Lambda = \tfrac{\mu_B}{\mu} + \tfrac{4}{3}$ \\
$\bba_1 = a \, \big(1 + \tfrac{B}{A}\big) + \nu \Lambda$ \\ 
$\bba_2(\sigma_0) = a \, \big(\nu \Lambda + a \, \tfrac{B}{A} + \sigma_0 \, \tfrac{B}{A} \, (\nu \Lambda - a)\big)$ with $\sigma_0 \in \{0,1\}$ \\
$\bb_3 = c_0^2$ \\
$\bba_4(\sigma_0) = a \, \big(1 + \sigma_0 \, \tfrac{B}{A}\big) \, c_0^2$ with $\sigma_0 \in \{0,1\}$ \\
$\bb_5(\sigma) = \tfrac{1}{c_0^2} \, \big(2 \, (1 - \sigma) + \tfrac{B}{A}\big)$ with $\sigma \in \{0,1\}$ \\
$\bb_6(\sigma) = \sigma$ with $\sigma \in \{0,1\}$ \\
$\bba_0(\sigma_0) = \frac{\bba_2(\sigma_0)}{\bba_4(\sigma_0)} = \tfrac{1}{c_0^2} \, \big(\nu \Lambda + (1 - \sigma_0) \, a \, \tfrac{B}{A}\big)$ with $\sigma_0 \in \{0,1\}$ \\
$\alpha = 1 + \bb_5(\sigma) \, \partial_{t} \psi$ with $\sigma \in \{0,1\}$ \\ 
$r = \bb_5(\sigma) \, \big(\partial_{tt} \psi\big)^2 + \bb_6(\sigma) \, \partial_{tt} \abs{\nabla \psi}^2$ with $\sigma \in \{0,1\}$ 
\end{quote}}
\caption{Decisive physical quantities and auxiliary abbreviations.}
\label{Table2}
\end{table}

\paragraph{Outline.}
Our work has the following structure. 
Basic notation and assumptions are introduced in Section~\ref{SectionNotationAssumptions}.
In Section~\ref{SectionModels}, we state the considered hierarchy of models; 
further details on the derivation of the most general model are found in~\ref{SectionAppendix}. 
Our main result on limiting systems is deduced in Section~\ref{SectionLimit};
auxiliary reformulations and a priori energy estimates are provided in Section~\ref{SectionAuxiliaries}.
\subsection{Basic notation and assumptions}	
\label{SectionNotationAssumptions}
In the following, we recall standard abbreviations and basic assumptions that are used throughout without further mention. 

\paragraph{Space and time domain.}
We assume that the considered spatial domain $\Omega \subset \RR^d$ is bounded and that its boundary~$\partial \Omega$ is sufficiently regular. 
In Sections~\ref{SectionModels} to~\ref{SectionLimit}, we are primarily interested in the most relevant three-dimensional case; 
however, with regard to numerical illustrations, we admit $d \in \{1,2,3\}$.
As time domain, we consider a bounded interval $[0, T]$; 
under certain regularity, compatibility, and smallness requirements on the prescribed initial data, existence and uniqueness of the solution to~\eqref{eq:Equation1} subject to homogeneous Dirichlet boundary conditions is ensured, see Proposition~\ref{thm:Proposition} and Remark~\ref{Remark1}.

\paragraph{Euclidian norm.}
Let $v = (v_1, \dots, v_d)^T \in \RR^d$ and $w = (w_1, \dots, w_d)^T \in \RR^d$.
As usual, the Euclidian inner product and the associated norm are denoted by 
\begin{equation*}
v \cdot w = \sum_{j=1}^{d} v_j \, w_j\,, \quad \abs{v} = \sqrt{v \cdot v}\,.
\end{equation*}

\paragraph{Space derivatives.}
For scalar-valued and vector-valued functions 
\begin{equation*}
\begin{gathered}
f: \Omega \longrightarrow \RR: x = (x_1, \dots, x_d)^T \longmapsto f(x)\,, \\
F: \Omega \longrightarrow \RR^d: x = (x_1, \dots, x_d)^T \longmapsto F(x) = \big(F_1(x), \dots, F_d(x)\big)^T\,,
\end{gathered}
\end{equation*} 
we denote by $(\partial_{x_j} f)_{j = 1}^{d}$ and $(\partial_{x_j} F_k)_{j,k = 1}^{d}$ their spatial derivatives. 
Gradient, Laplacian, and divergence are defined by 
\begin{equation*} 
\nabla f = \big(\partial_{x_1} f, \dots, \partial_{x_d} f\big)^T\,, \quad \Delta f = \sum_{j=1}^{d} \partial_{x_j}^2 f\,, \quad 
\nabla \cdot F = \sum_{j=1}^{d} \partial_{x_j} F_j\,. 
\end{equation*}

\paragraph{Lebesgue and Sobolev spaces.}
For exponents $p \in [1, \infty]$ and $k \in \NN_{\geq 1}$, we denote by $\Lp(\Omega, \RR)$ and $\Wpk(\Omega, \RR)$ the standard Lebesgue and Sobolev spaces;
as common, we set $\Hk(\Omega, \RR) = \Wtwok(\Omega, \RR)$.
In particular, the Hilbert space~$\Ltwo(\Omega, \RR)$ is endowed with inner product and associated norm given by 
\begin{equation*}
\scpbig{\Ltwo}{f}{g} = \int_{\Omega} f(x) \, g(x) \; \dd x\,, \quad 
\normbig{\Ltwo}{f} = \sqrt{\int_{\Omega} \big(f(x)\big)^2 \; \dd x}\,, \quad f,g \in \Ltwo(\Omega, \RR)\,;
\end{equation*}
accordingly, for vector-valued functions that arise in connection with the gradient, we set 
\begin{equation*}
\scpbig{\Ltwo}{F}{G} = \int_{\Omega} F(x) \cdot G(x) \; \dd x\,, \quad 
\normbig{\Ltwo}{F} = \sqrt{\int_{\Omega} \absbig{F(x)}^2 \; \dd x}\,, \quad F, G \in \Ltwo\big(\Omega, \RR^d\big)\,.
\end{equation*}

\paragraph{Lebesgue--Bocher spaces.}
In Section~\ref{SectionAuxiliaries}, we employ reformulations of the considered nonlinear damped wave equations as abstract evolution equations on Banach spaces and 
deduce a priori estimates with respect to the norms of different Bochner--Lebesgue spaces such as 
\begin{equation*}
\normbig{\Lp([0, T], \Lq(\Omega))}{\varphi} = \bigg(\int_{0}^{T} \normbig{\Lq}{\varphi(t)}^p \; \dd t\bigg)^{\frac{1}{p}}\,,  \quad 
\varphi \in \Lp\big([0, T], \Lq(\Omega)\big)\,.
\end{equation*}
\section{Fundamental models}	
\label{SectionModels}
In this section, we introduce fundamental models arising in nonlinear acoustics, the Blackstock--Crighton--Brunnhuber--Jordan--Kuznetsov or briefly Brunnhuber--Jordan--Kuznetsov (BJK) equation, the Blackstock--Crighton--Kuznetsov (BCK) equation, the Kuznetsov (K) equation, the Blackstock--Crighton--Brunnhuber--Jordan--Westervelt of briefly Brunnhuber--Jordan--Westervelt (BJW) equation, the Blackstock--Crighton--Westervelt (BCW) equation, and the Westervelt (W) equation;
these nonlinear damped wave equations form a hierarchy in the sense that some of them can be viewed as special cases of others, see Table~\ref{Table1}.
In Section~\ref{SectionModels1}, we specify the physical and mathematical principles employed in the derivation of the Brunnhuber--Jordan--Kuznetsov equation, which is the most general model studied in this work and provides the basis for reduced models such as the Kuznetsov and Westervelt equations. 
In Section~\ref{SectionModels2}, we review the considered nonlinear damped wave equations and put them into relation.
Our collection of models is by no means complete, and we refer to \MyReference{Kaltenbacher2015} for recent references from the active field of modelling in nonlinear acoustics as well as to the classical works \MyReference{Crighton1979}, \refcite{EnfloHedberg2006}, \refcite{HamiltonBlackstock1998}, \refcite{Kuznetsov1971}, \refcite{MakarovOchmann1996}, \refcite{MakarovOchmann1997a}, \refcite{MakarovOchmann1997b}, \refcite{Pierce1989}.
\subsection{Derivation of Brunnhuber--Jordan--Kuznetsov equation}	
\label{SectionModels1}
\paragraph{Notation.}
We meanwhile employ the notation~$\vecx$, $\vecv$, and~$\vecA$ so that the distinction between scalar-valued and vector-valued quantities becomes evident. 

\paragraph{Physical quantities.}
The main physical quantities for the description of sound propagation in thermoviscous fluids are the mass density~$\varrho$, the acoustic particle velocity~$\vecv$, the acoustic pressure~$p$, and the temperature~$T$. 
These space-time-dependent quantities are decomposed into their mean values and fluctuations
\begin{equation*}
\begin{gathered}
\varrho(\vecx,t) = \varrho_0 + \varrho_{\sim}(\vecx,t)\,, \quad \vecv(\vecx,t) = \vecv_0 + \vecv_{\sim}(\vecx,t) = \vecv_{\sim}(\vecx,t)\,, \\
p(\vecx,t) = p_0 + p_{\sim}(\vecx,t)\,, \quad T(\vecx,t) = T_0 + T_{\sim}(\vecx,t)\,;
\end{gathered}
\end{equation*}
in the situation relevant here, the mean value of the acoustic particle velocity may be assumed to vanish. 

\paragraph{Physical principles.}
A system of time-dependent nonlinear partial differential equations governing the interplay of these quantities results from the conservation laws for mass, momentum, and energy, supplemented with an equation of state.  
The conservation of mass is reflected by the continuity equation 
\begin{subequations}
\label{eq:ConservationLaws}
\begin{equation}
\label{eq:ConservationMass}
\partial_{t} \varrho + \nabla \cdot (\varrho \, \vecv) = 0\,. 
\end{equation}
The conservation of momentum corresponds to the relation 
\begin{equation}
\label{eq:ConservationMomentum}
\partial_{t} (\varrho \, \vecv) + \vecv \, \nabla \cdot (\varrho \, \vecv) + \varrho \, (\vecv \cdot \nabla) \, \vecv + \nabla p
= \mu \, \Delta \vecv + \big(\mu_B + \tfrac{1}{3} \, \mu\big) \, \nabla (\nabla \cdot \vecv)\,, 
\end{equation}
where~$\mu$ and~$\mu_B$ denote the shear and bulk viscosity, respectively.
The relation describing the conservation of energy reads 
\begin{equation*}
\varrho \, (\partial_{t} E + \vecv \cdot \nabla E) + p \, \nabla \cdot \vecv 
= a \, \Delta T + \big(\mu_B - \tfrac{2}{3} \, \mu\big) \, (\nabla \cdot \vecv)^2 + \tfrac{1}{2} \, \mu \, \normbig{F}{\nabla \vecv + (\nabla\vecv)^T}^2\,, 
\end{equation*}
see Eq.\,(3c) in \MyReference{Blackstock1963}.
Here, $E$ denotes the internal energy per unit mass and $a = \frac{\nu}{\mathrm{Pr}}$ the thermal conductivity, defined by the kinematic viscosity $\nu = \tfrac{\mu}{\varrho_0}$ and the Prandtl number~$\mathrm{Pr}$;
the subscript~$F$ indicates that the Frobenius norm is used. 
Rewriting the left hand side of this equation by means of the specific heat at constant volume and pressure, $c_V$ and~$c_p$, respectively, as well as the thermal expansion coefficient~$\alpha_V$, the conservation of energy is given by 
\begin{equation}
\label{eq:ConservationEnergy}
\begin{split}
&\varrho \, (c_V \, \partial_{t} T + c_V \, \vecv \cdot \nabla T + \tfrac{c_p-c_V}{\alpha_V} \, \nabla \cdot \vecv) \\
&\quad = a \, \Delta T + \big(\mu_B - \tfrac{2}{3} \, \mu\big) \, (\nabla \cdot \vecv)^2 
+ \tfrac{1}{2} \, \mu \, \normbig{F}{\nabla \vecv + (\nabla\vecv)^T}^2\,, 
\end{split}
\end{equation}
\end{subequations}
see Eq.\,(3c') in \MyReference{Blackstock1963}.
The heuristic equation of state for the acoustic pressure in dependence of mass density and temperature is approximated by the first terms of a Taylor-like expansion 
\begin{equation}
\label{eq:EquationState}
p_{\sim} \approx  A \, \tfrac{\varrho_{\sim}}{\varrho_0} + \tfrac{B}{2} \, \big(\tfrac{\varrho_{\sim}}{\varrho_0}\big)^2 + \hat{A} \, \tfrac{T_{\sim}}{T_0} 
\end{equation}
involving certain positive coefficients $A, B, \hat{A} > 0$, see 
Eq.\,(5d) in \MyReference{Blackstock1963} 
and also Table~\ref{Table2}.

\paragraph{Helmholtz decomposition.}
A Helmholtz decomposition of the acoustic particle velocity into an irrotational and a solenoidal part
\begin{equation}
\label{eq:Helmholtz}
\vecv_{\sim} = \nabla \psi + \nabla \times \vecA
\end{equation}
leads to a reformulation of the conservation laws~\eqref{eq:ConservationLaws} in terms of the acoustic velocity potential~$\psi$ and the vector potential~$\vecA$. 
We note that some authors use instead the relation $\vecv_{\sim} = - \nabla \psi + \nabla \times \vecA$ which explains a differing sign in the resulting nonlinear damped wave equations.

\paragraph{Derivation of reduced models.}
In order to derive reduced models from~\eqref{eq:ConservationLaws}-\eqref{eq:EquationState}, three categories of contributions are distinguished. 
First, terms that are linear with respect to the fluctuating quantities and not related to dissipative effects are taken into account (first-order contributions). 
Second, quadratic terms with respect to fluctuations and dissipative linear terms are included (second-order contributions). 
All remaining terms are considered to be higher-order contributions. 
Due to the fact that the conservation laws contain at least first-order space or time derivatives, zero-order terms with respect to the fluctuating quantities do not play a role further on. 
This classification and the so-called \emph{substitution corollary}, which allows to replace any quantity in a second-order or higher-order term by its first-order approximation, was introduced by~\textsc{Lighthill} in \MyReference{Lighthill1956} and described by~\textsc{Blackstock} in \MyReference{Blackstock1963}.

\paragraph{Linear wave equation.}
A natural approach for the derivation of a single higher-order partial differential equation is to combine the equations for conservation of mass and momentum. 
Subtracting the time derivative of~\eqref{eq:ConservationMass} from the divergence of~\eqref{eq:ConservationMomentum} and assuming interchangeability of space and time differentiation, the term $\partial_{t} \nabla \cdot (\varrho \, \vecv) = \nabla \cdot \partial_{t} (\varrho \, \vecv)$ cancels
\begin{equation*}
\nabla \cdot \Big(\vecv \, \nabla \cdot (\varrho \, \vecv) + \varrho \, (\vecv \cdot \nabla) \, \vecv\Big) + \Delta p - \partial_{tt} \varrho 
= \mu \, \Lambda \, \Delta (\nabla \cdot \vecv)\,;
\end{equation*}
here, we set $\Lambda = \tfrac{\mu_B}{\mu} + \tfrac{4}{3}$. 
Retaining only the first-order contribution $\Delta p_{\sim} - \partial_{tt} \varrho_{\sim}$ and replacing~\eqref{eq:EquationState} by the first-order approximation $\varrho_{\sim} \approx \frac{\varrho_0}{A} \, p_{\sim}$, where $A = c_0^2 \, \varrho_0$ and $c_0 = \sqrt{\frac{c_p \, p_0}{c_V \, \varrho_0}}$ denotes the speed of sound, yields a linear wave equation for the acoustic pressure 
\begin{equation*}
\partial_{tt} p_{\sim} - c_0^2 \, \Delta p_{\sim} = 0\,.
\end{equation*}

\paragraph{Nonlinear damped wave equation (Brunnhuber--Jordan--Kuznetsov equation).}
If additionally all second-order contributions are taken into account in~\eqref{eq:ConservationLaws} and~\eqref{eq:EquationState}, a more involved procedure for eliminating~$\varrho_{\sim}$, $p_{\sim}$, and~$T_{\sim}$ leads to a nonlinear damped wave equation for the acoustic velocity potential
\begin{subequations}
\label{eq:BrunnhuberJordanKuznetsov}
\begin{equation}
\begin{split}
&\partial_{ttt} \psi 
- \Big(a \, \big(1 + \tfrac{B}{A}\big) + \nu \Lambda\Big) \, \Delta \partial_{tt} \psi 
+ a \, \big(1 + \tfrac{B}{A}\big) \, \nu \Lambda \, \Delta^2 \partial_{t} \psi 
- c_0^2 \, \Delta \partial_{t} \psi \\
&\quad + a \, \big(1 + \tfrac{B}{A}\big) \, c_0^2 \, \Delta^2 \psi
+ \partial_{tt} \Big(\tfrac{1}{2 c_0^2} \, \tfrac{B}{A} \, \big(\partial_{t} \psi\big)^2 + \abs{\nabla \psi}^2\Big) = 0\,;
\end{split}
\end{equation}
details of the derivation are included in~\ref{SectionAppendix}.
As this equation coincides with Eq.\,(1.19) in \MyReference{Brunnhuber2015} and Eq.\,(4) in \MyReference{BrunnhuberJordan2016}, aside from the extension of the term~$a \, c_0^2 \, \Delta^2 \psi$ to $a \, (1 + \tfrac{B}{A}) \, c_0^2 \, \Delta^2 \psi$, we refer to it as Brunnhuber--Jordan--Kuznetsov equation.
We point out that the differential operator defining the linear contributions is given by the composition of a heat operator and a wave operator  
\begin{equation}
\begin{split}
&\Big(\partial_{t} - a \, \big(1 + \tfrac{B}{A}\big) \, \Delta\Big) \, \big(\partial_{tt} \psi - \nu \Lambda \, \Delta \partial_{t} \psi - c_0^2 \, \Delta \psi\big) \\
&\quad + \partial_{tt} \Big(\tfrac{1}{2 c_0^2} \, \tfrac{B}{A} \, \big(\partial_{t} \psi\big)^2 + \abs{\nabla \psi}^2\Big) = 0\,.
\end{split}
\end{equation}
\end{subequations}
see also Eq.\,(1) in \MyReference{Brunnhuber2015} and Eq.\,(1) in \MyReference{BrunnhuberJordan2016};
due to the fact that relation~\eqref{eq:ConservationEnergy} reflecting energy conservation involves the heat operator $\partial_{t} - a \, \Delta$, its appearance is quite intuitive. 
Our analysis, however, does not exploit the fact that the general model is factorisable and thus also applies to 
Eq.\,(1.19) in \MyReference{Brunnhuber2015} and Eq.\,(4) in \MyReference{BrunnhuberJordan2016}.
A significant discrepancy of~\eqref{eq:BrunnhuberJordanKuznetsov} compared to the model obtained by \textsc{Blackstock}, see Eq.\,(7) in \MyReference{Blackstock1963}, is the presence of the term comprising $\Delta^2\partial_{t} \psi$, which is essential for proving well-posedness, see \MyReference{Kaltenbacher2017}.

\paragraph{Limiting model (Kuznetsov equation).}
In situations where temperature constraints are insignificant, the Kuznetsov (K) equation 
\begin{equation}
\label{eq:Kuznetsov}
\partial_{tt} \psi - \nu \Lambda \, \Delta \partial_{t} \psi - c_0^2 \, \Delta \psi 
+ \partial_{t} \Big(\tfrac{1}{2 c_0^2} \, \tfrac{B}{A} \, \big(\partial_{t}\psi\big)^2 + \abs{\nabla \psi}^2\Big) = 0\,,
\end{equation}
see \MyReference{Kuznetsov1971}, results from~\eqref{eq:BrunnhuberJordanKuznetsov} by considering the formal limit $a = \tfrac{\nu}{\mathrm{Pr}} \to 0_{+}$ (but not necessarily $\nu \to 0_{+}$).  
More precisely, setting 
\begin{equation*}
F(\psi) = \partial_{tt} \psi - \nu \Lambda \, \Delta \partial_{t} \psi - c_0^2 \, \Delta \psi 
+ \partial_{t} \Big(\tfrac{1}{2 c_0^2} \, \tfrac{B}{A} \, \big(\partial_{t} \psi\big)^2 + \abs{\nabla \psi}^2\Big)\,,
\end{equation*}
it is evident that any solution to~\eqref{eq:Kuznetsov} satisfies $F(\psi) = 0$ and in particular fulfills $\partial_{t} F(\psi) = 0$, which corresponds to~\eqref{eq:BrunnhuberJordanKuznetsov} with $a = 0$;
on the other hand, integration of the condition $\partial_{t} F(\psi) = 0$ with respect to time implies that any solution to~\eqref{eq:BrunnhuberJordanKuznetsov} with $a = 0$ solves~\eqref{eq:Kuznetsov}, provided that the prescribed initial data satisfy a consistency condition such that $F(\psi(\cdot,0)) = 0$.
A rigorous justification of this limiting process is given in Section~\ref{SectionLimit}.
\subsection{Hierarchy of nonlinear damped wave equations}
\label{SectionModels2}
We next introduce the considered hierarchy of nonlinear damped wave equations, see also Table~\ref{Table1}; 
we distinguish equations of Kuznetsov and Westervelt type, respectively. 

\subsubsection*{Equations of Kuznetsov type.}
\begin{enumerate}[(1)]
\item 
For convenience, we restate the Brunnhuber--Jordan--Kuznetsov equation~\eqref{eq:BrunnhuberJordanKuznetsov} in elaborate and factorised form 
\begin{equation}
\tag{BJK}
\begin{split}
&\partial_{ttt} \psi 
- \Big(a \, \big(1 + \tfrac{B}{A}\big) + \nu \Lambda\Big) \, \Delta \partial_{tt} \psi 
+ a \, \big(1 + \tfrac{B}{A}\big) \, \nu \Lambda \, \Delta^2 \partial_{t} \psi 
- c_0^2 \, \Delta \partial_{t} \psi \\
&\quad + a \, \big(1 + \tfrac{B}{A}\big) \, c_0^2 \, \Delta^2 \psi
+ \partial_{tt} \Big(\tfrac{1}{2 c_0^2} \, \tfrac{B}{A} \, \big(\partial_{t} \psi\big)^2 + \abs{\nabla \psi}^2\Big) = 0\,, \\
&\Big(\partial_{t} - a \, \big(1 + \tfrac{B}{A}\big) \, \Delta\Big) \, \big(\partial_{tt} \psi - \nu \Lambda \, \Delta \partial_{t} \psi - c_0^2 \, \Delta \psi\big) \\
&\quad + \partial_{tt} \Big(\tfrac{1}{2 c_0^2} \, \tfrac{B}{A} \, \big(\partial_{t} \psi\big)^2 + \abs{\nabla \psi}^2\Big) = 0\,,
\end{split}
\end{equation}
see also Eq.\,(1.19) in \MyReference{Brunnhuber2015} and Eq.\,(4) in \MyReference{BrunnhuberJordan2016}.
\item 
In the special case of a monatomic gas, where the identity $\Lambda \, \Pr = 1$ holds, or, more generally, when $a \, (\Lambda \, \Pr - 1) \, \tfrac{B}{A} 
= (\nu \Lambda - a) \, \tfrac{B}{A}$ is negligible, i.e.~$\nu \Lambda \tfrac{B}{A} \approx a \, \tfrac{B}{A}$, the contribution involving~$\Delta^2 \partial_{t} \psi$ formally reduces to 
\begin{equation*}
a \, \big(1 + \tfrac{B}{A}\big) \, \nu \Lambda \, \Delta^2 \partial_{t} \psi 
\approx a \, \big(\nu \Lambda + a \, \tfrac{B}{A}\big) \, \Delta^2 \partial_{t} \psi\,;
\end{equation*}
if we replace in addition the term $a \, (1 + \tfrac{B}{A}) \, c_0^2 \, \Delta^2 \psi$ by $a \, c_0^2 \, \Delta^2 \psi$, we retain the factorisable reduced model 
\begin{equation}
\tag{BCK}
\begin{split}
&\partial_{ttt} \psi
- \Big(a \, \big(1 + \tfrac{B}{A}\big) + \nu \Lambda\Big) \, \Delta \partial_{tt} \psi 
+ a \, \big(\nu \Lambda + a \, \tfrac{B}{A}\big) \, \Delta^2 \partial_{t} \psi
- c_0^2 \, \Delta \partial_{t} \psi \\
&\quad + a \, c_0^2 \, \Delta^2 \psi
+ \partial_{tt} \Big(\tfrac{1}{2 c_0^2} \, \tfrac{B}{A} \, \big(\partial_{t} \psi\big)^2 + \abs{\nabla \psi}^2\Big) = 0\,, \\
&\big(\partial_{t} - a \, \Delta\big) \, \Big(\partial_{tt} \psi - \big(\nu \Lambda + a \, \tfrac{B}{A}\big) \, \Delta \partial_{t} \psi - c_0^2 \, \Delta \psi\Big) \\
&\quad + \partial_{tt} \Big(\tfrac{1}{2 c_0^2} \, \tfrac{B}{A} \, \big(\partial_{t} \psi\big)^2 + \abs{\nabla \psi}^2\Big) = 0\,, 
\end{split}
\end{equation}
which we refer to as Blackstock--Crighton--Kuznetsov equation, see also Eq.\,(1) in \MyReference{Brunnhuber2015} and Eq.\,(1) in  \MyReference{BrunnhuberJordan2016}.
\item 
As shown in Section~\ref{SectionLimit}, the Kuznetsov equation
\begin{equation}
\tag{K}
\partial_{tt} \psi - \nu \Lambda \, \Delta \partial_{t} \psi - c_0^2 \, \Delta \psi 
+ \partial_{t} \Big(\tfrac{1}{2 c_0^2} \, \tfrac{B}{A} \, \big(\partial_{t} \psi\big)^2 + \abs{\nabla \psi}^2\Big) = 0\,, 
\end{equation}
see also Eq.\,(3) in \MyReference{Brunnhuber2015} and \MyReference{Kuznetsov1971}, is obtained from (BJK) and (BCK) in the limit $a \to 0_{+}$; 
for this reduced model, the orders of the arising space and time derivatives are significantly lowered. 
\end{enumerate}

\subsubsection*{Equations of Westervelt type.}
\begin{enumerate}[(1)]
\item 
In certain situations, local nonlinear effects reflected by $\abs{\nabla \psi}^2 - \tfrac{1}{c_0^2} \, (\partial_{t} \psi)^2$ are negligible and thus the nonlinearity can be replaced by 
\begin{equation*}
\tfrac{1}{2 c_0^2} \, \tfrac{B}{A} \, \big(\partial_{t} \psi\big)^2 + \abs{\nabla \psi}^2 \approx \tfrac{1}{2 c_0^2} \, (2 + \tfrac{B}{A})\, \big(\partial_{t} \psi\big)^2\,;
\end{equation*}
in accordance with our derivation of the Brunnhuber--Jordan--Kuznetsov equation, we keep the term $a \, (1 + \tfrac{B}{A}) \, c_0^2 \, \Delta^2 \psi$. 
Altogether, this yields the nonlinear damped wave equation 
\begin{equation}
\tag{BJW}
\begin{split}
&\partial_{ttt} \psi 
- \Big(a \, \big(1 + \tfrac{B}{A}\big) + \nu \Lambda\Big) \, \Delta \partial_{tt} \psi 
+ a \, \big(1 + \tfrac{B}{A}\big) \, \nu \Lambda \, \Delta^2 \partial_{t} \psi 
- c_0^2 \, \Delta \partial_{t} \psi \\
&\quad + a \, \big(1 + \tfrac{B}{A}\big) \, c_0^2 \, \Delta^2 \psi
+ \tfrac{1}{2 c_0^2} \, \big(2 + \tfrac{B}{A}\big) \, \partial_{tt} \big(\partial_{t} \psi\big)^2 = 0\,,
\end{split}
\end{equation}
which we refer to as Brunnhuber--Jordan--Westervelt equation;
as in (BJK), the linear contributions are given by the composition of a wave and a heat operator. 
\item 
In analogy to (BCK), the Blackstock--Crighton--Westervelt equation 
\begin{equation}
\tag{BCW}
\begin{split}
&\partial_{ttt} \psi
- \Big(a \, \big(1 + \tfrac{B}{A}\big) + \nu \Lambda\Big) \, \Delta \partial_{tt} \psi 
+ a \, \big(\nu \Lambda + a \, \tfrac{B}{A}\big) \, \Delta^2 \partial_{t} \psi
- c_0^2 \, \Delta \partial_{t} \psi \\
&\quad + a \, c_0^2 \, \Delta^2 \psi
+ \tfrac{1}{2 c_0^2} \, \big(2 + \tfrac{B}{A}\big) \, \partial_{tt} \big(\partial_{t} \psi\big)^2 = 0
\end{split}
\end{equation}
is retained as a reduced model from (BJW), see also Eq.\,(2) in \MyReference{Brunnhuber2015}.
\item
The Westervelt equation is given by 
\begin{equation}
\tag{W}
\partial_{tt} \psi - \nu \Lambda \,  \Delta \partial_{t} \psi - c_0^2 \, \Delta \psi
+ \tfrac{1}{2 c_0^2} \, \big(2 + \tfrac{B}{A}\big) \, \partial_{t} \big(\partial_{t} \psi\big)^2 = 0\,,
\end{equation}
see also Eq.\,(4) in \MyReference{Brunnhuber2015} and \MyReference{Westervelt1963}; 
as justified in Section~\ref{SectionLimit}, it results as limiting model from (BJK) for vanishing thermal conductivity and negligible local nonlinear effects.
\end{enumerate}
\section{Auxiliary results}
\label{SectionAuxiliaries}
In this section, we state unifying representations of the nonlinear damped wave equations studied in this work. 
Furthermore, we deduce reformulations of the Brunnhuber--Jordan--Kuznetsov equation and a priori energy estimates that are needed in Section~\ref{SectionLimit}.
\subsection{Unifying representations}
\paragraph{Abbreviations.}
In view of a unifying representation, it is convenient to introduce switching variables $\sigma_0, \sigma \in \{0,1\}$ and abbreviations for the arising non-negative coefficients 
\begin{subequations}
\label{eq:GeneralModel}
\begin{equation}
\label{eq:Coefficients}
\begin{gathered}
\bba_1 = a \, \big(1 + \tfrac{B}{A}\big) + \nu \Lambda > 0\,, \\
\bba_2(\sigma_0) = a \, \big(\nu \Lambda + a \, \tfrac{B}{A} + \sigma_0 \, \tfrac{B}{A} \, (\nu \Lambda - a)\big) > 0\,, \\
\bb_3 = c_0^2 > 0\,, \quad 
\bba_4(\sigma_0) = a \, \big(1 + \sigma_0 \, \tfrac{B}{A}\big) \, c_0^2 > 0\,, \\
\bb_5(\sigma) = \tfrac{1}{c_0^2} \, \big(2 \, (1 - \sigma) + \tfrac{B}{A}\big) > 0\,, \quad 
\bb_6(\sigma) = \sigma \geq 0\,;
\end{gathered}
\end{equation}
recall that the quantities $a, \frac{B}{A}, \nu \Lambda, c_0^2 > 0$ are strictly positive. 
In addition, we set 
\begin{equation}
\bba_0(\sigma_0) = \frac{\bba_2(\sigma_0)}{\bba_4(\sigma_0)} = \tfrac{1}{c_0^2} \, \big(\nu \Lambda + (1 - \sigma_0) \, a \, \tfrac{B}{A}\big) > 0\,.
\end{equation}
Evidently, these definitions imply the relations 
\begin{equation}
\begin{gathered}
\bba_0(1) = \tfrac{1}{c_0^2} \, \nu \Lambda\,, \quad \bba_0(0) = \tfrac{1}{c_0^2} \, \big(\nu \Lambda + a \, \tfrac{B}{A}\big)\,, \\
\bba_2(1) = a \, \big(1 + \tfrac{B}{A}\big) \, \nu \Lambda\,, \quad 
\bba_2(0) = a \, \big(\nu \Lambda + a \, \tfrac{B}{A}\big)\,, \\
\bba_4(1) = a \, \big(1 + \tfrac{B}{A}\big) \, c_0^2\,, \quad \bba_4(0) = a \, c_0^2\,, \\
\bb_5(1) = \tfrac{1}{c_0^2} \, \tfrac{B}{A}\,, \quad 
\bb_5(0) = \tfrac{1}{c_0^2} \, \big(2 + \tfrac{B}{A}\big)\,, \\
\bb_6(1) = 1\,, \quad \bb_6(0) = 0\,;
\end{gathered}
\end{equation}
in the limit $a \to 0_{+}$, the following values are obtained 
\begin{equation}
\bbzero_0(\sigma_0) = \tfrac{1}{c_0^2} \, \nu \Lambda\,, \quad 
\bbzero_1 = \nu \Lambda\,, \quad 
\bbzero_2(\sigma_0) = 0\,, \quad 
\bbzero_4(\sigma_0) = 0\,.
\end{equation}

\paragraph{Unifying representations.}
Employing a compact formulation as abstract evolution equation, the Brunnhuber--Jordan--Kuznetsov equation takes the following form with $\sigma_0 = \sigma = 1$ 
\begin{equation}
\begin{split}
&\partial_{ttt} \psi(t) 
- \bba_1 \, \Delta \partial_{tt} \psi(t) 
+ \bba_2(\sigma_0) \, \Delta^2 \partial_{t} \psi(t) 
- \bb_3 \, \Delta \partial_{t} \psi(t) \\
&\quad + \bba_4(\sigma_0) \, \Delta^2 \psi(t) 
+ \partial_{tt} \Big(\tfrac{1}{2} \, \bb_5(\sigma) \, \big(\partial_{t} \psi(t)\big)^2 + \bb_6(\sigma) \, \abs{\nabla \psi(t)}^2\Big) = 0\,, 
\end{split}
\end{equation}
\end{subequations}
see~(BJK); 
the equations~(BCK), (BJW), and (BCW) are included as special cases, see Table~\ref{Table1}. 
Moreover, the Kuznetsov and Westervelt equations rewrite as 
\begin{equation}
\label{eq:GeneralModelLimit}
\begin{split}
&\partial_{tt} \psi(t) - \bbzero_1 \, \Delta \partial_{t} \psi(t) - \bb_3 \, \Delta \psi(t) \\
&\quad + \partial_{t} \Big(\tfrac{1}{2} \, \bb_5(\sigma) \, \big(\partial_{t} \psi(t)\big)^2 + \bb_6(\sigma) \, \abs{\nabla \psi(t)}^2\Big) = 0\,,
\end{split}
\end{equation}
when setting $\sigma = 1$ or $\sigma = 0$, respectively. 
\subsection{Reformulations}
\paragraph{Initial and boundary conditions.}
We henceforth consider~\eqref{eq:GeneralModel} on~$[0, T]$, imposing homogeneous Dirichlet conditions on certain space and time derivatives of the solution
\begin{subequations}
\label{eq:AssumptionDirichlet}
\begin{gather}
\label{eq:AssumptionDirichletLower}
\partial_{tt} \psi(t)\big\vert_{\partial \Omega} = 0\,, \quad \Delta \partial_{t} \psi(t)\big\vert_{\partial \Omega} =0\,, \quad 
\Delta \psi(t)\big\vert_{\partial \Omega} =0\,, \\
\label{eq:AssumptionDirichletHigher}
\partial_{ttt} \psi(t)\big\vert_{\partial \Omega} = 0\,, \quad \Delta \partial_{tt} \psi(t)\big\vert_{\partial \Omega} =0\,.
\end{gather}
\end{subequations}
Moreover, we suppose that the initial conditions 
\begin{equation}
\label{eq:InitialCondition}
\psi(0) = \psi_0\,, \quad \partial_{t} \psi(0) = \psi_1\,, \quad \partial_{tt} \psi(0) = \psi_2\,, 
\end{equation}
are fulfilled; 
the needed regularity, compatibility, and smallness requirements on $\psi_0$, $\psi_1$, and $\psi_2$ are specified in Proposition~\ref{thm:Proposition}. 

\paragraph{Reformulation by integration.}
With regard to~\eqref{eq:GeneralModelLimit}, assuming interchangeability of space and time differentiation, we set 
\begin{subequations}
\label{eq:GeneralModelReformulationIntegrationF}
\begin{equation}
\begin{split}
F\big(\psi(t)\big) 
&= \partial_{tt} \psi(t) - \bbzero_1 \, \Delta \partial_{t} \psi(t) - \bb_3 \, \Delta \psi(t) \\
&\qquad + \bb_5(\sigma) \, \partial_{tt} \psi(t) \, \partial_{t} \psi(t) + 2 \, \bb_6(\sigma) \, \nabla \partial_{t} \psi(t) \cdot \nabla \psi(t)\,;
\end{split}
\end{equation}
straightforward differentiation shows that its time derivative is given by 
\begin{equation*}
\begin{split}
\partial_{t} F\big(\psi(t)\big) 
&= \partial_{ttt} \psi(t) - \bbzero_1 \, \Delta \partial_{tt} \psi(t) - \bb_3 \, \Delta \partial_{t} \psi(t) \\
&\qquad + \bb_5(\sigma) \, \partial_{ttt} \psi(t) \, \partial_{t} \psi(t) 
+ \bb_5(\sigma) \, \big(\partial_{tt} \psi(t)\big)^2 \\
&\qquad + 2 \, \bb_6(\sigma) \, \nabla \partial_{tt} \psi(t) \cdot \nabla \psi(t)
+ 2 \, \bb_6(\sigma) \, \absbig{\nabla \partial_{t} \psi(t)}^2 
\end{split}
\end{equation*}
and that~\eqref{eq:GeneralModel} rewrites as 
\begin{equation*}
\begin{split}
\partial_{t} F\big(\psi(t)\big) 
&= \big(\bba_1 - \bbzero_1\big)\, \Delta \partial_{tt} \psi(t) 
- \bba_2(\sigma_0) \, \Delta^2 \partial_{t} \psi(t) - \bba_4(\sigma_0) \, \Delta^2 \psi(t)\,.
\end{split}
\end{equation*}
Provided that the prescribed initial data are sufficiently regular and satisfy the consistency condition 
\begin{equation}
\psi_2 - \bbzero_1 \, \Delta \psi_1 - \bb_3 \, \Delta \psi_0 + \bb_5(\sigma) \, \psi_2 \, \psi_1 + 2 \, \bb_6(\sigma) \, \nabla \psi_1 \cdot \nabla \psi_0 = 0
\end{equation}
such that $F(\psi(0)) = 0$, integration with respect to time implies
\begin{equation}
\begin{split}
F\big(\psi(t)\big) 
&= \big(\bba_1 - \bbzero_1\big)\, \big(\Delta \partial_{t} \psi(t) - \Delta \psi_1\big) \\
&\qquad - \bba_2(\sigma_0) \, \big(\Delta^2 \psi(t) - \Delta^2 \psi_0\big)
- \bba_4(\sigma_0) \int_{0}^{t} \Delta^2 \psi(\tau) \; \dd\tau\,.
\end{split}
\end{equation}
\end{subequations}

\paragraph{Reformulation by differentiation.}
A reformulation of~\eqref{eq:GeneralModel} is obtained by straightforward differentiation of the nonlinear term;  
supressing for the sake of notational simplicity the dependence on~$\psi$ and $\sigma \in \{0,1\}$, we set 
\begin{subequations}
\label{eq:GeneralModelReformulationDifferentiation}
\begin{equation}
\begin{split}
\alpha(t) &= 1 + \bb_5(\sigma) \, \partial_{t} \psi(t)\,, \\
r(t) &= \bb_5(\sigma) \, \big(\partial_{tt} \psi(t)\big)^2 + \bb_6(\sigma) \, \partial_{tt} \abs{\nabla \psi(t)}^2 \\
&= \bb_5(\sigma) \, \big(\partial_{tt} \psi(t)\big)^2 + 2 \, \bb_6(\sigma) \, \partial_{t} \big(\nabla \partial_{t} \psi(t) \cdot \nabla \psi(t)\big) \\
&= \bb_5(\sigma) \, \big(\partial_{tt} \psi(t)\big)^2 + 2 \, \bb_6(\sigma) \, \nabla \partial_{tt} \psi(t) \cdot \nabla \psi(t)
+ 2 \, \bb_6(\sigma) \, \absbig{\nabla \partial_{t} \psi(t)}^2\,, 
\end{split} 
\end{equation}
and, as a consequence, we obtain the relation 
\begin{equation}
\label{eq:GeneralModelReformulationDifferentiation1}
\begin{split}
&\alpha(t) \, \partial_{ttt} \psi(t) 
- \bba_1 \, \Delta \partial_{tt} \psi(t) 
+ \bba_2(\sigma_0) \, \Delta^2 \partial_{t} \psi(t) 
- \bb_3 \, \Delta \partial_{t} \psi(t) \\
&\quad + \bba_4(\sigma_0) \, \Delta^2 \psi(t) 
+ r(t) 
= 0\,;
\end{split}
\end{equation}
provided that non-degeneracy of $\alpha(t)$ is ensured, this further yields 
\begin{equation}
\label{eq:GeneralModelReformulationDifferentiation2}
\begin{split}
&\partial_{ttt} \psi(t) 
- \bba_1 \, \tfrac{1}{\alpha(t)} \, \Delta \partial_{tt} \psi(t) 
+ \bba_2(\sigma_0) \, \tfrac{1}{\alpha(t)} \, \Delta^2 \partial_{t} \psi(t) 
- \bb_3 \, \tfrac{1}{\alpha(t)} \, \Delta \partial_{t} \psi(t) \\
&\quad + \bba_4(\sigma_0) \, \tfrac{1}{\alpha(t)} \, \Delta^2 \psi(t) 
+ \tfrac{1}{\alpha(t)} \, r(t) 
= 0\,.
\end{split}
\end{equation}
\end{subequations}

\paragraph{Fixed-point argument.}
Our approach for the derivation of a priori energy estimates uses a fixed-point argument based on a suitable modification 
of~\eqref{eq:GeneralModelReformulationDifferentiation}; 
that is, we consider two functions~$\phi$ and~$\psi$ that satisfy the initial conditions 
\begin{equation}
\label{eq:InitialConditionphipsi}
\phi(0) = \psi(0) = \psi_0\,, \quad \partial_{t} \phi(0) = \partial_{t} \psi(0) = \psi_1\,, \quad \partial_{tt} \phi(0) = \partial_{tt} \psi(0) = \psi_2\,, 
\end{equation}
and replace~$\alpha$ and~$r$ in relations~\eqref{eq:GeneralModelReformulationDifferentiation1} and~\eqref{eq:GeneralModelReformulationDifferentiation2} by 
\begin{equation}
\label{eq:alpha_r_phi}
\begin{split}
\alphaphi(t) &= 1 + \bb_5(\sigma) \, \partial_{t} \phi(t)\,, \\
\rphi(t) 
&= \bb_5(\sigma) \, \partial_{tt} \psi(t) \, \partial_{tt} \phi(t) 
+ 2 \, \bb_6(\sigma) \, \nabla \partial_{tt} \psi(t) \cdot \nabla \phi(t) \\
&\qquad + 2 \, \bb_6(\sigma) \, \nabla \partial_{t} \psi(t) \cdot \nabla \partial_{t} \phi(t)\,.
\end{split} 
\end{equation}

\paragraph{Reformulation by testing.}
Our starting point is~\eqref{eq:GeneralModelReformulationDifferentiation1} with~$\alpha$ and~$r$ substituted by~$\alphaphi$ and~$\rphi$;
testing with~$\partial_{tt} \psi(t)$ yields 
\begin{equation*}
\begin{split}
&\scpbig{\Ltwo}{\alphaphi(t) \, \partial_{ttt} \psi(t)}{\partial_{tt} \psi(t)} 
- \bba_1 \, \scpbig{\Ltwo}{\Delta \partial_{tt} \psi(t)}{\partial_{tt} \psi(t)} \\
&\quad + \bba_2(\sigma_0) \, \scpbig{\Ltwo}{\Delta^2 \partial_{t} \psi(t)}{\partial_{tt} \psi(t)} 
- \bb_3 \, \scpbig{\Ltwo}{\Delta \partial_{t} \psi(t)}{\partial_{tt} \psi(t)} \\
&\quad + \bba_4(\sigma_0) \, \scpbig{\Ltwo}{\Delta^2 \psi(t)}{\partial_{tt} \psi(t)} 
+ \scpbig{\Ltwo}{\rphi(t)}{\partial_{tt} \psi(t)} 
= 0\,. 
\end{split}
\end{equation*}
In order to rewrite this relation as the time derivative of a function plus additional terms, we apply the identity 
\begin{equation*}
\begin{split}
\scpbig{\Ltwo}{\alphaphi(t) \, \partial_{ttt} \psi(t)}{\partial_{tt} \psi(t)}
&= \tfrac{1}{2} \, \partial_{t} \normBig{\Ltwo}{\sqrt{\alphaphi(t)} \, \partial_{tt} \psi(t)}^2 \\
&\qquad - \tfrac{1}{2} \, \scpbig{\Ltwo}{\partial_{t} \alphaphi(t) \, \partial_{tt} \psi(t)}{\partial_{tt} \psi(t)}\,;
\end{split}
\end{equation*}
under assumption~\eqref{eq:AssumptionDirichletLower}, integration-by-parts implies 
\begin{equation*}
\begin{split}
&\scpbig{\Ltwo}{\Delta \partial_{tt} \psi(t)}{\partial_{tt} \psi(t)} = - \, \normbig{\Ltwo}{\nabla \partial_{tt} \psi(t)}^2\,, \\
&\scpbig{\Ltwo}{\Delta^2 \partial_{t} \psi(t)}{\partial_{tt} \psi(t)} = \scpbig{\Ltwo}{\Delta \partial_{t} \psi(t)}{\Delta \partial_{tt} \psi(t)}
= \tfrac{1}{2} \, \partial_{t} \normbig{\Ltwo}{\Delta \partial_{t} \psi(t)}^2\,, \\
&\scpbig{\Ltwo}{\Delta \partial_{t} \psi(t)}{\partial_{tt} \psi(t)} = - \, \scpbig{\Ltwo}{\nabla \partial_{t} \psi(t)}{\nabla \partial_{tt} \psi(t)} 
= - \, \tfrac{1}{2} \, \partial_{t} \normbig{\Ltwo}{\nabla \partial_{t} \psi(t)}^2 \\
&\scpbig{\Ltwo}{\Delta^2 \psi(t)}{\partial_{tt} \psi(t)} = \scpbig{\Ltwo}{\Delta \psi(t)}{\Delta \partial_{tt} \psi(t)} \\
&\quad = \partial_{t} \scpbig{\Ltwo}{\Delta \partial_{t} \psi(t)}{\Delta \psi(t)} - \normbig{\Ltwo}{\Delta \partial_{t} \psi(t)}^2\,. 
\end{split}
\end{equation*} 
As a consequence, we have 
\begin{equation*}
\begin{split}
&\tfrac{1}{2} \, \partial_{t} \normBig{\Ltwo}{\sqrt{\alphaphi(t)} \, \partial_{tt} \psi(t)}^2
+ \bba_1 \, \normbig{\Ltwo}{\nabla \partial_{tt} \psi(t)}^2 
+ \tfrac{\bba_2(\sigma_0)}{2} \, \partial_{t} \normbig{\Ltwo}{\Delta \partial_{t} \psi(t)}^2 \\
&\quad + \tfrac{\bb_3}{2} \, \partial_{t} \normbig{\Ltwo}{\nabla \partial_{t} \psi(t)}^2 
+ \bba_4(\sigma_0) \, \partial_{t} \scpbig{\Ltwo}{\Delta \partial_{t} \psi(t)}{\Delta \psi(t)} 
- \bba_4(\sigma_0) \, \normbig{\Ltwo}{\Delta \partial_{t} \psi(t)}^2 \\
&\quad + \scpbig{\Ltwo}{\rphi(t) - \tfrac{1}{2} \, \partial_t \alphaphi(t) \, \partial_{tt} \psi(t)}{\partial_{tt} \psi(t)}
= 0\,; 
\end{split}
\end{equation*}
by means of the abbreviation 
\begin{subequations}
\label{eq:GeneralModelReformulationTest1}
\begin{equation}
\begin{split}
\MywEphipsi{0}{t} &= \tfrac{1}{2} \, \normBig{\Ltwo}{\sqrt{\alphaphi(t)} \, \partial_{tt} \psi(t)}^2
+ \tfrac{\bba_2(\sigma_0)}{2} \, \normbig{\Ltwo}{\Delta \partial_{t} \psi(t)}^2 \\
&\qquad + \tfrac{\bb_3}{2} \, \normbig{\Ltwo}{\nabla \partial_{t} \psi(t)}^2\,,
\end{split}
\end{equation}
the following relation results 
\begin{equation*}
\begin{split}
&\partial_{t} \MywEphipsi{0}{t} 
+ \bba_1 \, \normbig{\Ltwo}{\nabla \partial_{tt} \psi(t)}^2 \\
&\quad = - \, \bba_4(\sigma_0) \, \partial_{t} \scpbig{\Ltwo}{\Delta \partial_{t} \psi(t)}{\Delta \psi(t)} 
+ \bba_4(\sigma_0) \, \normbig{\Ltwo}{\Delta \partial_{t} \psi(t)}^2 \\
&\quad\qquad 
- \scpbig{\Ltwo}{\rphi(t) - \tfrac{1}{2} \, \partial_t \alphaphi(t) \, \partial_{tt} \psi(t)}{\partial_{tt} \psi(t)}\,.
\end{split}
\end{equation*}
Integration with respect to time finally yields 
\begin{equation}
\begin{split}
&\MywEphipsi{0}{t} + \bba_1 \int_{0}^{t} \normbig{\Ltwo}{\nabla \partial_{tt} \psi(\tau)}^2 \; \dd\tau \\
&\quad = \MywEphipsiAtzero{0}
+ \bba_4(\sigma_0) \, \scpbig{\Ltwo}{\Delta \psi_1}{\Delta \psi_0}
- \bba_4(\sigma_0) \, \scpbig{\Ltwo}{\Delta \partial_{t} \psi(t)}{\Delta \psi(t)} \\
&\quad\qquad + \bba_4(\sigma_0) \int_{0}^{t} \normbig{\Ltwo}{\Delta \partial_{t} \psi(\tau)}^2 \; \dd\tau \\
&\quad\qquad - \int_{0}^{t} \scpbig{\Ltwo}{\rphi(\tau) - \tfrac{1}{2} \, \partial_t \alphaphi(\tau) \, \partial_{tt} \psi(\tau)}{\partial_{tt} \psi(\tau)} \; \dd\tau\,;
\end{split}
\end{equation}
note that we here set 
\begin{equation}
\MywEphipsiAtzero{0} = \tfrac{1}{2} \, \normBig{\Ltwo}{\sqrt{1 + \bb_5(\sigma) \, \psi_1} \, \psi_2}^2
+ \tfrac{\bba_2(\sigma_0)}{2} \, \normbig{\Ltwo}{\Delta \psi_1}^2 
+ \tfrac{\bb_3}{2} \, \normbig{\Ltwo}{\nabla \psi_1}^2\,.
\end{equation}
\end{subequations}

\paragraph{Further reformulation by testing.}
On the other hand, we substitute~$\alpha$ and~$r$ in~\eqref{eq:GeneralModelReformulationDifferentiation2} by~$\alphaphi$ and~$\rphi$; 
by testing with~$\Delta \partial_{tt} \psi(t)$, we obtain 
\begin{equation*}
\begin{split}
&\scpbig{\Ltwo}{\partial_{ttt} \psi(t)}{\Delta \partial_{tt} \psi(t)} 
- \bba_1 \, \normBig{\Ltwo}{\tfrac{1}{\sqrt{\alphaphi(t)}} \, \Delta \partial_{tt} \psi(t)}^2 \\
&\quad + \bba_2(\sigma_0) \, \scpbig{\Ltwo}{\tfrac{1}{\alphaphi(t)} \, \Delta^2 \partial_{t} \psi(t)}{\Delta \partial_{tt} \psi(t)} 
- \bb_3 \, \scpBig{\Ltwo}{\tfrac{1}{\alphaphi(t)} \, \Delta \partial_{t} \psi(t)}{\Delta \partial_{tt} \psi(t)} \\
&\quad + \bba_4(\sigma_0) \, \scpBig{\Ltwo}{\tfrac{1}{\alphaphi(t)} \, \Delta^2 \psi(t)}{\Delta \partial_{tt} \psi(t)}
+ \scpBig{\Ltwo}{\tfrac{1}{\alphaphi(t)} \, \rphi(t)}{\Delta \partial_{tt} \psi(t)} 
= 0\,.
\end{split}
\end{equation*}
Similarly to before, we employ integration-by-parts under assumption~\eqref{eq:AssumptionDirichletHigher} and replace the arising space and time derivatives of~$\frac{1}{\alphaphi}$ by
\begin{equation*}
\nabla \tfrac{1}{\alphaphi(t)} = - \, \bb_5(\sigma) \, \tfrac{1}{(\alphaphi(t))^2} \, \nabla \partial_{t} \phi(t)\,, \quad 
\partial_{t} \tfrac{1}{\alphaphi(t)} = - \, \bb_5(\sigma) \, \tfrac{1}{(\alphaphi(t))^2} \, \partial_{tt} \phi(t)\,;
\end{equation*}
this yields the identities
\begin{equation*}
\begin{split}
&\scpbig{\Ltwo}{\partial_{ttt} \psi(t)}{\Delta \partial_{tt} \psi(t)} = - \, \scpbig{\Ltwo}{\nabla \partial_{ttt} \psi(t)}{\nabla \partial_{tt} \psi(t)} 
= - \, \tfrac{1}{2} \, \partial_{t} \normbig{\Ltwo}{\nabla \partial_{tt} \psi(t)}^2\,, \\
&\scpBig{\Ltwo}{\tfrac{1}{\alphaphi(t)} \, \Delta^2 \partial_{t} \psi(t)}{\Delta \partial_{tt} \psi(t)}
= - \, \scpBig{\Ltwo}{\nabla \Delta \partial_{t} \psi(t)}{\nabla \big(\tfrac{1}{\alphaphi(t)} \, \Delta \partial_{tt} \psi(t)\big)} \\
&\quad = - \, \scpBig{\Ltwo}{\nabla \Delta \partial_{t} \psi(t)}{\nabla \tfrac{1}{\alphaphi(t)} \, \Delta \partial_{tt} \psi(t)}
- \scpBig{\Ltwo}{\nabla \Delta \partial_{t} \psi(t)}{\tfrac{1}{\alphaphi(t)} \, \nabla \Delta \partial_{tt} \psi(t)} \\
&\quad = - \, \scpBig{\Ltwo}{\nabla \tfrac{1}{\alphaphi(t)}}{\Delta \partial_{tt} \psi(t) \, \nabla \Delta \partial_{t} \psi(t)} \\
&\quad\qquad - \tfrac{1}{2} \, \partial_{t} \normBig{\Ltwo}{\tfrac{1}{\sqrt{\alphaphi(t)}} \, \nabla \Delta \partial_{t} \psi(t)}^2 
+ \tfrac{1}{2} \, \scpBig{\Ltwo}{\partial_{t} \tfrac{1}{\alphaphi(t)}}{\absbig{\nabla \Delta \partial_{t} \psi(t)}^2} \\
&\quad = - \, \tfrac{1}{2} \, \partial_{t} \normBig{\Ltwo}{\tfrac{1}{\sqrt{\alphaphi(t)}} \, \nabla \Delta \partial_{t} \psi(t)}^2 \\
&\quad\qquad 
+ \bb_5(\sigma) \, \scpBig{\Ltwo}{\tfrac{1}{(\alphaphi(t))^2}}{\Delta \partial_{tt} \psi(t) \, \nabla \Delta \partial_{t} \psi(t) \cdot \nabla \partial_{t} \phi(t)} \\
&\quad\qquad - \, \tfrac{\bb_5(\sigma)}{2} \, \scpBig{\Ltwo}{\tfrac{1}{(\alphaphi(t))^2}}{\partial_{tt} \phi(t) \, \absbig{\nabla \Delta \partial_{t} \psi(t)}^2}\,, \\
&\scpBig{\Ltwo}{\tfrac{1}{\alphaphi(t)} \, \Delta \partial_{t} \psi(t)}{\Delta \partial_{tt} \psi(t)} \\
&\quad = \tfrac{1}{2} \, \partial_{t} \normBig{\Ltwo}{\tfrac{1}{\sqrt{\alphaphi(t)}} \, \Delta \partial_{t} \psi(t)}^2 
- \tfrac{1}{2} \, \scpBig{\Ltwo}{\partial_{t} \tfrac{1}{\alphaphi(t)}}{\big(\Delta \partial_{t} \psi(t)\big)^2} \\
&\quad = \tfrac{1}{2} \, \partial_{t} \normBig{\Ltwo}{\tfrac{1}{\sqrt{\alphaphi(t)}} \, \Delta \partial_{t} \psi(t)}^2 
+ \tfrac{\bb_5(\sigma)}{2} \, \scpBig{\Ltwo}{\tfrac{1}{(\alphaphi(t))^2}}{\partial_{tt} \phi(t) \, \big(\Delta \partial_{t} \psi(t)\big)^2}\,;
\end{split}
\end{equation*}
furthermore, we make use of the relation
\begin{equation*}
\begin{split}
&\scpBig{\Ltwo}{\tfrac{1}{\alphaphi(t)} \, \Delta^2 \psi(t)}{\Delta \partial_{tt} \psi(t)}
= - \, \scpBig{\Ltwo}{\nabla \Delta \psi(t)}{\nabla \big(\tfrac{1}{\alphaphi(t)} \, \Delta \partial_{tt} \psi(t)\big)} \\
&\quad = - \, \scpBig{\Ltwo}{\nabla \Delta \psi(t)}{\nabla \tfrac{1}{\alphaphi(t)} \, \Delta \partial_{tt} \psi(t)}
- \scpBig{\Ltwo}{\nabla \Delta \psi(t)}{\tfrac{1}{\alphaphi(t)} \, \nabla \Delta \partial_{tt} \psi(t)} \\
&\quad= - \, \scpBig{\Ltwo}{\nabla \tfrac{1}{\alphaphi(t)}}{\Delta \partial_{tt} \psi(t) \, \nabla \Delta \psi(t)} 
- \partial_{t} \scpBig{\Ltwo}{\tfrac{1}{\alphaphi(t)}}{\nabla \Delta \partial_{t} \psi(t) \cdot \nabla \Delta \psi(t)} \\
&\quad\qquad + \scpBig{\Ltwo}{\partial_{t} \tfrac{1}{\alphaphi(t)}}{\nabla \Delta \partial_{t} \psi(t) \cdot \nabla \Delta \psi(t)}
+ \normBig{\Ltwo}{\tfrac{1}{\sqrt{\alphaphi(t)}} \, \nabla \Delta \partial_{t} \psi(t)}^2 \\
&\quad= - \, \partial_{t} \scpBig{\Ltwo}{\tfrac{1}{\alphaphi(t)}}{\nabla \Delta \partial_{t} \psi(t) \cdot \nabla \Delta \psi(t)} \\
&\quad\qquad + \bb_5(\sigma) \, \scpBig{\Ltwo}{\tfrac{1}{(\alphaphi(t))^2}}{\Delta \partial_{tt} \psi(t) \, \nabla \partial_{t} \phi(t) \cdot \nabla \Delta \psi(t)} \\
&\quad\qquad - \, \bb_5(\sigma) \, \scpBig{\Ltwo}{\tfrac{1}{(\alphaphi(t))^2}}{\partial_{tt} \phi(t) \, \nabla \Delta \partial_{t} \psi(t) \cdot \nabla \Delta \psi(t)} \\
&\quad\qquad + \normBig{\Ltwo}{\tfrac{1}{\sqrt{\alphaphi(t)}} \, \nabla \Delta \partial_{t} \psi(t)}^2\,.
\end{split}
\end{equation*}
With the help of the abbreviation 
\begin{subequations}
\label{eq:GeneralModelReformulationTest2}
\begin{equation}
\begin{split}
\MywEphipsi{1}{t} &= \tfrac{1}{2} \, \normbig{\Ltwo}{\nabla \partial_{tt} \psi(t)}^2
+ \tfrac{\bba_2(\sigma_0)}{2} \, \normBig{\Ltwo}{\tfrac{1}{\sqrt{\alphaphi(t)}} \, \nabla \Delta \partial_{t} \psi(t)}^2 \\
&\qquad + \tfrac{\bb_3}{2} \, \normBig{\Ltwo}{\tfrac{1}{\sqrt{\alphaphi(t)}} \, \Delta \partial_{t} \psi(t)}^2\,, 
\end{split}
\end{equation}
we thus obtain 
\begin{equation*}
\begin{split}
&\partial_{t} \MywEphipsi{1}{t} 
+ \bba_1 \, \normBig{\Ltwo}{\tfrac{1}{\sqrt{\alphaphi(t)}} \, \Delta \partial_{tt} \psi(t)}^2 \\
&\quad = - \, \bba_4(\sigma_0) \, \partial_{t} \scpBig{\Ltwo}{\tfrac{1}{\alphaphi(t)}}{\nabla \Delta \partial_{t} \psi(t) \cdot \nabla \Delta \psi(t)} \\
&\quad\qquad + \bba_4(\sigma_0) \, \normBig{\Ltwo}{\tfrac{1}{\sqrt{\alphaphi(t)}} \, \nabla \Delta \partial_{t} \psi(t)}^2 \\
&\quad\qquad + \scpBig{\Ltwo}{\tfrac{1}{\alphaphi(t)} \, \rphi(t)}{\Delta \partial_{tt} \psi(t)} \\
&\quad\qquad + \bba_2(\sigma_0) \, \bb_5(\sigma) \, \scpBig{\Ltwo}{\tfrac{1}{(\alphaphi(t))^2}}{\Delta \partial_{tt} \psi(t) \, \nabla \Delta \partial_{t} \psi(t) \cdot \nabla \partial_{t} \phi(t)} \\
&\quad\qquad - \tfrac{\bba_2(\sigma_0) \, \bb_5(\sigma)}{2} \, \scpBig{\Ltwo}{\tfrac{1}{(\alphaphi(t))^2}}{\partial_{tt} \phi(t) \, \absbig{\nabla \Delta \partial_{t} \psi(t)}^2} \\
&\quad\qquad - \tfrac{\bb_3 \, \bb_5(\sigma)}{2} \, \scpBig{\Ltwo}{\tfrac{1}{(\alphaphi(t))^2}}{\partial_{tt} \phi(t) \, \big(\Delta \partial_{t} \psi(t)\big)^2} \\
&\quad\qquad + \bba_4(\sigma_0) \, \bb_5(\sigma) \, \scpBig{\Ltwo}{\tfrac{1}{(\alphaphi(t))^2}}{\Delta \partial_{tt} \psi(t) \, \nabla \partial_{t} \phi(t) \cdot \nabla \Delta \psi(t)} \\
&\quad\qquad - \bba_4(\sigma_0) \, \bb_5(\sigma) \, \scpBig{\Ltwo}{\tfrac{1}{(\alphaphi(t))^2}}{\partial_{tt} \phi(t) \, \nabla \Delta \partial_{t} \psi(t) \cdot \nabla \Delta \psi(t)}\,. 
\end{split}
\end{equation*}
Performing integration with respect to time, finally leads to 
\begin{equation}
\begin{split}
&\MywEphipsi{1}{t} + \bba_1 \int_{0}^{t} \normBig{\Ltwo}{\tfrac{1}{\sqrt{\alphaphi(\tau)}} \, \Delta \partial_{tt} \psi(\tau)}^2 \; \dd\tau \\
&\quad = \MywEphipsiAtzero{1} 
+ \bba_4(\sigma_0) \, \scpBig{\Ltwo}{\tfrac{1}{\alphaphi(0)}}{\nabla \Delta \psi_1 \cdot \nabla \Delta \psi_0} \\
&\qquad - \, \bba_4(\sigma_0) \, \scpBig{\Ltwo}{\tfrac{1}{\alphaphi(t)}}{\nabla \Delta \partial_{t} \psi(t) \cdot \nabla \Delta \psi(t)} \\
&\qquad + \bba_4(\sigma_0) \int_{0}^{t} \normBig{\Ltwo}{\tfrac{1}{\sqrt{\alphaphi(\tau)}} \, \nabla \Delta \partial_{t} \psi(\tau)}^2 \; \dd\tau \\
&\qquad + \int_{0}^{t} \scpBig{\Ltwo}{\tfrac{1}{\alphaphi(\tau)} \, \rphi(\tau)}{\Delta \partial_{tt} \psi(\tau)} \; \dd\tau \\
&\qquad + \bba_2(\sigma_0) \, \bb_5(\sigma) 
\int_{0}^{t} \scpBig{\Ltwo}{\tfrac{1}{(\alphaphi(\tau))^2}}{\Delta \partial_{tt} \psi(\tau) \, \nabla \Delta \partial_{t} \psi(\tau) \cdot \nabla \partial_{t} \phi(\tau)} \; \dd\tau \\
&\qquad - \, \tfrac{\bba_2(\sigma_0) \bb_5(\sigma)}{2} \int_{0}^{t} \scpBig{\Ltwo}{\tfrac{1}{(\alphaphi(\tau))^2}}{\partial_{tt} \phi(\tau) \, \absbig{\nabla \Delta \partial_{t} \psi(\tau)}^2} \; \dd\tau \\
&\qquad
- \tfrac{\bb_3 \bb_5(\sigma)}{2} \int_{0}^{t} \scpBig{\Ltwo}{\tfrac{1}{(\alphaphi(\tau))^2}}{\partial_{tt} \phi(\tau) \, \big(\Delta \partial_{t} \psi(\tau)\big)^2} \; \dd\tau \\
&\qquad + \bba_4(\sigma_0) \, \bb_5(\sigma) \int_{0}^{t} \scpBig{\Ltwo}{\tfrac{1}{(\alphaphi(\tau))^2}}{\Delta \partial_{tt} \psi(\tau) \, \nabla \partial_{t} \phi(\tau) \cdot \nabla \Delta \psi(\tau)} \; \dd\tau \\
&\qquad - \, \bba_4(\sigma_0) \, \bb_5(\sigma) \int_{0}^{t} \scpBig{\Ltwo}{\tfrac{1}{(\alphaphi(\tau))^2}}{\partial_{tt} \phi(\tau) \, \nabla \Delta \partial_{t} \psi(\tau) \cdot \nabla \Delta \psi(\tau)} \; \dd\tau\,; 
\end{split}
\end{equation}
similarly to before, we here set 
\begin{equation}
\begin{split}
\MywEphipsiAtzero{1} &= \tfrac{1}{2} \, \normbig{\Ltwo}{\nabla \psi_2}^2
+ \tfrac{\bba_2(\sigma_0)}{2} \, \normBig{\Ltwo}{\tfrac{1}{\sqrt{1 + \bb_5(\sigma) \, \psi_1}} \, \nabla \Delta \psi_1}^2 \\
&\qquad + \tfrac{\bb_3}{2} \, \normBig{\Ltwo}{\tfrac{1}{\sqrt{1 + \bb_5(\sigma) \, \psi_1 \alphaphi(t)}} \, \Delta \psi_1}^2\,, 
\end{split}
\end{equation}
\end{subequations}
\subsection{Energy estimates}
\paragraph{Objective.}
In the following, we deduce a priori estimates for the energy functionals 
\begin{subequations}
\label{eq:E0E1E2}
\begin{equation}
\begin{split}
\MyEphipsi{0}{t}
&= \tfrac{1}{2} \, \normBig{\Ltwo}{\sqrt{\alphaphi(t)} \, \partial_{tt} \psi(t)}^2
+ \tfrac{\bba_2(\sigma_0)}{4} \, \normbig{\Ltwo}{\Delta \partial_{t} \psi(t)}^2 \\
&\qquad + \tfrac{\bb_3}{2} \, \normbig{\Ltwo}{\nabla \partial_{t} \psi(t)}^2\,, \\
\MyEphipsi{1}{t}
&= \tfrac{1}{2} \, \normbig{\Ltwo}{\nabla \partial_{tt} \psi(t)}^2
+ \tfrac{\bba_2(\sigma_0)}{4} \, \normBig{\Ltwo}{\tfrac{1}{\sqrt{\alphaphi(t)}} \, \nabla \Delta \partial_{t} \psi(t)}^2 \\
&\qquad + \tfrac{\bb_3}{2} \, \normBig{\Ltwo}{\tfrac{1}{\sqrt{\alphaphi(t)}} \, \Delta \partial_{t} \psi(t)}^2\,, 
\end{split}
\end{equation}
on bounded time intervals $[0, T]$;
recall that $\alphaphi = 1 + \bb_5(\sigma) \, \partial_{t} \phi$ and note that the values at the initial time are given by 
\begin{equation}
\begin{split}
\MyEphipsiAtzero{0}
&= \tfrac{1}{2} \, \normBig{\Ltwo}{\sqrt{1 + \bb_5(\sigma) \, \psi_1} \, \psi_2}^2
+ \tfrac{\bba_2(\sigma_0)}{4} \, \normbig{\Ltwo}{\Delta \psi_1}^2 \\
&\qquad + \tfrac{\bb_3}{2} \, \normbig{\Ltwo}{\nabla \psi_1}^2\,, \\
\MyEphipsiAtzero{1}
&= \tfrac{1}{2} \, \normbig{\Ltwo}{\nabla \psi_2}^2
+ \tfrac{\bba_2(\sigma_0)}{4} \, \normBig{\Ltwo}{\tfrac{1}{\sqrt{1 + \bb_5(\sigma) \, \psi_1}} \, \nabla \Delta \psi_1}^2 \\
&\qquad + \tfrac{\bb_3}{2} \, \normBig{\Ltwo}{\tfrac{1}{\sqrt{1 + \bb_5(\sigma) \, \psi_1}} \, \Delta \psi_1}^2\,, 
\end{split}
\end{equation}
see~\eqref{eq:InitialConditionphipsi}.
In order to keep the formulas short, we introduce auxiliary abbreviations for the basic components 
\begin{equation}
\begin{split}
\MyEphipsi{01}{t} &= \normBig{\Ltwo}{\sqrt{\alphaphi(t)} \, \partial_{tt} \psi(t)}^2\,, \\
\MyEphipsi{02}{t} &= \normbig{\Ltwo}{\Delta \partial_{t} \psi(t)}^2\,, \\
\MyEphipsi{03}{t} &= \normbig{\Ltwo}{\nabla \partial_{t} \psi(t)}^2\,, \\
\MyEphipsi{0}{t} &= \tfrac{1}{2} \, \MyEphipsi{01}{t} + \tfrac{\bba_2(\sigma_0)}{4} \, \MyEphipsi{02}{t} \\
&\qquad + \tfrac{\bb_3}{2} \, \MyEphipsi{03}{t}\,, \\
\MyEphipsi{11}{t} &= \normbig{\Ltwo}{\nabla \partial_{tt} \psi(t)}^2\,, \\
\MyEphipsi{12}{t} &= \normBig{\Ltwo}{\tfrac{1}{\sqrt{\alphaphi(t)}} \, \nabla \Delta \partial_{t} \psi(t)}^2\,, \\
\MyEphipsi{13}{t} &= \normBig{\Ltwo}{\tfrac{1}{\sqrt{\alphaphi(t)}} \, \Delta \partial_{t} \psi(t)}^2\,, \\
\MyEphipsi{1}{t} &= \tfrac{1}{2} \, \MyEphipsi{11}{t} + \tfrac{\bba_2(\sigma_0)}{4} \, \MyEphipsi{12}{t} \\
&\qquad + \tfrac{\bb_3}{2} \, \MyEphipsi{13}{t}\,;
\end{split}
\end{equation}
we in particular apply the relations 
\begin{equation}
\begin{gathered}
\MyEphipsi{01}{t} \leq 2 \, \MyEphipsi{0}{t}\,, \quad 
\MyEphipsi{03}{t} \leq \tfrac{2}{\bb_3} \, \MyEphipsi{0}{t}\,, \\
\MyEphipsi{11}{t} \leq 2 \, \MyEphipsi{1}{t}\,, \quad 
\MyEphipsi{13}{t} \leq \tfrac{2}{\bb_3} \, \MyEphipsi{1}{t}\,.
\end{gathered}
\end{equation}
Moreover, we denote  
\begin{equation}
\begin{gathered}
\MyEphipsi{20}{t} = \normBig{\Ltwo}{\tfrac{1}{\sqrt{\alphaphi(t)}} \, \Delta \partial_{tt} \psi(t)}^2\,, \\
\MyEphipsi{2}{t} = \tfrac{1}{4} \, \MywEphipsi{2}{t} = \tfrac{\bba_1}{4} \, \MyEphipsi{20}{t}\,. 
\end{gathered}
\end{equation}
Our essential premise in the proof of Proposition~\ref{thm:Proposition} is boundedness of the energy functionals by positive constants~$\BoundE{0}, \BoundE{1}, \BoundE{2} > 0$, when inserting~$\phi$ twice 
\begin{equation}
\begin{gathered}
\sup_{t \in [0, T]} \MyEphiphi{0}{t} \leq \BoundE{0}\,, \quad 
\sup_{t \in [0, T]} \MyEphiphi{1}{t} \leq \BoundE{1}\,, \\
\int_{0}^{T} \MyEphiphi{2}{t} \; \dd t \leq \BoundE{2}\,;
\end{gathered}
\end{equation}
evidently, this yields the relations 
\begin{equation}
\begin{gathered}
\sup_{t \in [0, T]} \MyEphiphi{01}{t} \leq 2 \, \BoundE{0}\,, \quad \sup_{t \in [0, T]} \MyEphiphi{03}{t} \leq \tfrac{2}{\bb_3} \, \BoundE{0}\,, \\
\sup_{t \in [0, T]} \MyEphiphi{11}{t} \leq 2 \, \BoundE{1}\,, \quad \sup_{t \in [0, T]} \MyEphiphi{13}{t} \leq \tfrac{2}{\bb_3} \, \BoundE{1}\,.
\end{gathered}
\end{equation}
\end{subequations}
We note that $\bba_2(\sigma_0) \to 0$ if $a \to 0_{+}$;
for this reason, $E_{02}$ will be related to~$E_{13}$, employing uniform boundedness of~$\alphaphi$ from above and below.

\paragraph{Basic auxiliary estimates.}
Considering in the first instance regular bounded spatial domains $\Omega \subset \RR^3$, we exploit the Poincar\'{e}--Friedrichs inequality, the continuous embeddings $\Hone(\Omega) \hookrightarrow \Lsix(\Omega)$ as well as  $\Htwo(\Omega) \hookrightarrow \Linfty(\Omega)$, and assume elliptic regularity; 
the application of Hölder's inequality with exponent $p = 3$ and conjugate exponent $p^{*} = \frac{p}{p-1} = \frac{3}{2}$ also shows $\Hone(\Omega) \hookrightarrow \Lsix(\Omega) \hookrightarrow \Lfour(\Omega)$, since 
\begin{equation*}
\normbig{\Lfour}{f}^4 = \int_{\Omega} \big(f(x)\big)^4 \; \dd x
\leq \bigg(\int_{\Omega} 1 \; \dd x\bigg)^{\frac{1}{p}} \, \bigg(\int_{\Omega} \big(f(x)\big)^{4 p^{*}} \; \dd x\bigg)^{\frac{1}{p^{*}}} 
= \abs{\Omega}^{\frac{1}{3}} \, \normbig{\Lsix}{f}^{4}\,.
\end{equation*}
To summarise, we apply the estimates
\begin{equation}
\label{eq:PoincareEmbeddingEllipticity}
\begin{gathered}
\normbig{\Hone}{f} \leq C_{\text{PF}} \, \normbig{\Ltwo}{\nabla f}\,, \quad f \in \Honezero(\Omega)\,, \\
\normbig{\Lfour}{f} \leq C_{\Lfour \leftarrow \Hone} \, \normbig{\Hone}{f}\,, \quad 
\normbig{\Lsix}{f} \leq C_{\Lsix \leftarrow \Hone} \, \normbig{\Hone}{f}\,, \quad f \in \Hone(\Omega)\,, \\
\normbig{\Linfty}{f} \leq C_{\Linfty \leftarrow \Htwo} \, \normbig{\Htwo}{f}\,, \quad f \in \Htwo(\Omega)\,, \\
\normbig{\Htwo}{f} \leq C_{\Delta} \, \normbig{\Ltwo}{\Delta f}\,, \quad f \in \Htwo(\Omega) \cap \Honezero(\Omega)\,;
\end{gathered}
\end{equation}
in all cases, the arising constant depends on the space domain.

\paragraph{Gronwall-type inequality.}
In addition, we make use of the fact that a non-negative function $f: [0, T] \to \RR$ satisfying a differential equation of the form 
\begin{equation*}
f'(t) = \gamma^2 \, f(t) + g(t)
\end{equation*}
with positive weight $\gamma > 0$ and non-negative function $g: [0, T] \to \RR$ or the corresponding integral equation 
\begin{equation*}
f(t) = f(0) + \gamma^2 \int_{0}^{t} f(\tau) \; \dd\tau + \int_{0}^{t} g(\tau) \; \dd\tau\,,
\end{equation*}
respectively, is given by 
\begin{equation*}
f(t) = \ee^{\gamma^2 t} \, f(0) + \int_{0}^{t} \ee^{\gamma^2 (t - \tau)} \, g(\tau) \; \dd\tau
\end{equation*}
and in particular satisfies the bound 
\begin{equation*}
f(t) \leq \ee^{\gamma^2 t} \, \bigg(f(0) + \int_{0}^{t} g(\tau) \; \dd\tau\bigg)\,. 
\end{equation*}
We apply this relation to a function of the form $f(t) = \norm{\Ltwo}{\varphi(t)}^2$.
More precisely, integration with respect to time and straightforward estimation by Cauchy's inequality as well as Young's inequality with weight $\gamma > 0$ implies 
\begin{equation*}
\begin{split}
&\normbig{\Ltwo}{\varphi(t)}^2 
= \normbig{\Ltwo}{\varphi(0)}^2 
+ 2 \int_{0}^{t} \scpbig{\Ltwo}{\partial_{t} \varphi(\tau)}{\varphi(\tau)} \; \dd\tau \\
&\quad \leq \normbig{\Ltwo}{\varphi(0)}^2 
+ 2 \int_{0}^{t} \normbig{\Ltwo}{\partial_{t} \varphi(\tau)} \, \normbig{\Ltwo}{\varphi(\tau)} \; \dd\tau \\
&\quad \leq \normbig{\Ltwo}{\varphi(0)}^2 
+ \gamma^2 \int_{0}^{t} \normbig{\Ltwo}{\varphi(\tau)}^2 \; \dd\tau
+ \tfrac{1}{\gamma^2} \int_{0}^{t} \normbig{\Ltwo}{\partial_{t} \varphi(\tau)}^2 \; \dd\tau\,;
\end{split}
\end{equation*}
with the help of the above estimate and the special choice $\gamma = \tfrac{1}{\sqrt{t}}$ such that $\ee^{\gamma^2 t} = e \leq 3$ and $\tfrac{1}{\gamma^2} = t \leq T$, this further shows 
\begin{equation}
\label{eq:Gronwall}
\normbig{\Ltwo}{\varphi(t)}^2 \leq 3 \, \normbig{\Ltwo}{\varphi(0)}^2 + 3 \, T \int_{0}^{t} \normbig{\Ltwo}{\partial_{t} \varphi(\tau)}^2 \; \dd\tau\,.
\end{equation}
 
\paragraph{Auxiliary estimates ensuring non-degeneracy.}
We first prove that the time-dependent function $\alphaphi = 1 + \bb_5(\sigma) \, \partial_{t} \phi$ defined in~\eqref{eq:GeneralModelReformulationDifferentiation} is uniformly bounded from below and above   
\begin{subequations}
\label{eq:BoundednessAlpha}
\begin{equation}
0 < \underline{\alpha} = \tfrac{1}{2} \leq \normbig{\Linfty([0, T], \Linfty(\Omega))}{\alphaphi} \leq \overline{\alpha} = \tfrac{3}{2}\,, 
\end{equation}
provided that the upper bound for the higher-order energy functional on the considered time interval $[0, T]$ satisfies the smallness requirement 
\begin{equation}
C_0 \, \BoundE{1} \leq \tfrac{1}{12}\,, \quad C_0 = \tfrac{(C_{\Delta} C_{\Linfty \leftarrow \Htwo} \bb_5(\sigma))^2}{\bb_3}\,,
\end{equation}
see also~\eqref{eq:GeneralModel}, ~\eqref{eq:E0E1E2} and~\eqref{eq:PoincareEmbeddingEllipticity};
we point out that the arising constant $C_0 > 0$ does not dependent on $a > 0$.
With regard to the relation 
\begin{equation*}
\absBig{1 - \normbig{\Linfty}{\alphaphi(t) - 1}} \leq \normbig{\Linfty}{\alphaphi(t)} \leq 1 + \normbig{\Linfty}{\alphaphi(t) - 1}
\end{equation*}
obtained by triangular inequalities, it remains to show boundedness of $\norm{\Linfty}{\alphaphi(t) - 1}$ for any $t \in [0, T]$.
By means of~\eqref{eq:PoincareEmbeddingEllipticity}, we have 
\begin{equation*}
\begin{split}
&\normbig{\Linfty}{\alphaphi(t) - 1} = \bb_5(\sigma) \, \normbig{\Linfty}{\partial_{t} \phi(t)} 
\leq C_{\Linfty \leftarrow \Htwo} \, \bb_5(\sigma) \, \normbig{\Htwo}{\partial_{t} \phi(t)} \\
&\quad \leq C_{\Delta} \, C_{\Linfty \leftarrow \Htwo} \, \bb_5(\sigma) \, \normbig{\Ltwo}{\Delta \partial_{t} \phi(t)} \\
&\quad \leq C_{\Delta} \, C_{\Linfty \leftarrow \Htwo} \, \bb_5(\sigma) \, \normbig{\Linfty}{\sqrt{\alphaphi(t)}} \, 
\normbig{\Ltwo}{\tfrac{1}{\sqrt{\alphaphi(t)}} \, \Delta \partial_{t} \phi(t)} \\
&\quad \leq \sqrt{C_0} \, \sqrt{\bb_3 \, \MyEphiphi{13}{t}} \, \sqrt{\normbig{\Linfty}{\alphaphi(t)}} \\
&\quad \leq \sqrt{2 \, C_0 \, \BoundE{1}} \, \sqrt{1 + \normbig{\Linfty}{\alphaphi(t) - 1}}\,, 
\end{split}
\end{equation*}
see also~\eqref{eq:E0E1E2}.
Due to the smallness requirement $C_0 \, \BoundE{1} \leq \tfrac{1}{12}$, the positive solution to this inequality satisfies 
\begin{equation*}
\begin{gathered}
\eta \leq \sqrt{2 \, C_0 \, \BoundE{1}} \, \sqrt{1 + \eta}\,, \quad
\eta^2 - 2 \, C_0 \, \BoundE{1} \, \eta - 2 \, C_0 \, \BoundE{1} \leq 0\,, \\
\big(\eta - C_0 \, \BoundE{1}\big)^2 \leq \big(2 + C_0 \, \BoundE{1}\big) \, C_0 \, \BoundE{1}\,, \quad 
0 \leq \eta \leq C_0 \, \BoundE{1} + \sqrt{\big(2 + C_0 \, \BoundE{1}\big) \, C_0 \, \BoundE{1}} \leq \tfrac{1}{2}\,; 
\end{gathered}
\end{equation*}
this implies the stated relation, since  
\begin{equation*}
\tfrac{1}{2} \leq \absBig{1 - \normbig{\Linfty}{\alphaphi(t) - 1}} \leq \normbig{\Linfty}{\alphaphi(t)} \leq 1 + \normbig{\Linfty}{\alphaphi(t) - 1} \leq \tfrac{3}{2}\,, 
\end{equation*}
and in particular ensures non-degeneracy
\begin{equation}
0 < \tfrac{1}{\overline{\alpha}} = \tfrac{2}{3} \leq \normbig{\Linfty([0, T], \Linfty(\Omega))}{\tfrac{1}{\alphaphi}} \leq \tfrac{1}{\underline{\alpha}} = 2\,.
\end{equation}
\end{subequations}

\paragraph{Auxiliary estimate for nonlinearity.}
We next deduce an auxiliary estimate for the nonlinearity
\begin{equation*}
\rphi 
= \bb_5(\sigma) \, \partial_{tt} \psi \, \partial_{tt} \phi 
+ 2 \, \bb_6(\sigma) \, \nabla \partial_{tt} \psi \cdot \nabla \phi 
+ 2 \, \bb_6(\sigma) \, \nabla \partial_{t} \psi \cdot \nabla \partial_{t} \phi\,,
\end{equation*}
see~\eqref{eq:alpha_r_phi} and recall~\eqref{eq:E0E1E2}. 
The estimation of the first term uses Cauchy's inequality and relation~\eqref{eq:PoincareEmbeddingEllipticity}; 
that is, we have 
\begin{equation*}
\begin{split}
&\normBig{\Ltwo}{\partial_{tt} \psi(t) \, \partial_{tt} \phi(t)}^2 
\leq \normbig{\Lfour}{\partial_{tt} \psi(t)}^2 \, \normbig{\Lfour}{\partial_{tt} \phi(t)}^2
\leq C_{\Lfour \leftarrow \Hone}^4 \, \normbig{\Hone}{\partial_{tt} \psi(t)}^2 \, \normbig{\Hone}{\partial_{tt} \phi(t)}^2 \\
&\quad \leq C_{\text{PF}}^4 \, C_{\Lfour \leftarrow \Hone}^4 \, \normbig{\Ltwo}{\nabla \partial_{tt} \psi(t)}^2 \, \normbig{\Ltwo}{\nabla \partial_{tt} \phi(t)}^2 \\
&\quad \leq C_{\text{PF}}^4 \, C_{\Lfour \leftarrow \Hone}^4 \, \MyEphipsi{11}{t} \, \MyEphiphi{11}{t} \\
&\quad \leq 4 \, C_{\text{PF}}^4 \, C_{\Lfour \leftarrow \Hone}^4 \, \BoundE{1} \, \MyEphipsi{1}{t}\,. 
\end{split}
\end{equation*}
For the third term, we apply the same arguments and use boundedness of~$\alphaphi$ by $\overline{\alpha} = \frac{3}{2}$, see~\eqref{eq:BoundednessAlpha}, to obtain 
\begin{equation*}
\begin{split}
&\normBig{\Ltwo}{\nabla \partial_{t} \psi(t) \cdot \nabla \partial_{t} \phi(t)}^2 
\leq \normbig{\Lfour}{\nabla \partial_{t} \psi(t)}^2 \, \normbig{\Lfour}{\nabla \partial_{t} \phi(t)}^2 \\
&\quad \leq C_{\Lfour \leftarrow \Hone}^4 \, \normbig{\Hone}{\nabla \partial_{t} \psi(t)}^2 \, \normbig{\Hone}{\nabla \partial_{t} \phi(t)}^2 
\leq C_{\text{PF}}^4 \, C_{\Lfour \leftarrow \Hone}^4 \, \normbig{\Ltwo}{\Delta \partial_{t} \psi(t)}^2 \, \normbig{\Ltwo}{\Delta \partial_{t} \phi(t)}^2 \\
&\quad \leq C_{\text{PF}}^4 \, C_{\Lfour \leftarrow \Hone}^4 \, \normbig{\Linfty}{\alphaphi(t)}^2 \, 
\normbig{\Ltwo}{\tfrac{1}{\sqrt{\alphaphi(t)}} \, \Delta \partial_{t} \psi(t)}^2 \, \normbig{\Ltwo}{\tfrac{1}{\sqrt{\alphaphi(t)}} \, \Delta \partial_{t} \phi(t)}^2 \\
&\quad \leq C_{\text{PF}}^4 \, C_{\Lfour \leftarrow \Hone}^4 \, \overline{\alpha}^2 \, \MyEphipsi{13}{t} \, \MyEphiphi{13}{t} \\
&\quad \leq \tfrac{9 \, C_{\text{PF}}^4 C_{\Lfour \leftarrow \Hone}^4}{\bb_3^2} \, \BoundE{1} \, \MyEphipsi{1}{t}\,.
\end{split}
\end{equation*}
For the second term, we in addition employ the Gronwall-type inequality~\eqref{eq:Gronwall} with $\varphi = \Delta \phi$; 
this yields 
\begin{equation*}
\begin{split}
&\normbig{\Ltwo}{\nabla \partial_{tt} \psi(t) \cdot \nabla \phi(t)}^2
\leq \normbig{\Lfour}{\nabla \partial_{tt} \psi(t)}^2 \, \normbig{\Lfour}{\nabla \phi(t)}^2 \\
&\quad \leq C_{\Lfour \leftarrow \Hone}^4 \, \normbig{\Hone}{\nabla \partial_{tt} \psi(t)}^2 \, \normbig{\Hone}{\nabla \phi(t)}^2
\leq C_{\text{PF}}^4 \, C_{\Lfour \leftarrow \Hone}^4 \, \normbig{\Ltwo}{\Delta \partial_{tt} \psi(t)}^2 \, \normbig{\Ltwo}{\Delta \phi(t)}^2 \\
&\quad \leq C_{\text{PF}}^4 \, C_{\Lfour \leftarrow \Hone}^4 \, \normbig{\Ltwo}{\Delta \partial_{tt} \psi(t)}^2 \, 
\bigg(3 \, \normbig{\Ltwo}{\Delta \psi_0}^2 + 3 \, T \int_{0}^{t} \normbig{\Ltwo}{\Delta \partial_{t} \phi(\tau)}^2 \; \dd\tau\bigg) \\
&\quad \leq 3 \, C_{\text{PF}}^4 \, C_{\Lfour \leftarrow \Hone}^4 \, \overline{\alpha} \, \normbig{\Ltwo}{\tfrac{1}{\sqrt{\alphaphi(t)}} \, \Delta \partial_{tt} \psi(t)}^2 \\
&\quad\qquad \times 
\bigg(\normbig{\Ltwo}{\Delta \psi_0}^2 + \overline{\alpha} \, T \int_{0}^{t} \normbig{\Ltwo}{\tfrac{1}{\sqrt{\alphaphi(t)}} \, \Delta \partial_{t} \phi(\tau)}^2 \; \dd\tau\bigg) \\
&\quad \leq 3 \, C_{\text{PF}}^4 \, C_{\Lfour \leftarrow \Hone}^4 \, \overline{\alpha} \, \MyEphipsi{20}{t} \, 
\Big(\normbig{\Ltwo}{\Delta \psi_0}^2 + \overline{\alpha} \, T^2 \, \sup_{t \in [0, T]} \MyEphiphi{13}{t}\Big) \\
&\quad \leq \tfrac{9 \, C_{\text{PF}}^4 C_{\Lfour \leftarrow \Hone}^4}{2 \, \bba_1} \, \MywEphipsi{2}{t} \, 
\Big(\normbig{\Ltwo}{\Delta \psi_0}^2 + \tfrac{3}{\bb_3} \, T^2 \, \BoundE{1}\Big)\,.
\end{split}
\end{equation*}
By the elementary inequality $(a_1 + a_2 + a_3)^2 \leq 3 \, (a_1^2 + a_2^2 + a_3^2)$, valid for positive real numbers $a_1, a_2, a_3 > 0$, the bound 
\begin{equation*}
\begin{split}
&\int_{0}^{t} \normbig{\Ltwo}{\rphi(\tau)}^2 \; \dd\tau \\
&\quad \leq 3 \int_{0}^{t} \Big(\big(\bb_5(\sigma)\big)^2 \, \normBig{\Ltwo}{\partial_{tt} \psi(\tau) \, \partial_{tt} \phi(\tau)}^2 
+ 4 \, \big(\bb_6(\sigma)\big)^2 \, \normBig{\Ltwo}{\nabla \partial_{t} \psi(\tau) \, \nabla \partial_{t} \phi(\tau)}^2 \\
&\quad\qquad + 4 \, \big(\bb_6(\sigma)\big)^2 \, \normbig{\Ltwo}{\nabla \partial_{tt} \psi(\tau) \cdot \nabla \phi(\tau)}^2\Big) \; \dd\tau \\
&\quad \leq 12 \, C_{\text{PF}}^4 \, C_{\Lfour \leftarrow \Hone}^4 \, \Big(\big(\bb_5(\sigma)\big)^2 + \tfrac{9 \, (\bb_6(\sigma))^2}{\bb_3^2}\Big) \, \BoundE{1} 
\int_{0}^{t} \MyEphipsi{1}{\tau} \; \dd\tau \\
&\quad\qquad + \tfrac{54 \, \, C_{\text{PF}}^4 C_{\Lfour \leftarrow \Hone}^4 (\bb_6(\sigma))^2}{\bba_1} \, 
\Big(\normbig{\Ltwo}{\Delta \psi_0}^2 + \tfrac{3}{\bb_3} \, T^2 \, \BoundE{1}\Big) \int_{0}^{t} \MywEphipsi{2}{\tau} \; \dd\tau
\end{split}
\end{equation*}
follows. 
With the help of the abbreviations 
\begin{subequations}
\label{eq:EstimateNonlinearity}
\begin{equation}
\label{eq:C1C2C3}
\begin{gathered}
C_1 = 12 \, C_{\text{PF}}^4 \, C_{\Lfour \leftarrow \Hone}^4 \, \Big(\big(\bb_5(\sigma)\big)^2 + \tfrac{9 \, (\bb_6(\sigma))^2}{\bb_3^2}\Big)\,, \\
C_2 = \tfrac{54 \, C_{\text{PF}}^4 C_{\Lfour \leftarrow \Hone}^4 (\bb_6(\sigma))^2}{\bba_1}\,, \quad C_3 = \tfrac{3}{\bb_3}\,, 
\end{gathered}
\end{equation}
we arrive at the auxiliary estimate 
\begin{equation}
\begin{split}
&\int_{0}^{t} \normbig{\Ltwo}{\rphi(\tau)}^2 \; \dd\tau \\
&\quad \leq C_1 \, \BoundE{1} \int_{0}^{t} \MyEphipsi{1}{\tau} \; \dd\tau \\
&\quad\qquad + C_2 \, \Big(\normbig{\Ltwo}{\Delta \psi_0}^2 + C_3 \, T^2 \, \BoundE{1}\Big) \int_{0}^{t} \MywEphipsi{2}{\tau} \; \dd\tau \\
&\quad \leq C_1 \, \BoundE{1} \int_{0}^{t} \MyEphipsi{1}{\tau} \; \dd\tau \\
&\quad\qquad + 4 \, C_2 \, \Big(\normbig{\Ltwo}{\Delta \psi_0}^2 + C_3 \, T^2 \, \BoundE{1}\Big) \int_{0}^{t} \MyEphipsi{2}{\tau} \; \dd\tau \\
\end{split}
\end{equation}
\end{subequations}

\paragraph{First energy estimate.}
Our starting point is~\eqref{eq:GeneralModelReformulationTest1}, which we restate for convenience 
\begin{equation*}
\begin{split}
&\MywEphipsi{0}{t} + \bba_1 \int_{0}^{t} \MyEphipsi{11}{\tau} \; \dd\tau \\
&\quad = \MywEphipsiAtzero{0}
+ \bba_4(\sigma_0) \, \scpbig{\Ltwo}{\Delta \psi_1}{\Delta \psi_0} 
- \bba_4(\sigma_0) \, \scpbig{\Ltwo}{\Delta \partial_{t} \psi(t)}{\Delta \psi(t)} \\
&\quad\qquad + \bba_4(\sigma_0) \int_{0}^{t} \normbig{\Ltwo}{\Delta \partial_{t} \psi(\tau)}^2 \; \dd\tau \\
&\quad\qquad + \tfrac{1}{2} \int_{0}^{t} \scpbig{\Ltwo}{\partial_t \alphaphi(\tau) \, \partial_{tt} \psi(\tau)}{\partial_{tt} \psi(\tau)} \; \dd\tau 
- \int_{0}^{t} \scpbig{\Ltwo}{\rphi(\tau)}{\partial_{tt} \psi(\tau)} \; \dd\tau\,,
\end{split}
\end{equation*}
see also~\eqref{eq:E0E1E2}.
In order to suitably estimate and absorb the terms arising on the right-hand side, we proceed as follows.
\begin{enumerate}[(i)]
\item 
By means of Cauchy's inequality and Young's inequality, we have 
\begin{equation*}
\begin{split}
&\bba_4(\sigma_0) \, \absBig{\scpbig{\Ltwo}{\Delta \psi_1}{\Delta \psi_0}} 
\leq \bba_4(\sigma_0) \, \normbig{\Ltwo}{\Delta \psi_1} \, \normbig{\Ltwo}{\Delta \psi_0} \\
&\quad \leq \tfrac{\bba_4(\sigma_0)}{2} \, \normbig{\Ltwo}{\Delta \psi_1}^2 
+ \tfrac{\bba_4(\sigma_0)}{2} \, \normbig{\Ltwo}{\Delta \psi_0}^2\,. 
\end{split}
\end{equation*}
\item 
In a similar manner, incorporating an additional weight $\gamma_1 > 0$, we obtain 
\begin{equation*}
\begin{split}
&\bba_4(\sigma_0) \, \absBig{\scpbig{\Ltwo}{\Delta \partial_{t} \psi(t)}{\Delta \psi(t)}} 
\leq \bba_4(\sigma_0) \, \normbig{\Ltwo}{\Delta \partial_{t} \psi(t)} \, \normbig{\Ltwo}{\Delta \psi(t)} \\
&\quad \leq \tfrac{\gamma_1^2 \bba_4(\sigma_0)}{2} \, \normbig{\Ltwo}{\Delta \partial_{t} \psi(t)}^2 
+ \tfrac{\bba_4(\sigma_0)}{2 \gamma_1^2} \, \normbig{\Ltwo}{\Delta \psi(t)}^2\,;
\end{split}
\end{equation*}
with regard to the relation $\bba_2(\sigma_0) = \bba_0(\sigma_0) \, \bba_4(\sigma_0)$, we set $\gamma_1^2 = \tfrac{\bba_0(\sigma_0)}{2}$ such that 
\begin{equation*}
\begin{split}
&\bba_4(\sigma_0) \, \absBig{\scpbig{\Ltwo}{\Delta \partial_{t} \psi(t)}{\Delta \psi(t)}} \\
&\quad \leq \tfrac{\bba_2(\sigma_0)}{4} \, \normbig{\Ltwo}{\Delta \partial_{t} \psi(t)}^2 
+ \tfrac{\bba_4(\sigma_0)}{\bba_0(\sigma_0)} \, \normbig{\Ltwo}{\Delta \psi(t)}^2\,. 
\end{split}
\end{equation*}
This permits to absorb the first term involving~$\norm{\Ltwo}{\Delta \partial_{t} \psi(t)}^2$ and explains the definition of the energy functional 
\begin{equation*}
\begin{split}
&\MyEphipsi{0}{t} 
= \MywEphipsi{0}{t} - \tfrac{\bba_2(\sigma_0)}{4} \, \normbig{\Ltwo}{\Delta \partial_{t} \psi(t)}^2 \\
&\quad= \tfrac{1}{2} \, \normbig{\Ltwo}{\sqrt{\alphaphi(t)} \, \partial_{tt} \psi(t)}^2
+ \tfrac{\bba_2(\sigma_0)}{4} \, \normbig{\Ltwo}{\Delta \partial_{t} \psi(t)}^2
+ \tfrac{\bb_3}{2} \, \normbig{\Ltwo}{\nabla \partial_{t} \psi(t)}^2\,; 
\end{split}
\end{equation*}
for the second term, we apply the Gronwall-type inequality~\eqref{eq:Gronwall} with $\varphi = \Delta \psi$, which yields 
\begin{equation*}
\normbig{\Ltwo}{\Delta \psi (t)}^2 \leq 3 \, \normbig{\Ltwo}{\Delta \psi_0}^2 + 3 \, T \int_{0}^{t} \normbig{\Ltwo}{\Delta \partial_{t} \psi(\tau)}^2 \; \dd\tau\,.
\end{equation*}
\item 
Again by Cauchy's inequality, we have 
\begin{equation*}
\scpBig{\Ltwo}{\big(\partial_{tt} \psi(\tau)\big)^2}{\partial_{tt} \phi(\tau)} 
\leq \normbig{\Lfour}{\partial_{tt} \psi(\tau)}^2 \, \normbig{\Ltwo}{\partial_{tt} \phi(\tau)}\,;
\end{equation*}
relation~\eqref{eq:PoincareEmbeddingEllipticity} and the uniform bound $\tfrac{1}{\underline{\alpha}} = 2$, see~\eqref{eq:BoundednessAlpha}, imply
\begin{equation*}
\begin{split}
&\tfrac{1}{2} \int_{0}^{t} \scpbig{\Ltwo}{\partial_{t} \alphaphi(\tau) \, \partial_{tt} \psi(\tau)}{\partial_{tt} \psi(\tau)} \; \dd\tau \\
&\quad \leq \tfrac{\bb_5(\sigma)}{2} \int_{0}^{t} \normbig{\Lfour}{\partial_{tt} \psi(\tau)}^2 \, \normbig{\Ltwo}{\partial_{tt} \phi(\tau)} \; \dd\tau \\
&\quad \leq \tfrac{C_{\text{PF}}^2 C_{\Lfour \leftarrow \Hone}^2 \bb_5(\sigma)}{2} \, \sqrt{\tfrac{1}{\underline{\alpha}}} 
\int_{0}^{t} \normbig{\Ltwo}{\nabla \partial_{tt} \psi(\tau)}^2 \, \normBig{\Ltwo}{\sqrt{\alphaphi(\tau)} \, \partial_{tt} \phi(\tau)} \; \dd\tau \\
&\quad \leq \tfrac{C_{\text{PF}}^2 C_{\Lfour \leftarrow \Hone}^2 \bb_5(\sigma)}{2} \, \sqrt{\tfrac{1}{\underline{\alpha}}} 
\int_{0}^{t} \sqrt{\MyEphiphi{01}{\tau}} \, \MyEphipsi{11}{\tau} \; \dd\tau \\
&\quad \leq C_{\text{PF}}^2 \, C_{\Lfour \leftarrow \Hone}^2 \, \bb_5(\sigma) \, \sqrt{\BoundE{0}} \int_{0}^{t} \MyEphipsi{11}{\tau} \; \dd\tau\,.
\end{split}
\end{equation*}
Provided that the smallness requirement 
\begin{equation*}
\tfrac{C_{\text{PF}}^2 C_{\Lfour \leftarrow \Hone}^2 \bb_5(\sigma)}{\bba_1} \, \sqrt{\BoundE{0}} \leq \tfrac{1}{2}
\end{equation*}
is satisfied, the resulting term 
\begin{equation*}
\tfrac{1}{2} \int_{0}^{t} \scpbig{\Ltwo}{\partial_{t} \alphaphi(\tau) \, \partial_{tt} \psi(\tau)}{\partial_{tt} \psi(\tau)} \; \dd\tau
\leq \tfrac{\bba_1}{2} \int_{0}^{t} \MyEphipsi{11}{\tau} \; \dd\tau
\end{equation*}
can be absorbed by the corresponding term arising on the left-hand side.
\item 
Cauchy's inequality and Young's inequality with weight $\gamma_2 > 0$ and~\eqref{eq:PoincareEmbeddingEllipticity} yields 
\begin{equation*}
\begin{split}
&\absBig{\scpbig{\Ltwo}{\rphi(\tau)}{\partial_{tt} \psi(\tau)}}
\leq \normbig{\Ltwo}{\rphi(\tau)} \normbig{\Ltwo}{\partial_{tt} \psi(\tau)} \\
&\quad \leq \tfrac{1}{2 \gamma_2^2} \, \normbig{\Ltwo}{\rphi(\tau)}^2 + \tfrac{\gamma_2^2}{2} \, \normbig{\Ltwo}{\partial_{tt} \psi(\tau)}^2 
\leq \tfrac{1}{2 \gamma_2^2} \, \normbig{\Ltwo}{\rphi(\tau)}^2 + \tfrac{\gamma_2^2}{2} \, \normbig{\Hone}{\partial_{tt} \psi(\tau)}^2 \\
&\quad \leq \tfrac{1}{2 \gamma_2^2} \, \normbig{\Ltwo}{\rphi(\tau)}^2 + \tfrac{C_{\text{PF}}^2 \gamma_2^2}{2} \, \MyEphipsi{11}{\tau}\,;
\end{split}
\end{equation*}
with the special choice $\gamma_2^2 = \frac{\bba_1}{2 \, C_{\text{PF}}^2}$ such that $\frac{C_{\text{PF}}^2 \gamma_2^2}{2} = \frac{\bba_1}{4}$ the second term arising on the right-hand side of 
\begin{equation*}
\begin{split}
&\int_{0}^{t} \absBig{\scpbig{\Ltwo}{\rphi(\tau)}{\partial_{tt} \psi(\tau)}} \; \dd\tau \\
&\quad \leq \tfrac{C_{\text{PF}}^2}{\bba_1} \int_{0}^{t} \normbig{\Ltwo}{\rphi(\tau)}^2 \; \dd\tau + \tfrac{\bba_1}{4} \int_{0}^{t} \MyEphipsi{11}{\tau} \; \dd\tau
\end{split}
\end{equation*}
can be absorbed. 
\end{enumerate}
The above considerations imply the estimate 
\begin{equation*}
\begin{split}
&\MyEphipsi{0}{t} + \tfrac{\bba_1}{4} \int_{0}^{t} \MyEphipsi{11}{\tau} \; \dd\tau \\
&\quad \leq \MywEphipsiAtzero{0} 
+ \tfrac{\bba_4(\sigma_0)}{2} \, \normbig{\Ltwo}{\Delta \psi_1}^2 
+ \bba_4(\sigma_0) \, \Big(\tfrac{1}{2} + \tfrac{3}{\bba_0(\sigma_0)}\Big) \, \normbig{\Ltwo}{\Delta \psi_0}^2 \\
&\quad\qquad
+ \bba_4(\sigma_0) \, \Big(1 + \tfrac{3 \, T}{\bba_0(\sigma_0)}\Big) \int_{0}^{t} \normbig{\Ltwo}{\Delta \partial_{t} \psi(\tau)}^2 \; \dd\tau 
+ \tfrac{C_{\text{PF}}^2}{\bba_1} \int_{0}^{t} \normbig{\Ltwo}{\rphi(\tau)}^2 \; \dd\tau\,; 
\end{split}
\end{equation*}
by means of relation~\eqref{eq:BoundednessAlpha} providing the uniform bound~$\overline{\alpha} = \tfrac{3}{2}$ and estimate~\eqref{eq:EstimateNonlinearity}, we further obtain 
\begin{equation*}
\begin{split}
&\MyEphipsi{0}{t} + \tfrac{\bba_1}{4} \int_{0}^{t} \MyEphipsi{11}{\tau} \; \dd\tau \\
&\quad \leq \MywEphipsiAtzero{0} 
+ \tfrac{\bba_4(\sigma_0)}{\bb_3} \, \overline{\alpha} \, \MyEphipsiAtzero{1} 
+ \bba_4(\sigma_0) \, \Big(\tfrac{1}{2} + \tfrac{3}{\bba_0(\sigma_0)}\Big) \, \normbig{\Ltwo}{\Delta \psi_0}^2 \\
&\quad\qquad + \bigg(\tfrac{2 \, \bba_4(\sigma_0)}{\bb_3} \, \Big(1 + \tfrac{3 \, T}{\bba_0(\sigma_0)}\Big) \, \overline{\alpha} 
+ \tfrac{C_{\text{PF}}^2 C_1}{\bba_1} \, \BoundE{1}\bigg) \int_{0}^{t} \MyEphipsi{1}{\tau} \; \dd\tau \\
&\quad\qquad + \tfrac{4 \, C_{\text{PF}}^2 C_2}{\bba_1} \, \Big(\normbig{\Ltwo}{\Delta \psi_0}^2 + C_3 \, T^2 \, \BoundE{1}\Big) 
\int_{0}^{t} \MyEphipsi{2}{\tau} \; \dd\tau \\
&\quad \leq \MywEphipsiAtzero{0} 
+ \tfrac{3 \, \bba_4(\sigma_0)}{2 \, \bb_3} \, \MyEphipsiAtzero{1} 
+ \bba_4(\sigma_0) \, \Big(\tfrac{1}{2} + \tfrac{3}{\bba_0(\sigma_0)}\Big) \, \normbig{\Ltwo}{\Delta \psi_0}^2 \\
&\quad\qquad + \bigg(\tfrac{3 \, \bba_4(\sigma_0)}{\bb_3} \, \Big(1 + \tfrac{3 \, T}{\bba_0(\sigma_0)}\Big) 
+ \tfrac{C_{\text{PF}}^2 C_1}{\bba_1} \, \BoundE{1}\bigg) \int_{0}^{t} \MyEphipsi{1}{\tau} \; \dd\tau \\
&\quad\qquad + \tfrac{4 C_{\text{PF}}^2 C_2}{\bba_1} \, \Big(\normbig{\Ltwo}{\Delta \psi_0}^2 + C_3 \, T^2 \, \BoundE{1}\Big) 
\int_{0}^{t} \MyEphipsi{2}{\tau} \; \dd\tau\,.
\end{split}
\end{equation*}
Altogether, this shows that the lower-order energy functional is bounded in terms of the higher-order energy functional. 
More precisely, applying~\eqref{eq:BoundednessAlpha} as well as~\eqref{eq:PoincareEmbeddingEllipticity} to estimate~$\MywEphipsiAtzero{0}$ by~$\MyEphipsiAtzero{1}$ and recalling the definitions of $C_1, C_2, C_3$, see~\eqref{eq:C1C2C3}, we arrive at a relation of the form 
\begin{equation*}
\begin{split}
&\MyEphipsi{0}{t} \\
&\quad \leq \Phi_0\Big(C_{\text{PF}}, C_{\Lfour \leftarrow \Hone}, \tfrac{1}{\bba_0(\sigma_0)}, \tfrac{1}{\bba_1}, \bba_2(\sigma_0), \tfrac{1}{\bb_3}, \bba_4(\sigma_0), \bb_5(\sigma), \bb_6(\sigma), T\Big) \\ 
&\quad\qquad \times \bigg(\MyEphipsiAtzero{1} + \normbig{\Ltwo}{\Delta \psi_0}^2 
+ \big(1 + \BoundE{1}\big) \int_{0}^{t} \MyEphipsi{1}{\tau} \; \dd\tau \\
&\quad\qquad\qquad + \Big(\normbig{\Ltwo}{\Delta \psi_0}^2 + \BoundE{1}\Big) \int_{0}^{t} \MyEphipsi{2}{\tau} \; \dd\tau\bigg)\,;
\end{split}
\end{equation*}
we note that $\tfrac{1}{\bba_0(\sigma_0)}$, $\tfrac{1}{\bba_1}$, $\bba_2(\sigma_0)$, and $\bba_4(\sigma_0)$ remain bounded for $a \to 0_{+}$. 

\paragraph{Second energy estimate.}
In order to deduce a suitable a priori estimate for the higher-order energy functional, our starting point is 
\begin{equation*}
\begin{split}
&\MywEphipsi{1}{t} + \int_{0}^{t} \MywEphipsi{2}{\tau} \; \dd\tau \\
&\quad = \MywEphipsiAtzero{1} 
+ \bba_4(\sigma_0) \, \scpBig{\Ltwo}{\tfrac{1}{\alphaphi(0)}}{\nabla \Delta \psi_1 \cdot \nabla \Delta \psi_0} \\
&\qquad - \, \bba_4(\sigma_0) \, \scpBig{\Ltwo}{\tfrac{1}{\alphaphi(t)}}{\nabla \Delta \partial_{t} \psi(t) \cdot \nabla \Delta \psi(t)} 
+ \bba_4(\sigma_0) \int_{0}^{t} \MyEphipsi{12}{\tau} \; \dd\tau \\
&\qquad + \int_{0}^{t} \scpBig{\Ltwo}{\tfrac{1}{\alphaphi(\tau)} \, \rphi(\tau)}{\Delta \partial_{tt} \psi(\tau)} \; \dd\tau 
+ R(t)\,,
\end{split}
\end{equation*}
where we employ the convenient abbreviation 
\begin{equation*}
\begin{split}
R(t) &= \bba_2(\sigma_0) \, \bb_5(\sigma) 
\int_{0}^{t} \scpBig{\Ltwo}{\tfrac{1}{(\alphaphi(\tau))^2}}{\Delta \partial_{tt} \psi(\tau) \, \nabla \Delta \partial_{t} \psi(\tau) \cdot \nabla \partial_{t} \phi(\tau)} \; \dd\tau \\
&\qquad - \, \tfrac{\bba_2(\sigma_0) \bb_5(\sigma)}{2} \int_{0}^{t} \scpBig{\Ltwo}{\tfrac{1}{(\alphaphi(\tau))^2}}{\partial_{tt} \phi(\tau) \, \absbig{\nabla \Delta \partial_{t} \psi(\tau)}^2} \; \dd\tau \\
&\qquad
- \tfrac{\bb_3 \bb_5(\sigma)}{2} \int_{0}^{t} \scpBig{\Ltwo}{\tfrac{1}{(\alphaphi(\tau))^2}}{\partial_{tt} \phi(\tau) \, \big(\Delta \partial_{t} \psi(\tau)\big)^2} \; \dd\tau \\
&\qquad + \bba_4(\sigma_0) \, \bb_5(\sigma) \int_{0}^{t} \scpBig{\Ltwo}{\tfrac{1}{(\alphaphi(\tau))^2}}{\Delta \partial_{tt} \psi(\tau) \, \nabla \partial_{t} \phi(\tau) \cdot \nabla \Delta \psi(\tau)} \; \dd\tau \\
&\qquad - \, \bba_4(\sigma_0) \, \bb_5(\sigma) \int_{0}^{t} \scpBig{\Ltwo}{\tfrac{1}{(\alphaphi(\tau))^2}}{\partial_{tt} \phi(\tau) \, \nabla \Delta \partial_{t} \psi(\tau) \cdot \nabla \Delta \psi(\tau)} \; \dd\tau\,,
\end{split}
\end{equation*}
see also~\eqref{eq:GeneralModelReformulationTest2}; 
similar arguments to before permit to estimate and absorb the arising terms. 
\begin{enumerate}[(i)]
\item 
The application of Cauchy's inequality, Young's inequality with 
\begin{equation*}
\gamma_1^2 = \tfrac{\bba_2(\sigma_0)}{2 \, \bba_4(\sigma_0)} = \tfrac{\bba_0(\sigma_0)}{2}\,,
\end{equation*}  
and the uniform bound $\tfrac{1}{\underline{\alpha}} = 2$, see~\eqref{eq:BoundednessAlpha}, yields 
\begin{equation*}
\begin{split}
&\bba_4(\sigma_0) \, \absBig{\scpbig{\Ltwo}{\tfrac{1}{\alphaphi(0)}}{\nabla \Delta \psi_1 \cdot \nabla \Delta \psi_0}} \\
&\quad \leq \bba_4(\sigma_0) \, \normBig{\Ltwo}{\tfrac{1}{\sqrt{\alphaphi(0)}} \, \nabla \Delta \psi_1} \,  
\normBig{\Ltwo}{\tfrac{1}{\sqrt{\alphaphi(0)}} \, \nabla \Delta \psi_0} \\
&\quad \leq \tfrac{\gamma_1^2 \bba_4(\sigma_0)}{2} \, \normBig{\Ltwo}{\tfrac{1}{\sqrt{\alphaphi(0)}} \, \nabla \Delta \psi_1}^2
+ \tfrac{\bba_4(\sigma_0)}{2 \, \gamma_1^2} \, \normBig{\Ltwo}{\tfrac{1}{\sqrt{\alphaphi(0)}} \, \nabla \Delta \psi_0}^2 \\
&\quad \leq \tfrac{\bba_2(\sigma_0)}{4} \, \MyEphipsiAtzero{12} 
+ \tfrac{\bba_4(\sigma_0)}{\bba_0(\sigma_0)} \, \tfrac{1}{\underline{\alpha}} \normbig{\Ltwo}{\nabla \Delta \psi_0}^2 \\
&\quad \leq \MyEphipsiAtzero{1} 
+ \tfrac{2 \, \bba_4(\sigma_0)}{\bba_0(\sigma_0)} \, \normbig{\Ltwo}{\nabla \Delta \psi_0}^2\,.
\end{split}
\end{equation*}
\item 
Using in addition the Gronwall-type inequality~\eqref{eq:Gronwall} with $\varphi = \nabla \Delta \psi$ and the uniform bound $\overline{\alpha} = \tfrac{3}{2}$, see again~\eqref{eq:BoundednessAlpha}, we obtain 
\begin{equation*}
\begin{split}
&\bba_4(\sigma_0) \, \absBig{\scpbig{\Ltwo}{\tfrac{1}{\alphaphi(t)}}{\nabla \Delta \partial_{t} \psi(t) \cdot \nabla \Delta \psi(t)}} \\
&\quad \leq \bba_4(\sigma_0) \, \normBig{\Ltwo}{\tfrac{1}{\sqrt{\alphaphi(t)}} \, \nabla \Delta \partial_{t} \psi(t)} \, 
\normBig{\Ltwo}{\tfrac{1}{\sqrt{\alphaphi(t)}} \, \nabla \Delta \psi(t)} \\
&\quad \leq \tfrac{\gamma_1^2 \bba_4(\sigma_0)}{2} \, \normBig{\Ltwo}{\tfrac{1}{\sqrt{\alphaphi(t)}} \, \nabla \Delta \partial_{t} \psi(t)}^2 
+ \tfrac{\bba_4(\sigma_0)}{2 \, \gamma_1^2} \, \normBig{\Ltwo}{\tfrac{1}{\sqrt{\alphaphi(t)}} \, \nabla \Delta \psi(t)}^2 \\
&\quad \leq \tfrac{\bba_2(\sigma_0)}{4} \, \MyEphipsi{12}{t} 
+ \tfrac{\bba_2(\sigma_0)}{(\bba_0(\sigma_0))^2} \, \tfrac{1}{\underline{\alpha}} \, \normbig{\Ltwo}{\nabla \Delta \psi(t)}^2 \\
&\quad \leq \tfrac{\bba_2(\sigma_0)}{4} \, \MyEphipsi{12}{t} \\
&\quad\qquad + \tfrac{6 \, \bba_2(\sigma_0)}{(\bba_0(\sigma_0))^2} \, \Bigg(\normbig{\Ltwo}{\nabla \Delta \psi_0}^2 
+ T \int_{0}^{t} \normbig{\Ltwo}{\nabla \Delta \partial_{t} \psi(\tau)}^2 \; \dd \tau\Bigg) \\
&\quad \leq \tfrac{\bba_2(\sigma_0)}{4} \, \MyEphipsi{12}{t} \\
&\quad\qquad + \tfrac{6 \, \bba_2(\sigma_0)}{(\bba_0(\sigma_0))^2} \, \Bigg(\normbig{\Ltwo}{\nabla \Delta \psi_0}^2 
+ \overline{\alpha} \, T \int_{0}^{t} \MyEphipsi{12}{\tau} \; \dd \tau\Bigg) \\
&\quad \leq \tfrac{\bba_2(\sigma_0)}{4} \, \MyEphipsi{12}{t} \\
&\quad\qquad + \tfrac{6 \, \bba_2(\sigma_0)}{(\bba_0(\sigma_0))^2} \, \normbig{\Ltwo}{\nabla \Delta \psi_0}^2 
+ \tfrac{9 \, \bba_2(\sigma_0)}{(\bba_0(\sigma_0))^2} \, T \int_{0}^{t} \MyEphipsi{12}{\tau} \; \dd \tau \\
&\quad \leq \tfrac{\bba_2(\sigma_0)}{4} \, \MyEphipsi{12}{t} \\
&\quad\qquad + \tfrac{6 \, \bba_4(\sigma_0)}{\bba_0(\sigma_0)} \, \normbig{\Ltwo}{\nabla \Delta \psi_0}^2 
+ \tfrac{36}{(\bba_0(\sigma_0))^2} \, T \int_{0}^{t} \MyEphipsi{1}{\tau} \; \dd \tau\,;
\end{split}
\end{equation*}
this shows that the first term on the right-hand side can be absorbed and explains the definition of the energy functional 
\begin{equation*}
\MyEphipsi{1}{t} = \MywEphipsi{1}{t} - \tfrac{\bba_2(\sigma_0)}{4} \, \MyEphipsi{12}{t}\,.
\end{equation*}
\item 
Recalling once more the abbreviation $\bba_0(\sigma_0) = \frac{\bba_2(\sigma_0)}{\bba_4(\sigma_0)}$, the bound 
\begin{equation*}
\bba_4(\sigma_0) \int_{0}^{t} \MyEphipsi{12}{\tau} \; \dd\tau
\leq \tfrac{4}{\bba_0(\sigma_0)} \int_{0}^{t} \MyEphipsi{1}{\tau} \; \dd\tau
\end{equation*}
is obvious.
\item 
By Cauchy's inequality, Young's inequality with weight $\gamma_2^2 = \bba_1$, and the upper bound $\tfrac{1}{\underline{\alpha}} = 2$, we have 
\begin{equation*}
\begin{split}
&\int_{0}^{t} \scpbig{\Ltwo}{\tfrac{1}{\alphaphi(\tau)} \, \rphi(\tau)}{\Delta \partial_{tt} \psi(\tau)} \; \dd\tau \\
&\quad \leq \int_{0}^{t} \normBig{\Ltwo}{\tfrac{1}{\sqrt{\alphaphi(\tau)}} \, \rphi(\tau)} \, 
\normBig{\Ltwo}{\tfrac{1}{\sqrt{\alphaphi(\tau)}} \, \Delta \partial_{tt} \psi(\tau)} \; \dd\tau \\
&\quad \leq \tfrac{1}{2 \, \gamma_2^2} \, \int_{0}^{t} \normBig{\Ltwo}{\tfrac{1}{\sqrt{\alphaphi(\tau)}} \, \rphi(\tau)}^2 \; \dd\tau 
+ \tfrac{\gamma_2^2}{2} \int_{0}^{t} \normBig{\Ltwo}{\tfrac{1}{\sqrt{\alphaphi(\tau)}} \, \Delta \partial_{tt} \psi(\tau)}^2 \; \dd\tau \\
&\quad \leq \tfrac{1}{\bba_1} \int_{0}^{t} \normbig{\Ltwo}{\rphi(\tau)}^2 \; \dd\tau 
+ \tfrac{1}{2} \int_{0}^{t} \MywEphipsi{2}{\tau} \; \dd\tau\,;
\end{split}
\end{equation*}
together with estimate~\eqref{eq:EstimateNonlinearity} for the nonlinearity, this implies 
\begin{equation*}
\begin{split}
&\int_{0}^{t} \scpbig{\Ltwo}{\tfrac{1}{\alphaphi(\tau)} \, \rphi(\tau)}{\Delta \partial_{tt} \psi(\tau)} \; \dd\tau \\
&\quad \leq \tfrac{C_1}{\bba_1} \, \BoundE{1} \int_{0}^{t} \MyEphipsi{1}{\tau} \; \dd\tau \\
&\quad\qquad + \bigg(\tfrac{1}{2} + \tfrac{C_2}{\bba_1} \, \Big(\normbig{\Ltwo}{\Delta \psi_0}^2 + C_3 \, T^2 \, \BoundE{1}\Big)\Bigg) 
\int_{0}^{t} \MywEphipsi{2}{\tau} \; \dd\tau\,.
\end{split}
\end{equation*}
Under the additional smallness requirement 
\begin{equation*}
\tfrac{C_2}{\bba_1} \, \Big(\normbig{\Ltwo}{\Delta \psi_0}^2 + C_3 \, T^2 \, \BoundE{1}\Big) \leq \tfrac{1}{4}\,, 
\end{equation*}
we obtain the relation 
\begin{equation*}
\begin{split}
&\int_{0}^{t} \scpbig{\Ltwo}{\tfrac{1}{\alphaphi(\tau)} \, \rphi(\tau)}{\Delta \partial_{tt} \psi(\tau)} \; \dd\tau \\
&\quad \leq \tfrac{C_1}{\bba_1} \, \BoundE{1} \int_{0}^{t} \MyEphipsi{1}{\tau} \; \dd\tau 
+ \tfrac{3}{4} \int_{0}^{t} \MywEphipsi{2}{\tau} \; \dd\tau\,;
\end{split}
\end{equation*}
thus, the second term involving~$\wE_2$ can be absorbed and yields the integral over~$E_2$. 
\end{enumerate}
As an intermediate result, we have a bound of the form 
\begin{equation}
\label{eq:Intermediate1}
\begin{split}
&\MyEphipsi{1}{t} + \int_{0}^{t} \MyEphipsi{2}{\tau} \; \dd\tau \\
&\quad \leq \Phi_1\Big(C_{\text{PF}}, C_{\Lfour \leftarrow \Hone}, \tfrac{1}{\bba_0(\sigma_0)}, \tfrac{1}{\bba_1}, \bba_2(\sigma_0), \\ &\quad\qquad\qquad \tfrac{1}{\bb_3}, \bba_4(\sigma_0), \bb_5(\sigma), \bb_6(\sigma), T, \BoundE{1}\Big) \\ 
&\quad\qquad \times \bigg(\MyEphipsiAtzero{1} + \normbig{\Ltwo}{\nabla \Delta \psi_0}^2 + \int_{0}^{t} \MyEphipsi{1}{\tau} \; \dd\tau\bigg)
+ \absbig{R(t)}\,.
\end{split}
\end{equation}
The remaining terms are estimated with the help of Cauchy's inequality and~\eqref{eq:PoincareEmbeddingEllipticity}, that is, we use that a product of functions satisfies the relation 
\begin{equation*}
\begin{split}
&\absBig{\scpbig{\Ltwo}{\varphi_1(\tau) \, \varphi_2(\tau)}{\varphi_3(\tau)}}
\leq \normbig{\Ltwo}{\varphi_1(\tau) \, \varphi_2(\tau)} \, \normbig{\Ltwo}{\varphi_3(\tau)} \\
&\quad \leq \normbig{\Linfty}{\varphi_1(\tau)} \, \normbig{\Ltwo}{\varphi_2(\tau)} \, \normbig{\Ltwo}{\varphi_3(\tau)} 
\leq C_{\Linfty \leftarrow \Htwo} \, \normbig{\Htwo}{\varphi_1(\tau)} \, \normbig{\Ltwo}{\varphi_2(\tau)} \, \normbig{\Ltwo}{\varphi_3(\tau)} \\
&\quad \leq C_{\Delta} \, C_{\Linfty \leftarrow \Htwo} \, \normbig{\Ltwo}{\Delta \varphi_1(\tau)} \, \normbig{\Ltwo}{\varphi_2(\tau)} \, \normbig{\Ltwo}{\varphi_3(\tau)}\,.
\end{split}
\end{equation*}
As a consequence, by~\eqref{eq:E0E1E2}, inserting again $\tfrac{1}{\underline{\alpha}} = 2$, we obtain
\begin{equation*}
\begin{split}
&\absbig{R(t)} \\
&\quad\leq 2 \, \bba_2(\sigma_0) \, \bb_5(\sigma) 
\int_{0}^{t} \sqrt{\MyEphipsi{20}{\tau}} \, \sqrt{\MyEphipsi{12}{\tau}} \, \normbig{\Linfty}{\nabla \partial_{t} \phi(\tau)} \; \dd\tau \\
&\qquad + \bba_2(\sigma_0) \, \bb_5(\sigma) \int_{0}^{t} \normbig{\Linfty}{\partial_{tt} \phi(\tau)} \, \MyEphipsi{12}{\tau} \; \dd\tau \\
&\qquad + \bb_3 \, \bb_5(\sigma) \int_{0}^{t} \normbig{\Linfty}{\partial_{tt} \phi(\tau)} \, \MyEphipsi{13}{\tau} \; \dd\tau \\
&\qquad + 2 \sqrt{2} \, \bba_4(\sigma_0) \, \bb_5(\sigma) 
\int_{0}^{t} \sqrt{\MyEphipsi{20}{\tau}} \, \normbig{\Linfty}{\nabla \partial_{t} \phi(\tau)} \, \normbig{\Ltwo}{\nabla \Delta \psi(\tau)} \; \dd\tau \\
&\qquad + 2 \sqrt{2} \, \bba_4(\sigma_0) \, \bb_5(\sigma) 
\int_{0}^{t} \normbig{\Linfty}{\partial_{tt} \phi(\tau)} \, \sqrt{\MyEphipsi{12}{\tau}} \, \normbig{\Ltwo}{\nabla \Delta \psi(\tau)} \; \dd\tau\,. 
\end{split}
\end{equation*}
Recalling the upper bound $\overline{\alpha} = \tfrac{3}{2}$, we employ the estimates 
\begin{equation*}
\begin{split}
&\normbig{\Linfty}{\nabla \partial_{t} \phi(\tau)} 
\leq C_{\Delta} \, C_{\Linfty \leftarrow \Htwo} \, \normbig{\Ltwo}{\nabla \Delta \partial_{t} \phi(\tau)} \\
&\quad \leq \sqrt{\tfrac{3}{2}} \, C_{\Delta} \, C_{\Linfty \leftarrow \Htwo} \, \sqrt{\MyEphiphi{12}{\tau}}\,, \\
&\normbig{\Linfty}{\partial_{tt} \phi(\tau)} 
\leq C_{\Delta} \, C_{\Linfty \leftarrow \Htwo} \, \normbig{\Ltwo}{\Delta \partial_{tt} \phi(\tau)}\\
&\quad \leq \sqrt{\tfrac{3}{2}} \, C_{\Delta} \, C_{\Linfty \leftarrow \Htwo} \, \sqrt{\MyEphiphi{20}{\tau}}\,;
\end{split}
\end{equation*} 
moreover, the Gronwall-type inequality~\eqref{eq:Gronwall} applied with~$\varphi = \nabla \Delta \psi$ and the elementary relation $\sqrt{x^2 + y^2} \leq x + y$, 
valid for positive real numbers $x, y > 0$, implies 
\begin{equation*}
\begin{split}
&\normbig{\Ltwo}{\nabla \Delta \psi(\tau)}^2
\leq 3 \, \normbig{\Ltwo}{\nabla \Delta \psi_0}^2 
+ 3 \, T \int_{0}^{\tau} \normbig{\Ltwo}{\nabla \Delta \partial_{t} \psi(\widetilde{\tau})}^2 \; \dd \widetilde{\tau} \\ 
&\quad \leq 3 \, \normbig{\Ltwo}{\nabla \Delta \psi_0}^2 
+ 3 \, T \, \overline{\alpha} \int_{0}^{\tau} \MyEphipsi{12}{\widetilde{\tau}} \; \dd \widetilde{\tau}\,, \\
&\normbig{\Ltwo}{\nabla \Delta \psi(\tau)}
\leq \sqrt{3} \, \normbig{\Ltwo}{\nabla \Delta \psi_0} 
+ \tfrac{3}{\sqrt{2}} \, \sqrt{T} \, \sqrt{\int_{0}^{\tau} \MyEphipsi{12}{\widetilde{\tau}} \; \dd \widetilde{\tau}}\,.
\end{split}
\end{equation*}
Introducing the auxiliary abbreviation 
\begin{equation}
\label{eq:C4} 
C_4 = C_{\Delta} \, C_{\Linfty \leftarrow \Htwo} \, \tfrac{\bb_5(\sigma)}{\sqrt{\bba_1}} \, 
\max\bigg\{8 \sqrt{6}\,, \tfrac{24 \sqrt{\bba_2(\sigma_0)}}{\bba_0(\sigma_0)}\,, \tfrac{24 \sqrt{6}}{\bba_0(\sigma_0)} \, \sqrt{T} 
\bigg\}\,, 
\end{equation}
as well as 
\begin{equation*}
\begin{gathered}
R_1(t) = \int_{0}^{t} \MyEphipsi{1}{\tau} \, \sqrt{\MyEphiphi{2}{\tau}} \; \dd\tau\,, \\
R_2(t) = \int_{0}^{t} \sqrt{\MyEphipsi{1}{\tau}} \, \sqrt{\int_{0}^{\tau} \MyEphipsi{1}{\widetilde{\tau}} \; \dd \widetilde{\tau}} \sqrt{\MyEphiphi{2}{\tau}} \; \dd\tau\,, 
\end{gathered}
\end{equation*}
this leads to the relation 
\begin{equation*}
\begin{split}
\absbig{R(t)} 
&\leq C_4 \\ 
&\quad \times \bigg(\int_{0}^{t} \sqrt{\MyEphipsi{1}{\tau}} \, \sqrt{\MyEphipsi{2}{\tau}} \, \sqrt{\MyEphiphi{1}{\tau}} \; \dd\tau \\
&\quad\qquad + \normbig{\Ltwo}{\nabla \Delta \psi_0} \int_{0}^{t} \sqrt{\MyEphipsi{2}{\tau}} \, \sqrt{\MyEphiphi{1}{\tau}} \; \dd\tau \\
&\quad\qquad + \normbig{\Ltwo}{\nabla \Delta \psi_0} \int_{0}^{t} \sqrt{\MyEphipsi{1}{\tau}} \, \sqrt{\MyEphiphi{2}{\tau}} \; \dd\tau \\
&\quad\qquad + \int_{0}^{t} \sqrt{\MyEphipsi{2}{\tau}} \, \sqrt{\int_{0}^{\tau} \MyEphipsi{1}{\widetilde{\tau}} \; \dd \widetilde{\tau}} \, 
\sqrt{\MyEphiphi{1}{\tau}} \; \dd\tau \\
&\quad\qquad + R_1(t) + R_2(t)\bigg)\,.
\end{split}
\end{equation*}
We next make use of the fundamental assumption 
\begin{equation*}
\sup_{t \in [0, T]} \MyEphiphi{1}{t} \leq \BoundE{1}\,, \quad 
\int_{0}^{T} \MyEphiphi{2}{t} \; \dd t \leq \BoundE{2}\,,
\end{equation*}
see also~\eqref{eq:E0E1E2}. 
Replacing the interval of integration~$[0, \tau]$ by $[0, t]$ and applying Cauchy's inequality, we have 
\begin{equation*}
\begin{split}
R_2(t) 
&\leq \sqrt{\int_{0}^{t} \MyEphipsi{1}{\widetilde{\tau}} \; \dd \widetilde{\tau}} \int_{0}^{t} \sqrt{\MyEphipsi{1}{\tau}} \, \sqrt{\MyEphiphi{2}{\tau}} \; \dd\tau \\
&\leq \sqrt{\int_{0}^{t} \MyEphipsi{1}{\widetilde{\tau}} \; \dd \widetilde{\tau}} \, \sqrt{\int_{0}^{t} \MyEphipsi{1}{\tau} \; \dd\tau} \, \sqrt{\int_{0}^{t} \MyEphiphi{2}{\tau} \; \dd\tau} \\
&\leq \sqrt{\BoundE{2}} \int_{0}^{t} \MyEphipsi{1}{\tau} \; \dd\tau\,;
\end{split}
\end{equation*}
together with Young's inequality, this shows 
\begin{equation*}
\begin{split}
\absbig{R(t)} 
&\leq C_4 \\ 
&\quad \times \Bigg(\tfrac{1}{2} \, \sqrt{\BoundE{1}} \int_{0}^{t} \Big(\MyEphipsi{1}{\tau} + \MyEphipsi{2}{\tau}\Big) \; \dd\tau \\
&\quad\qquad + \tfrac{1}{2} \, \normbig{\Ltwo}{\nabla \Delta \psi_0} \bigg(T \, \BoundE{1} + \int_{0}^{t} \MyEphipsi{2}{\tau} \; \dd\tau\bigg) \\
&\quad\qquad + \tfrac{1}{2} \, \normbig{\Ltwo}{\nabla \Delta \psi_0} \bigg(\BoundE{2} + \int_{0}^{t} \MyEphipsi{1}{\tau} \; \dd\tau\bigg) \\
&\quad\qquad + \tfrac{1}{2} \, \sqrt{\BoundE{1}} \bigg(\int_{0}^{t} \MyEphipsi{2}{\tau} \; \dd\tau 
+ T \int_{0}^{t} \MyEphipsi{1}{\tau} \; \dd\tau\bigg) \\
&\quad\qquad + R_1(t) + R_2(t)\Bigg) \\
&\leq C_4 \\ 
&\quad \times \Bigg(\tfrac{1}{2} \, \normbig{\Ltwo}{\nabla \Delta \psi_0} \, \big(T \, \BoundE{1} + \BoundE{2}\big) \\
&\quad\qquad + \tfrac{1}{2} \, \Big(\normbig{\Ltwo}{\nabla \Delta \psi_0} + (1 + T) \, \sqrt{\BoundE{1}} + \sqrt{\BoundE{2}} \Big) \int_{0}^{t} \MyEphipsi{1}{\tau} \; \dd\tau\bigg)\\
&\quad\qquad + \Big(\tfrac{1}{2} \, \normbig{\Ltwo}{\nabla \Delta \psi_0} + \sqrt{\BoundE{1}}\Big) \int_{0}^{t} \MyEphipsi{2}{\tau} \; \dd\tau \\
&\quad\qquad + R_1(t)\Bigg)\,.
\end{split}
\end{equation*}
Under the smallness requirement 
\begin{equation*}
C_4 \, \Big(\tfrac{1}{2} \, \normbig{\Ltwo}{\nabla \Delta \psi_0} + \sqrt{\BoundE{1}}\Big) \leq \tfrac{1}{2}\,, 
\end{equation*}
the term involving~$E_2$ can be absorbed and we arrive at an estimate of the form 
\begin{equation}
\begin{split}
&\MyEphipsi{1}{t} + \int_{0}^{t} \MyEphipsi{2}{\tau} \; \dd\tau \\
&\quad \leq \MyEphipsiAtzero{1} \\
&\qquad + \Phi_2\Big(C_{\text{PF}}, C_{\Delta}, C_{\Lfour \leftarrow \Hone}, C_{\Linfty \leftarrow \Htwo}, \tfrac{1}{\bba_0(\sigma_0)}, \tfrac{1}{\bba_1}, \bba_2(\sigma_0), \tfrac{1}{\bb_3}, \\
&\quad\qquad\qquad \bba_4(\sigma_0), \bb_5(\sigma), \bb_6(\sigma), T, \normbig{\Ltwo}{\nabla \Delta \psi_0}, \BoundE{1}, \BoundE{2}\Big) \\ 
&\quad\qquad \times \bigg(\MyEphipsiAtzero{1} + \normbig{\Ltwo}{\nabla \Delta \psi_0} \\
&\quad\qquad\qquad + \int_{0}^{t} \Big(1 + \sqrt{\MyEphiphi{2}{\tau}}\Big) \, \MyEphipsi{1}{\tau} \; \dd\tau\bigg)\,, 
\end{split}
\end{equation}
see also~\eqref{eq:C4}.
This corresponds to the relation 
\begin{equation*}
\MyEphipsi{1}{t} 
\leq \MyEphipsiAtzero{1} + \Phi_2 \, \delta + \Phi_2 \int_{0}^{t} \omega(\tau) \, \MyEphipsi{1}{\tau} \; \dd\tau   
\end{equation*}
involving a (small) constant $\delta > 0$ and the weight function 
\begin{equation*}
\omega(t) = 1 + \sqrt{\MyEphiphi{2}{t}}\,;
\end{equation*}
due to the fact that Cauchy's inequality ensures boundedness from above 
\begin{equation*}
\int_{0}^{t} \omega(\tau) \; \dd\tau \leq \overline{\omega} = T + \sqrt{T} \, \sqrt{\BoundE{2}}
\end{equation*}
and that the solution to the associated non-autonomous homogeneous linear differential equation fulfills 
\begin{equation*}
f'(t) = \Phi_2 \, \omega(t) \, f(t)\,, \qquad 
f(t) = \exp\bigg(\Phi_2 \int_{0}^{t} \omega(\tau) \; \dd\tau\bigg) \, f(0) \leq \ee^{\Phi_2 \, \overline{\omega}} \, f(0)\,,
\end{equation*}
a Gronwall-type inequality leads to an upper bound of the form 
\begin{equation*}
\MyEphipsi{1}{t} \leq \Phi_3 \big(\MyEphipsiAtzero{1} + \delta\big)\,, \quad t \in [0,T]\,.
\end{equation*}
More precisely, we obtain an energy estimate of the form  
\begin{equation}
\label{eq:EnergyEstimate}
\begin{split}
&\MyEphipsi{0}{t} + \MyEphipsi{1}{t} + \int_{0}^{t} \MyEphipsi{2}{\tau} \; \dd\tau \\
&\quad \leq \Phi\Big(C_{\text{PF}}, C_{\Delta}, C_{\Lfour \leftarrow \Hone}, C_{\Linfty \leftarrow \Htwo}, \tfrac{1}{\bba_0(\sigma_0)}, \tfrac{1}{\bba_1}, \bba_2(\sigma_0), \tfrac{1}{\bb_3}, \\
&\qquad\qquad \bba_4(\sigma_0), \bb_5(\sigma), \bb_6(\sigma), T, \MyEphipsiAtzero{1}, \normbig{\Ltwo}{\nabla \Delta \psi_0}, \BoundE{1}, \BoundE{2}\Big) \\ 
&\quad\qquad \times \Big(\MyEphipsiAtzero{1} + \normbig{\Ltwo}{\nabla \Delta \psi_0}\Big)\,; 
\end{split}
\end{equation}
due to the fact that the quantities $\frac{1}{\bba_0(\sigma_0)}$, $\frac{1}{\bba_1}$, and $\bba_4(\sigma_0)$ remain bounded for $a \to 0_{+}$, this relation holds uniformly for $a \in [0, \overline{a}]$.
A fixed-point argument detailed below proves the following statement;
uniqueness of the solution is provided in the situation of Remark~\ref{Remark1}. 

\BEMERKUNGBARBARA{
\paragraph{Difference energy estimate.}
We finally provide a lower order energy estimate for the difference $\psi^1-\psi^2$ between solutions $\psi^i$ to ~\eqref{eq:GeneralModelReformulationDifferentiation1} with~$\alpha$ and~$r$ substituted by~$\alphaphi$ and~$\rphi$ with $\phi=\phi^i$, $i\in\{1,2\}$. 
This difference satisfies a similar equation --- actually the linear part remains the same. 
Imposing the same initial data for $\psi^1,\psi^2,\phi^1,\phi^2$, and using the abbreviations $\hat{\psi}=\psi^1-\psi^2$, $\hat{\phi}=\phi^1-\phi^2$, $\alpha^1=(1+\bb_5(\sigma) \, \partial_{t} \phi^1)$, after integration with respect to time, we can write this equation as follows.
\begin{equation*}
\begin{split}
&\alpha^1 \partial_{tt} \hat{\psi}(t) 
- \bba_1 \, \Delta \partial_{t} \hat{\psi}(t) 
+ \bba_2(\sigma_0) \, \Delta^2 \hat{\psi}(t) 
- \bb_3 \, \Delta \hat{\psi}(t)\\
&\qquad + \bba_4(\sigma_0) \, \int_{0}^{t} \Delta^2 \hat{\psi}(\tau) \; \dd\tau 
+\hat{r}(t)=0 \,,
\end{split}
\end{equation*}
where 
\begin{equation*}
\hat{r}(t)=\bb_5(\sigma) \, \partial_{t} \hat{\phi} \partial_{tt} \psi^2(t)
+ 2 \, \bb_6(\sigma) \, \nabla \phi^1(t) \cdot \nabla \partial_{t} \hat{\psi}(t) 
+ 2 \, \bb_6(\sigma) \, \nabla \hat{\phi}(t) \cdot \nabla \partial_{t} \psi^2(t)\,.
\end{equation*}
Testing with $\partial_{t}\hat{\psi}(t)$ and integrating by parts, using the fact that the boundary terms vanish and that $\alpha^1$ is bounded from below by $\underline{\alpha}>0$, we obtain the identity
\begin{equation*}
\begin{split}
&\tfrac{1}{2} \, \partial_{t} \normBig{\Ltwo}{\sqrt{\alpha^1(t)} \, \partial_{t} \hat{\psi}(t)}^2
+ \bba_1 \, \normBig{\Ltwo}{\nabla \partial_{t} \hat{\psi}(t)}^2 
+ \bba_2(\sigma_0) \, \tfrac{1}{2} \, \partial_{t} \normBig{\Ltwo}{\Delta \hat{\psi}(t)}^2\\
&\qquad  
+ \bb_3 \, \tfrac{1}{2} \, \partial_{t} \normBig{\Ltwo}{\nabla \hat{\psi}(t)}^2
+ \bba_4(\sigma_0) \, \scpBig{\Ltwo}{\int_{0}^{t} \Delta\hat{\psi}(\tau) \; \dd\tau}{\Delta\partial_{t}\hat{\psi}(t)}\\
&\qquad 
+\scpBig{\Ltwo}{\hat{r}(t)-\tfrac12 \partial_{t}\alpha^1(t)\partial_{t}\hat{\psi}(t)}{\partial_{t}\hat{\psi}(t)}
=0 \,.
\end{split}
\end{equation*}
Integrating with respect to time and using the fact that the terms evaluated at initial time vanish, we get
\begin{equation*}
\begin{split}
&\tfrac{1}{2} \, \normBig{\Ltwo}{\sqrt{\alpha^1(t)} \, \partial_{t} \hat{\psi}(t)}^2
+ \bba_1 \, \int_0^t\normBig{\Ltwo}{\nabla \partial_{t} \hat{\psi}(\tau)}^2 \; \dd\tau
+ \tfrac{\bba_2(\sigma_0)}{2} \, \normBig{\Ltwo}{\Delta \hat{\psi}(t)}^2\\
&\qquad  
+ \frac{\bb_3}{2} \, \normBig{\Ltwo}{\nabla \hat{\psi}(t)}^2
=- \bba_4(\sigma_0) \, \int_0^t\scpBig{\Ltwo}{\int_{0}^{\tau} \Delta\hat{\psi}(\widetilde{\tau}) \; \dd\widetilde{\tau}}{\Delta\partial_{\tau}\hat{\psi}(\tau)}\; \dd\tau\\
&\qquad -\int_0^t\scpBig{\Ltwo}{\hat{r}(\tau)-\tfrac12 \partial_{t}\alpha^1(\tau)\partial_{t}\hat{\psi}\tau)}{\partial_{t}\hat{\psi}(\tau)}\; \dd\tau
=0\,,
\end{split}
\end{equation*}
where we can estimate as follows.
\begin{equation*}
\begin{split}
&-\bba_4(\sigma_0) \, \int_0^t\scpBig{\Ltwo}{\int_{0}^{\tau} \Delta\hat{\psi}(\widetilde{\tau}) \; \dd\widetilde{\tau}}{\Delta\partial_{\tau}\hat{\psi}(\tau)}\; \dd\tau\\
&=\bba_4(\sigma_0) \, \int_0^t\normBig{\Ltwo}{\Delta\hat{\psi}(\tau)}^2\; \dd\tau
- \bba_4(\sigma_0) \, \scpBig{\Ltwo}{\int_{0}^{t} \Delta\hat{\psi}(\tau) \; \dd\tau}{\Delta\hat{\psi}(t)}\\
&\qquad \leq \bba_4(\sigma_0)\int_0^t\normbig{\Ltwo}{\Delta\hat{\psi}(\tau)}^2\; \dd\tau
+\tfrac{\bba_2(\sigma_0)}{4} \normbig{\Ltwo}{\Delta\hat{\psi}(t)}^2
+\tfrac{\bba_4(\sigma_0)^2}{\bba_2(\sigma_0)} T \int_0^t \normbig{\Ltwo}{\Delta\hat{\psi}(\tau)}^2 \; \dd\tau
\end{split}
\end{equation*}
and 
\begin{equation*}
\begin{split}
&\absBig{\int_0^t\scpBig{\Ltwo}{\hat{r}(\tau)-\tfrac12 \partial_{t}\alpha^1(\tau)\partial_{t}\hat{\psi}(\tau)}{\partial_{t}\hat{\psi}(\tau)}\; \dd\tau}\\
&\leq 
C_{\text{PF}} \, C_{\Lfour \leftarrow \Hone}
\int_0^t\normBig{\Lfouroverthree}{\hat{r}(\tau)-\tfrac12 \partial_{t}\alpha^1(\tau)\partial_{t}\hat{\psi}(\tau)}\normBig{\Ltwo}{\nabla\partial_{t}\hat{\psi}(\tau)}\; \dd\tau\\
&\leq
\tfrac{\bba_1}{2} \int_0^t \normBig{\Ltwo}{\nabla\partial_{t}\hat{\psi}(\tau)}^2\; \dd\tau
+\tfrac{C_{\text{PF}}^2 \, C_{\Lfour \leftarrow \Hone}^2}{2\bba_1}
\int_0^t\normBig{\Lfouroverthree}{\hat{r}(\tau)-\tfrac12\partial_{t}\alpha^1(\tau)\partial_{t}\hat{\psi}(\tau)}^2
\; \dd\tau
\end{split}
\end{equation*}
where by H\"older's inequality with exponents $\frac32,3$,
\begin{equation*}
\begin{split}
&\int_0^t\normBig{\Lfouroverthree}{\hat{r}(\tau)-\tfrac12\partial_{t}\alpha^1(\tau)\partial_{t}\hat{\psi}(\tau)}^2 \, \dd\tau\\
&\leq
\int_0^t\Bigl(
2\bb_5(\sigma) \, \normBig{\Ltwo}{\partial_{t} \hat{\psi}(\tau)}^2 
\normBig{\Lfour}{\partial_{tt} \phi^1(\tau)}^2
+4\bb_5(\sigma) \, \normBig{\Ltwo}{\partial_{t} \hat{\phi}(\tau)}^2 
\normBig{\Lfour}{\partial_{tt} \psi^2(\tau)}^2\, \\
&\qquad
+ 8 \, \bb_6(\sigma) \, \normBig{\Ltwo}{\nabla \partial_{t} \hat{\psi}(\tau)}^2 
\normBig{\Lfour}{\nabla \phi^1(\tau)}^2
+ 8 \, \bb_6(\sigma) \, \normBig{\Ltwo}{\nabla \hat{\phi}(\tau)}^2 
\normBig{\Lfour}{\nabla \partial_{t} \psi^2(\tau)}^2\Bigr)\, \dd\tau\\
&\leq
2\bb_5(\sigma) \, C_{\text{PF}}^2 \, C_{\Lfour \leftarrow \Hone}^2\, 
\sup_{\tau\in[0,t]}\normBig{\Ltwo}{\nabla\partial_{tt} \psi^1(\tau)}^2\,
\int_0^t\normBig{\Ltwo}{\partial_{t} \hat{\psi}(\tau)}^2 \\
&\quad+4\bb_5(\sigma) \, C_{\text{PF}}^2 \, C_{\Lfour \leftarrow \Hone}^2\, 
\sup_{\tau\in[0,t]} \normBig{\Ltwo}{\nabla\partial_{tt} \psi^2(\tau)}^2\, 
\int_0^t\normBig{\Ltwo}{\partial_{t} \hat{\phi}(\tau)}^2 \, \dd\tau
\\
&\quad+ 8 \, \bb_6(\sigma)^2 \, C_{\Delta}^2 \, C_{\Lfour \leftarrow \Hone}^2\, 
\sup_{\tau\in[0,t]} \normBig{\Ltwo}{\Delta \phi^1(\tau)}^2
\, \int_0^t\normBig{\Ltwo}{\nabla \partial_{t} \hat{\psi}(\tau)}^2\, \dd\tau 
\\
&\quad+ 8 \, \bb_6(\sigma) \, C_{\Delta}^2 \, C_{\Lfour \leftarrow \Hone}^2
\sup_{\tau\in[0,t]} \normBig{\Ltwo}{\Delta \partial_{t} \psi^2(\tau)}
\int_0^t\normBig{\Ltwo}{\nabla \hat{\phi}(\tau)} ^2\, \dd\tau
\,.
\end{split}
\end{equation*}
Thus, with the notation 
\begin{equation*}
\begin{split}
E_{-1}(t) 
&= \tfrac{1}{2} \, \normBig{\Ltwo}{\sqrt{\alpha^1(t)} \, \partial_{t} \hat{\psi}(t)}^2
+ \tfrac{\bba_2(\sigma_0)}{4} \, \normbig{\Ltwo}{\Delta \hat{\psi}(t)}^2 
 + \tfrac{\bb_3}{2} \, \normbig{\Ltwo}{\nabla \hat{\psi}(t)}^2\,, \\
\Ephi_{-1}(t) 
&= \tfrac{1}{2} \, \normBig{\Ltwo}{\sqrt{\alpha^1(t)} \, \partial_{t} \hat{\phi}(t)}^2
+ \tfrac{\bba_2(\sigma_0)}{4} \, \normbig{\Ltwo}{\Delta \hat{\phi}(t)}^2 
 + \tfrac{\bb_3}{2} \, \normbig{\Ltwo}{\nabla \hat{\phi}(t)}^2\,, 
\end{split}
\end{equation*}
and assuming that
\begin{equation*}
E_{i0}(t)\leq \bE_0\,, \quad
E_{i1}(t)\leq \bE_1\,, \quad
\Ephi_{i0}(t)\leq \bE_0\,, \quad
\Ephi_{i1}(t)\leq \bE_1\,, \quad i\in\{1,2\}\,,
\end{equation*}
where $E_{ij}$, $\Ephi_{ij}$ are the energies defined in \eqref{eq:E0E1E2} with $\psi,\phi$ replaced by $\psi^i,\phi^i$, $i\in\{1,2\}$, $j\in\{0,1\}$, as well as the smallness coondition
\begin{equation*}
 8 \, \bb_6(\sigma)^2 \, C_{\Delta}^2 \, C_{\Lfour \leftarrow \Hone}^2\, \Bigl(
3\, \normBig{\Ltwo}{\Delta \phi^1(0)}^2 + 6 T \, \tfrac{\overline{\alpha}\bE_1}{\bb_3}\Bigr)
\leq \tfrac{\bba_1}{4}\,,
\end{equation*}
we arrive at an energy estimate of the form
\begin{equation*}
E_{-1}(t)+ \tfrac{\bba_1}{4} \, \int_0^t\normBig{\Ltwo}{\nabla \partial_{t} \hat{\psi}(\tau)}^2 \; \dd\tau
\leq C_6 \int_0^t E_{-1}(\tau)\, d\tau +C_7 \int_0^t \Ephi_{-1}(\tau)\, d\tau
\end{equation*}
with 
\begin{equation*}
\begin{split}
&C_6= \tfrac{4}{\bba_0(\sigma_0)} \Bigl(1+\tfrac{1}{\bba_0(\sigma_0)} T\Bigr) + 
2C_{\text{PF}}^2 \, C_{\Lfour \leftarrow \Hone}^2 \tfrac{\bba_5(\sigma)}{\underline{\alpha}^2}\bE_1\\
&C_7= 
2C_{\text{PF}}^2 \, \Bigl(C_{\Lfour \leftarrow \Hone}^2 \tfrac{\bba_5(\sigma)}{\underline{\alpha}}+8C_\Delta\,\tfrac{\overline{\alpha}\bba_6(\sigma)}{\bb_3}\Bigr)\bE_1\,.
\end{split}
\end{equation*}
Gronwall's inequality therefore implies
\begin{equation*}
\sup_{t\in[0,T]} E_{-1}(t)
\leq C_7 \bigl(1+C_6 e^{C_6T}\Bigr)\int_0^T\Ephi_{-1}(\tau)\,\dd\tau
\leq C_7 \bigl(1+C_6 e^{C_6T}\Bigr)T\sup_{t\in[0,T]}\Ephi_{-1}(t)
\end{equation*}
thus, together with the smallness condition
\begin{equation*}
C_7 \bigl(1+C_6 e^{C_6T}\Bigr)T <1\,,
\end{equation*}
contractivity of the fixed point operator defined below, with respect to the norm induced by $\sup_{t\in[0,T]} E_{-1}(t)$.
}

\begin{proposition}
\label{thm:Proposition}
Consider the nonlinear damped wave equation~\eqref{eq:GeneralModel} under homogeneous Dirichlet boundary conditions~\eqref{eq:AssumptionDirichlet} and the initial conditions~\eqref{eq:InitialCondition}.
Suppose that the prescribed initial data satisfy the regularity and compatibility conditions 
\begin{equation*}
\psi_0, \psi_1 \in \Hthree(\Omega) \cap \Honezero(\Omega)\,, \quad \Delta \psi_0, \Delta \psi_1, \psi_2 \in \Honezero(\Omega)\,;
\end{equation*}
assume in addition that for $\norm{\Ltwo}{\Delta \psi_0}$, $\norm{\Ltwo}{\nabla \Delta \psi_0}$, and upper bounds $\Bounde{0}, \Bounde{1} > 0$ on the initial energies 
\begin{equation*}
\begin{gathered}
\normbig{\Ltwo}{\psi_2}^2 + \bba_2(\sigma_0) \, \normbig{\Ltwo}{\Delta \psi_1}^2 + \normbig{\Ltwo}{\nabla \psi_1}^2 \leq \Bounde{0}\,, \\
\normbig{\Ltwo}{\nabla \psi_2}^2 + \bba_2(\sigma_0) \, \normbig{\Ltwo}{\nabla \Delta \psi_1}^2 + \normbig{\Ltwo}{\Delta \psi_1}^2 \leq \Bounde{1}\,, 
\end{gathered}
\end{equation*}
the quantity 
\begin{equation*}
\begin{split}
M\big(\Bounde{0}, \Bounde{1}\big) 
&= \tfrac{C_{\text{PF}}^2 \, C_{\Lfour \leftarrow \Hone}^2 \bb_5(\sigma)}{\bba_1} \, \sqrt{\Bounde{0}} 
+ \tfrac{(C_{\Delta} C_{\Linfty \leftarrow \Htwo} \bb_5(\sigma))^2}{\bb_3} \, \Bounde{1} \\
&\qquad + \tfrac{C_2}{\bba_1} \, \Big(\normbig{\Ltwo}{\Delta \psi_0}^2 + C_3 \, T^2 \, \Bounde{1}\Big) 
+ C_4 \, \Big(\tfrac{1}{2} \, \normbig{\Ltwo}{\nabla \Delta \psi_0} + \sqrt{\Bounde{1}}\Big) 
\end{split}
\end{equation*}
is sufficiently small, see~\eqref{eq:PoincareEmbeddingEllipticity},~\eqref{eq:C1C2C3}, and~\eqref{eq:C4} for the definition of the arising constants. 
Then, there exists a weak solution 
\begin{equation*}
\begin{gathered}
\psi \in X = \Htwo\big([0, T], \Htwodiamond(\Omega)\big) \cap \Winftytwo\big([0, T], \Honezero(\Omega)\big) \cap \Winftyone\big([0, T], \Hthreediamond(\Omega)\big)\,, \\
\Htwodiamond(\Omega) = \big\{\chi \in \Htwo(\Omega): \chi \in \Honezero(\Omega)\big\}\,, \quad 
\Hthreediamond(\Omega) = \big\{\chi \in \Hthree(\Omega): \chi, \Delta \chi \in \Honezero(\Omega)\big\}\,,
\end{gathered}
\end{equation*}
to the associated equation 
\begin{equation*}
\begin{split}
&\partial_{tt} \psi(t) - \, \psi_2 - \bba_1 \, \Delta \big(\partial_{t} \psi(t) - \psi_1\big) 
+ \bba_2(\sigma_0) \, \Delta^2 \big(\psi(t) - \psi_0\big) - \bb_3 \, \Delta \big(\psi(t) - \psi_0\big) \\
&\quad + \bba_4(\sigma_0) \, \int_{0}^{t} \Delta^2 \psi(\tau) \; \dd\tau + \bb_5(\sigma) \, \big(\partial_{tt} \psi(t) \, \partial_{t} \psi(t) - \psi_2 \, \psi_1\big) \\
&\quad 
+ 2 \, \bb_6(\sigma) \, \big(\nabla \partial_{t} \psi(t) \cdot \nabla \psi(t) - \nabla \psi_1 \cdot \nabla \psi_0\big) = 0\,,   
\end{split}
\end{equation*}
obtained by integration with respect to time.
This solution satisfies a priori energy estimates of the form 
\begin{equation*}
\begin{gathered}
\EE_0\big(\psi(t)\big) 
= \normbig{\Ltwo}{\partial_{tt} \psi(t)}^2 + \bba_2(\sigma_0) \, \normbig{\Ltwo}{\Delta \partial_{t} \psi(t)}^2 + \normbig{\Ltwo}{\nabla \partial_{t} \psi(t)}^2\,, \\
\EE_1\big(\psi(t)\big)
= \normbig{\Ltwo}{\nabla \partial_{tt} \psi(t)}^2 + \bba_2(\sigma_0) \, \normbig{\Ltwo}{\nabla \Delta \partial_{t} \psi(t)}^2 
+ \normbig{\Ltwo}{\Delta \partial_{t} \psi(t)}^2\,, \\
\sup_{t \in [0, T]} \EE_0\big(\psi(t)\big) \leq \BoundE{0}\,, \quad \sup_{t \in [0, T]} \EE_1\big(\psi(t)\big) \leq \BoundE{1}\,, \quad 
\int_{0}^{T} \normbig{\Ltwo}{\Delta \partial_{tt} \psi(t)}^2 \; \dd t \leq \BoundE{2}\,, 
\end{gathered}
\end{equation*}
which hold uniformly for $a \in [0, \overline{a}]$.
In particular, the quantity $M(\BoundE{0}, \BoundE{1})$ remains sufficiently small to ensure uniform boundedness and hence non-degeneracy of the first time derivative 
\begin{equation*}
\begin{gathered}
0 < \underline{\alpha} = \tfrac{1}{2} \leq \normbig{\Linfty([0, T], \Linfty(\Omega))}{1 + \bb_5(\sigma) \, \partial_{t} \psi} \leq \overline{\alpha} = \tfrac{3}{2}\,, \\
0 < \tfrac{1}{\overline{\alpha}} = \tfrac{2}{3} \leq \normBig{\Linfty([0, T], \Linfty(\Omega))}{\big(1 + \bb_5(\sigma) \, \partial_{t} \psi\big)^{-1}} \leq \tfrac{1}{\underline{\alpha}} = 2\,.
\end{gathered}
\end{equation*}
\end{proposition}
\begin{proof} 
As indicated before, our proof relies on a fixed-point argument. 
For suitably chosen positive constants $\BoundE{0}, \BoundE{1}, \BoundE{2} > 0$ and suitably chosen inital data 
\begin{equation*}
\psi_0 \in \Hthreediamond(\Omega)\,, \quad \psi_1 \in \Hthreediamond(\Omega)\,, \quad \psi_2 \in \Honezero(\Omega)\,,
\end{equation*}
such that $M(\BoundE{0}, \BoundE{1})$ is sufficiently small, we introduce the nonempty closed subset 
\begin{equation*}
\begin{split}
\nM &= \bigg\{\phi \in X: \phi(0) = \psi_0\,, \partial_{t} \phi(0) = \psi_1\,, \partial_{tt} \phi(0) = \psi_2\,, \\ 
&\qquad \sup_{t \in [0, T]} \EE_0\big(\phi(t)\big) \leq \BoundE{0}\,, \sup_{t \in [0, T]} \EE_1\big(\phi(t)\big) \leq \BoundE{1}\,, 
\int_{0}^{T} \normbig{\Ltwo}{\Delta \partial_{tt} \phi(t)}^2 \; \dd t \leq \BoundE{2}\bigg\}\,. 
\end{split}
\end{equation*}
The nonlinear operator is defined by
\begin{equation*}
\nT: \nM \longrightarrow \nM: \phi \longmapsto \psi\,, 
\end{equation*}
where~$\psi$ is the solution to 
\begin{equation*}
\begin{split}
&\big(1 + \bb_5(\sigma) \, \partial_{t} \phi\big) \, \partial_{ttt} \psi - \bba_1 \, \Delta \partial_{tt} \psi
+ \bba_2(\sigma_0) \, \Delta^2 \partial_{t} \psi
- \bb_3 \, \Delta \partial_{t} \psi + \bba_4(\sigma_0) \, \Delta^2 \psi \\
&\quad + \bb_5(\sigma) \, \partial_{tt} \psi \, \partial_{tt} \phi 
+ 2 \, \bb_6(\sigma) \, \nabla \partial_{tt} \psi \cdot \nabla \phi 
+ 2 \, \bb_6(\sigma) \, \nabla \partial_{t} \psi \cdot \nabla \partial_{t} \phi = 0\,;
\end{split}
\end{equation*}
that is, in~\eqref{eq:GeneralModelReformulationDifferentiation1}, we replace~$\alpha$ and~$r$ by~$\alphaphi$ and~$\rphi$, see also~\eqref{eq:alpha_r_phi}. 
\begin{enumerate}[(i)]
\item 
\emph{Well-definedness.} 
The a priori energy estimate~\eqref{eq:EnergyEstimate} deduced before implies well-definedness and self-mapping of~$\nT$ into~$\nM$.
\item 
\emph{Continuity.} 
The set~$\nM$ is a weak* compact and convex subset of the Banach space $X$; 
thus, for ensuring existence of a fixed point of~$\nT$ from the general version of Schauder's Fixed Point Theorem in locally convex topological spaces, see \MyReference{Fan1952}, we have to prove weak* continuity of $\nT$.
For any sequence $(\phik)_{k \in \NN_{\geq 0}}$ in~$\nM$ converging weakly* to some $\phistar\in\nM$, the sequence of corresponding images defined by 
\begin{equation*}
\psik = \nT\big(\phik\big) \in \nM\,, \quad k \in \NN_{\geq 0}\,, 
\end{equation*}
is bounded in~$X$; 
hence, there exists a subsequence that converges to a function~$\psistar \in \nM$ in the following sense 
\begin{equation}\label{weakconvX}
\begin{gathered}
\psik \stackrel{*}{\rightharpoonup} \psistar \text{ in } X \text{ as } k \to \infty\,, \\
\psik \rightarrow \psistar \text{ in } \wX = \Hone\big([0, T], \Wfourone(\Omega)\big) \text{ as } k \to \infty\,, 
\end{gathered}
\end{equation}
with compact embedding $X \hookrightarrow \wX$. 
We apply a subsequence-subsequence argument for proving weak* convergence of~$\psik$ to~$\nT(\phistar)$. 
For this purpose, we consider an arbitrary weakly* convergent subsequence of $(\psik)_{k \in \NN_{\geq 0}}$ and prove that its limit~$\psistar$ concides with~$\nT(\phistar)$. 
Due to boundedness in~$X$, there is a sub-subsequence (not relabeled in the following) which converges in the sense of~\eqref{weakconvX}; 
the same type  of convergence can be assumed for the corresponding subsequence of preimages (also not relabeled) $(\phik)_{k \in \NN_{\geq 0}}$ to~$\phistar$.
It remains to verify the solution property $\psistar = \nT(\phistar)$.
\item 
\emph{Verification of solution property.}
We employ convenient abbreviations for the linear and the nonlinear terms 
\begin{equation}
\label{eq:LN}
\begin{split}
\big(\nLa \chi\big)(t) &= \partial_{tt} \chi(t) - \bba_1 \, \Delta \partial_{t} \chi(t) + \bba_2(\sigma_0) \, \Delta^2 \chi(t) - \bb_3 \, \Delta \chi(t) \\
&\qquad + \bba_4(\sigma_0) \, \int_{0}^{t} \Delta^2 \chi(\tau) \; \dd\tau\,, \\
\nLa_0 &= - \, \psi_2 + \bba_1 \, \Delta \psi_1 - \bba_2(\sigma_0) \, \Delta^2 \psi_0+ \bb_3 \, \Delta \psi_0\,, \\
\nN(\phi, \chi) &= \bb_5(\sigma) \, \partial_{tt} \chi \, \partial_{t} \phi + 2 \, \bb_6(\sigma) \, \nabla \partial_{t} \chi \cdot \nabla \phi\,, \\
\nN_0 &= - \, \bb_5(\sigma) \, \psi_2 \, \psi_1 - 2 \, \bb_6(\sigma) \, \nabla \psi_1 \cdot \nabla \psi_0\,;  
\end{split}
\end{equation}
the relation 
\begin{equation*}
\nLa \psik + \nLa_0 + \nN\big(\phik, \psik\big) + \nN_0 = 0
\end{equation*}
thus corresponds to the given reformulation of the defining equation, obtained by integration with respect to time. 
In order to verify that~$\psistar$ is a solution to 
\begin{equation*}
\nLa \psistar + \nLa_0 + \nN\big(\phistar, \psistar\big) + \nN_0 = 0\,, 
\end{equation*}
we consider the difference 
\begin{equation*}
\begin{split}
&\nLa \big(\psik - \psistar\big) + \nN\big(\phik, \psik\big) - \nN\big(\phistar, \psistar\big) \\
&\quad = \nLa \big(\psik - \psistar\big) + \nN\big(\phik - \phistar, \psik\big) 
+ \nN\big(\phistar, \psik - \psistar\big)\,. 
\end{split}
\end{equation*}
Due to the fact that $\phik \stackrel{*}{\rightharpoonup} \phistar$ in~$X$ as $k \to \infty$, the linear contribution tends to zero in $\Linfty([0, T], \Hminus(\Omega))$.
The first terms in the nonlinearity satisfy 
\begin{equation*}
\begin{split}
&\bb_5(\sigma) \, \normbig{\Linfty([0, T], \Lfour(\Omega))}{\partial_{tt} \psik} \, 
\normbig{\Ltwo([0, T], \Lfour(\Omega))}{\partial_{t} \big(\phik - \phistar\big)} \\
&\qquad + 2 \, \bb_6(\sigma) \, \normbig{\Linfty([0, T], \Lfour(\Omega))}{\nabla \partial_{t} \psik} \, 
\normbig{\Ltwo([0, T], \Lfour(\Omega))}{\nabla \big(\phik - \phistar\big)} \\
&\qquad + 2 \, \bb_6(\sigma) \, \normbig{\Ltwo([0, T], \Lfour(\Omega))}{\nabla \partial_{t} \big(\psik - \psistar\big)} \, 
\normbig{\Linfty([0, T], \Lfour(\Omega))}{\nabla \phistar} \\
&\quad \leq C_{\Lfour \leftarrow \Hone} \, \Big(\big(\bb_5(\sigma) 
+ 2 \, \bb_6(\sigma)\big)  \, \normbig{X}{\psik} \, \normbig{\wX}{\phik - \phistar} \\
&\qquad + 2 \, \bb_6(\sigma) \, \normbig{\wX}{\psik - \psistar} \, \normbig{X}{\phistar}\Big) 
\end{split}
\end{equation*}
and therefore tend to zero by the strong convergence of~$\phik$ and~$\psik$ in~$\wX$;
for any $v \in \Ltwo([0, T], \Ltwo(\Omega))$, due to the fact that 
\begin{equation*}
\begin{gathered}
\partial_{tt} (\psik - \psistar) \rightharpoonup 0 \text{ in } \Ltwo\big([0, T], \Ltwo(\Omega)\big) \text{ as } k \to \infty\,, \\
\partial_{t} \phistar \, v \in \Ltwo\big([0, T], \Ltwo(\Omega)\big)\,, 
\end{gathered}
\end{equation*}
we further have  
\begin{equation*}
\bb_5(\sigma) \int_{0}^{T} \scpBig{\Ltwo}{\partial_{tt} \big(\psik(t) - \psistar(t)\big)}{\partial_{t} \phistar(t) \, v(t)} \; \dd t 
\to 0 \text{ as } k \to \infty\,, 
\end{equation*}
which concludes the proof. 
\hfill $\diamond$ 
\end{enumerate}
\end{proof}

\begin{remark}
\label{Remark1}
Our result compares with \MyReference{Kaltenbacher2017}, where under the stronger regularity requirements $\psi_0\in \Hfour(\Omega)$, $\psi_1 \in \Hthree(\Omega)$, $\psi_2 \in \Htwo(\Omega)$ and additional compatibility conditions on the initial data existence and uniqueness of a solution 
\begin{equation*}
\begin{split}
\psi &\in \Hthree\big((0,\infty), \Ltwo(\Omega)\big) \cap \Winftytwo\big((0,\infty), \Hone(\Omega)\big) \cap \Htwo\big((0,\infty), \Htwo(\Omega)\big) \\
&\qquad \cap \Winftyone\big((0,\infty), \Hthree(\Omega)\big) \cap \Hone\big((0,\infty), \Hfour(\Omega)\big) \cap \Linfty\big((0,\infty), \Hfour(\Omega)\big)
\end{split}
\end{equation*} 
to the general model is proven. 
\end{remark}
\section{Limiting systems}
\label{SectionLimit}
The transition from the Brunnhuber--Jordan--Kuznetsov equation to the Kuznetsov and Westervelt equations permits a significant reduction of the temporal order of differentiation from three to two, which is for instance of relevance with regard to numerical simulations.
In this section, we rigorously justify this limiting process under a suitable compatibility condition on the initial data. 

\paragraph{Situation.}
We consider the unifying representation~\eqref{eq:GeneralModel} including (BJK), (BCK), (BJW), and (BCW), respectively; 
for the sake of clearness, we indicate the dependence of the solution on the decisive parameter $a > 0$.
We suppose that the assumptions of Proposition~\ref{thm:Proposition} are satisfied; 
note that the prescribed initial data are independent of $a > 0$ and that the fundamental smallness requirement on $M(\Bounde{0}, \Bounde{1})$ or $M(\BoundE{0}, \BoundE{1})$, respectively, can be fulfilled uniformly for $a \in (0, \overline{a}]$.
The main result of this work, given below, ensures convergence in a weak sense towards the solution of the Kuznetsov and Westervelt equation, respectively.  
In contrast to Proposition~\ref{thm:Proposition}, the canonical solution space is now 
\begin{equation*}
X_0 = \Htwo\big([0, T], \Htwodiamond(\Omega)\big) \cap \Winftytwo\big([0, T], \Honezero(\Omega)\big)\,, 
\end{equation*}
that is, we employ the regularity properties 
\begin{equation*}
\int_{0}^{T} \normbig{\Ltwo}{\Delta \partial_{tt} \psia(t)}^2 
+ \, \underset{t \in [0, T]}{\text{ess sup}} \, \normbig{\Ltwo}{\nabla \partial_{tt} \psia(t)} < \infty\,; 
\end{equation*}
due to the fact that $\bba_2(\sigma_0) \to 0$ as $a \to 0_{+}$ and hence the terms 
\begin{equation*}
\bba_2(\sigma_0) \, \normbig{\Ltwo}{\Delta \partial_{t} \psi(t)}^2\,, \quad 
\bba_2(\sigma_0) \, \normbig{\Ltwo}{\nabla \Delta \partial_{t} \psi(t)}^2 
\end{equation*}
arising in the energy estimates vanish, the higher regularity of the solution space~$X$ can not be achieved. 

\begin{theorem}
In the situation of Proposition~\ref{thm:Proposition}, assume in addition that the prescribed initial data satisfy the consistency condition 
\begin{equation}
\label{eq:ICConsistency}
\psi_2 - \bbzero_1 \, \Delta \psi_1 - \bb_3 \, \Delta \psi_0 + \bb_5(\sigma) \, \psi_2 \, \psi_1 + 2 \, \bb_6(\sigma) \, \nabla \psi_1 \cdot \nabla \psi_0 = 0\,.
\end{equation}
For any $a \in (0, \overline{a}]$, let $\psia:[0, T] \to \Ltwo(\Omega)$ denote the solution to the nonlinear damped wave equation 
\begin{equation*}
\begin{split}
&\partial_{ttt} \psia(t) 
- \bba_1 \, \Delta \partial_{tt} \psia(t) 
+ \bba_2(\sigma_0) \, \Delta^2 \partial_{t} \psia(t) 
- \bb_3 \, \Delta \partial_{t} \psia(t) \\
&\quad + \bba_4(\sigma_0) \, \Delta^2 \psia(t) 
+ \partial_{tt} \Big(\tfrac{1}{2} \, \bb_5(\sigma) \, \big(\partial_{t} \psia(t)\big)^2 + \bb_6(\sigma) \, \abs{\nabla \psia(t)}^2\Big) = 0
\end{split}
\end{equation*}
under homogeneous Dirichlet boundary conditions and the initial conditions 
\begin{equation*}
\psia(0) = \psi_0\,, \quad \partial_{t} \psia(0) = \psi_1\,, \quad \partial_{tt} \psia(0) = \psi_2\,,
\end{equation*}
or of the following reformulation obtained by integration and application of~\eqref{eq:ICConsistency}
\begin{equation*}
\begin{split}
&\partial_{tt} \psia(t) 
- \bbzero_1 \, \Delta \partial_{t} \psia(t) 
- \big(\bba_1 - \bbzero_1\big) \, \big(\Delta \partial_{t} \psia(t) - \Delta \psi_1\big) \\
&\quad + \bba_2(\sigma_0) \, \big(\Delta^2 \psia(t) - \Delta^2 \psi_0\big) 
- \bb_3 \, \Delta \psia(t) 
+ \bba_4(\sigma_0) \int_{0}^{t} \Delta^2 \psia(\tau) \; \dd\tau \\
&\quad + \bb_5(\sigma) \, \partial_{tt} \psia(t) \, \partial_{t} \psia(t) + 2 \, \bb_6(\sigma) \, \nabla \partial_{t} \psia(t) \cdot \nabla \psia(t) = 0\,,
\end{split}
\end{equation*}
respectively, see~\eqref{eq:GeneralModel} and~\eqref{eq:GeneralModelReformulationIntegrationF}. 
Then, as $a \to 0_{+}$, the family $(\psia)_{a \in (0, \overline{a}]}$ converges to the solution $\psizero:[0, T] \to \Ltwo(\Omega)$ of the limiting system
\begin{equation}
\label{eq:Solution2}
\begin{split}
&\partial_{tt} \psizero(t) - \bbzero_1 \, \Delta \partial_{t} \psizero(t) - \bb_3 \, \Delta \psizero(t) \\
&\quad + \bb_5(\sigma) \, \partial_{tt} \psizero(t) \, \partial_{t} \psizero(t) 
+ 2 \, \bb_6(\sigma) \, \nabla \partial_{t} \psizero(t) \cdot \nabla \psizero(t) = 0\,,
\end{split}
\end{equation}
see~\eqref{eq:GeneralModelLimit}; 
more precisely, for the solution to the associated weak formulation, obtained by testing with $v \in \Lone([0, T], \Honezero(\Omega))$ and performing integration-by-parts, convergence is ensured in the following sense
\begin{equation*}
\psia \stackrel{*}{\rightharpoonup} \psizero \text{ in } X_0 \text{ as } a \to 0_{+}\,. 
\end{equation*}
\end{theorem}
\begin{proof}
\begin{enumerate}[(i)]
\item 
\emph{Convergence.} 
In the present situation, as a consequence of Proposition~\ref{thm:Proposition}, a sequence of positive numbers $(a_k)_{k \in \NN}$ with limit zero exists such that the associated sequence~$(\psiak)_{k \in \NN}$ converges to a function $\psizero \in X_0$ in the following sense 
\begin{equation*}
\begin{gathered}
\psiak \stackrel{*}{\rightharpoonup} \psizero \text{ in } X_0 \text{ as } k \to \infty\,, \\
\psiak \rightarrow \psizero \text{ in } \wX = \Hone\big([0, T], \Wfourone(\Omega)\big) \text{ as } k \to \infty\,.
\end{gathered}
\end{equation*}
\item 
\emph{Verification of solution property.} 
In order to verify that~$\psizero$ is a solution to~\eqref{eq:Solution2}, we make use of the fact that any function~$\psiak$ satisfies 
\begin{equation*}
\nLak \psiak + \nLak_0 + \nN\big(\psiak, \psiak\big) + \nN_0 = 0\,, 
\end{equation*}
see~\eqref{eq:LN}, and prove that the difference 
\begin{equation*}
\begin{split}
&\nLak \psiak - \nLzero \psizero + \nN\big(\psiak, \psiak\big) - \nN\big(\psizero, \psizero\big) \\
&= \big(\nLak - \nLzero\big) \, \psiak + \nLzero \big(\psiak - \psizero\big) \\
&\qquad + \nN\big(\psiak - \psizero, \psiak\big) + \nN\big(\psizero, \psiak - \psizero\big)
\end{split}
\end{equation*}
tends to zero in a weak sense. 
On the one hand, testing the reformulation of the general model with $v \in \Lone([0, T], \Honezero(\Omega))$ and employing integration-by-parts, yields 
\begin{equation*}
\begin{split}
&\int_{0}^{T} \scpBig{\Ltwo}{\big(\nLak - \nLzero\big) \, \psiak(t)}{v(t)} \; \dd t \\
&\quad = \int_{0}^{T} \Big(\big(\bbak_1 - \bbzero_1\big) \, \scpbig{\Ltwo}{\nabla \partial_{t} \psiak(t)}{\nabla v(t)} \\
&\quad\qquad - \bbak_2(\sigma_0) \, \scpbig{\Ltwo}{\nabla \Delta \psiak(t)}{\nabla v(t)}\Big) \; \dd t \\
&\quad\qquad - \tfrac{\bbak_2(\sigma_0)}{\bbak_0(\sigma_0)} \int_{0}^{T} \int_{0}^{t} \scpbig{\Ltwo}{\nabla \Delta \psiak(\tau)}{\nabla v(t)} \; \dd \tau \, \dd t\,, 
\end{split}
\end{equation*}
which tends to zero, since 
\begin{equation*}
\normbig{\Linfty([0, T], \Ltwo(\Omega))}{\nabla \partial_{t} \psiak}\,, \quad 
\sqrt{\bbak_2(\sigma_0)} \, \normbig{\Linfty([0, T], \Ltwo(\Omega))}{\nabla \Delta \psiak}\,, 
\end{equation*} 
are uniformly bounded for $a_k \in [0, \overline{a}]$.
On the other hand, it is seen that 
\begin{equation*}
\begin{split}
&\int_{0}^{T} \scpBig{\Ltwo}{\nLzero \big(\psiak(t) - \psizero(t)\big)}{v(t)} \; \dd t \\
&\quad = \int_{0}^{T} \bigg(\scpBig{\Ltwo}{\partial_{tt} \big(\psiak(t) - \psizero(t)\big)}{v(t)} \\
&\quad\qquad 
- \bbzero_1 \, \scpBig{\Ltwo}{\Delta \partial_{t} \big(\psiak(t) - \psizero(t)\big)}{v(t)} \\
&\quad\qquad - \bb_3 \, \scpBig{\Ltwo}{\Delta \big(\psiak(t) - \psizero(t)\big)}{v(t)}\bigg) \; \dd t 
\end{split}
\end{equation*}
tends to zero by the weak convergence in~$X_0$.
For the nonlinear part, the same argument as given in the proof of Proposition~\ref{thm:Proposition} applies. 
We finally note that convergence of the family $(\psia)_{a \in (0, \overline{a}}$ follows from a subsequence-subsequence argument and uniqueness of the solutions to the Kuznetsov and Westervelt equations.
Altogether, we thus obtain 
\begin{equation*}
\begin{split}
&\int_{0}^{T} \Big(\scpbig{\Ltwo}{\partial_{tt} \psia(t)}{v(t)} 
+ \bbzero_1 \, \scpbig{\Ltwo}{\nabla \partial_{t} \psia(t)}{\nabla v(t)} \\
&\qquad + \bb_3 \, \scpbig{\Ltwo}{\nabla \psia(t)}{\nabla v(t)} \\
&\qquad + \big(\bba_1 - \bbzero_1\big) \, \scpbig{\Ltwo}{\nabla \partial_{t} \psia(t) - \nabla \psi_1}{\nabla v(t)} \\
&\qquad - \bba_2(\sigma_0) \, \scpbig{\Ltwo}{\nabla \Delta \psia(t) - \nabla \Delta \psi_0}{\nabla v(t)} \\
&\qquad - \bba_4(\sigma_0) \int_{0}^{T} \int_{0}^{t} \scpbig{\Ltwo}{\nabla \Delta \psia(\tau)}{\nabla v(t)} \; \dd \tau \\
&\qquad + \bb_5(\sigma) \, \scpbig{\Ltwo}{\partial_{tt} \psia(t) \, \partial_{t} \psia(t)}{v(t)} \\
&\qquad + 2 \, \bb_6(\sigma) \, \scpbig{\Ltwo}{\nabla \partial_{t} \psia(t) \cdot \nabla \psia(t)}{v(t)}\Big) \; \dd t \\
\overset{a \to 0_{+}}{\longrightarrow} \, 
&\int_{0}^{T} \Big(\scpbig{\Ltwo}{\partial_{tt} \psizero(t)}{v(t)} 
+ \bbzero_1 \, \scpbig{\Ltwo}{\nabla \partial_{t} \psizero(t)}{\nabla v(t)} \\
&\qquad + \bb_3 \, \scpbig{\Ltwo}{\nabla \psizero(t)}{\nabla v(t)} \\
&\qquad + \bb_5(\sigma) \, \scpbig{\Ltwo}{\partial_{tt} \psizero(t) \, \partial_{t} \psizero(t)}{v(t)} \\
&\qquad + 2 \, \bb_6(\sigma) \, \scpbig{\Ltwo}{\nabla \partial_{t} \psizero(t) \cdot \nabla \psizero(t)}{v(t)}\Big) \; \dd t\,, 
\end{split}
\end{equation*}
which concludes the proof. 
\hfill $\diamond$ 
\end{enumerate}
\end{proof}

\begin{remark}
Under stronger regularity and compatibility requirements on the initial data, the solution space considered in \MyReference{KaltenbacherLasiecka2012} for the Kuznetsov equation is 
\begin{equation*}
\begin{split}
X_0 &\cap \Winftythree\big([0, T], \Ltwo(\Omega)\big) \cap \Hthree\big([0, T], \Honezero (\Omega)\big) \cap \Winftytwo\big([0, T], \Honezero(\Omega)\big) \\
& \cap \Winftyone\big([0, T], \Htwo(\Omega)\big)\,, 
\end{split}
\end{equation*}
see Theorem 1.1 with $u = \partial_{t} \psi$. 
\end{remark}

\appendix 
\section{Detailed derivation of most general model}
\label{SectionAppendix}
In the following, we deduce the Blackstock--Crighton--Brunnhuber--Jordan--Kuznetsov equation~\eqref{eq:BrunnhuberJordanKuznetsov} from the conservation laws for mass, momentum, and energy as well as a heuristic equation of state relating mass density, acoustic pressure, and temperature, see~\eqref{eq:ConservationLaws} and~\eqref{eq:EquationState}.
For notational simplicity, we include detailed calculations for the one-dimensional case; 
the extension to higher space dimensions is then straightforward. 
In order to indicate that only terms which are linear or quadratic with respect to the fluctuating quantities are taken into account, we introduce a (small) positive real number $\varepsilon > 0$ and set 
\begin{equation}
\label{eq:MeanFlucations}
\varrho = \varrho_0 + \varepsilon \, \varrho_{\sim}\,, \quad v = \varepsilon \, \partial_x \psi\,, \quad 
p = p_0 + \varepsilon \, p_{\sim}\,, \quad T = T_0 + \varepsilon \, T_{\sim}\,;
\end{equation}
here, we anticipate that inserting the Helmholtz composition~\eqref{eq:Helmholtz} into the fundamental relations permits a decoupling into irrotational and  rotational parts. 
Moreover, to identify terms that are related to dissipative effects, we replace $\mu_B, \mu, c_V, c_p$ as well as~$\hat{A}$ and~$a$ by 
\begin{equation*}
\delta \mu_B\,, \quad \delta \mu\,, \quad \delta c_V\,, \quad \delta c_p\,, \quad \gamma \hat{A}\,, \quad \lambda \, a\,, 
\end{equation*}
where $\delta, \gamma, \lambda > 0$ denote (small) positive real numbers that will be adjusted later on. 

\paragraph{Fundamental relations.} 
For convenience, we restate the fundamental equations~\eqref{eq:ConservationLaws} and~\eqref{eq:EquationState} employing~\eqref{eq:MeanFlucations}. 
In a single space dimension, the relation reflecting conservation of mass~\eqref {eq:ConservationMass} reads 
\begin{subequations}
\begin{equation}
\label{eq:Rel(i)1}
\varepsilon \, \partial_{t} \varrho_{\sim} 
+ \varepsilon \, \varrho_0 \, \partial_{xx} \psi 
+ \varepsilon^2 \, \partial_x \varrho_{\sim} \, \partial_x \psi 
+ \varepsilon^2 \, \varrho_{\sim} \, \partial_{xx} \psi = 0\,.
\end{equation}
Omitting higher-order contributions, i.e.~terms of the form $o(\varepsilon^2)$, the relation describing conservation of momentum~\eqref{eq:ConservationMomentum} simplifies as follows 
\begin{equation*}
\begin{split}
&\varepsilon \, \varrho_0 \, \partial_{xt} \psi 
+ \varepsilon \, \partial_x p_{\sim} 
- \, \varepsilon \, \delta \, \big(\mu_B + \tfrac{4}{3} \, \mu\big) \, \partial_{xxx} \psi \\ 
&\quad + \varepsilon^2 \, \partial_{t} \varrho_{\sim} \, \partial_x \psi 
+ \varepsilon^2 \, \varrho_{\sim} \, \partial_{xt} \psi 
+ 2 \, \varepsilon^2 \, \varrho_0 \, \partial_x \psi \, \partial_{xx} \psi = 0\,.
\end{split}
\end{equation*}
Substracting the $\varepsilon \, \partial_x \psi$ multiple of~\eqref{eq:Rel(i)1}, leads to 
\begin{equation}
\label{eq:Rel(ii)1}
\begin{split}
&\partial_x \, \Big(\varepsilon \, \varrho_0 \, \partial_{t} \psi 
+ \varepsilon \, p_{\sim} 
- \, \varepsilon \, \delta \, \big(\mu_B + \tfrac{4}{3} \, \mu\big) \, \partial_{xx} \psi
+ \varepsilon^2 \, \tfrac{\varrho_0}{2} \, \big(\partial_x \psi\big)^2\Big) \\
&\quad + \varepsilon^2 \, \varrho_{\sim} \, \partial_{xt} \psi 
= 0\,.
\end{split}
\end{equation}
Neglecting contributions of the form $o(\varepsilon^2)$, we obtain the following relation reflecting the conservation of energy~\eqref{eq:ConservationEnergy} in a single space dimension
\begin{equation}
\label{eq:Rel(iii)1}
\varepsilon \, \delta \, \tfrac{c_p-c_V}{\alpha_V} \, \varrho_0 \, \partial_{xx} \psi
- \varepsilon \, \lambda \, a \, \partial_{xx} T_{\sim} 
+ \varepsilon \, \delta \, c_V \varrho_0 \, \partial_{t} T_{\sim} = 0\,.
\end{equation}
Omitting higher-order contributions, the equation of state~\eqref{eq:EquationState} reduces to 
\begin{equation}
\label{eq:Rel(iv)1}
\varepsilon \, p_{\sim} = \varepsilon \, \tfrac{A}{\varrho_0} \, \varrho_{\sim} + \varepsilon^2 \, \tfrac{B}{2 \varrho_0^2} \, \varrho_{\sim}^2 
+ \varepsilon \, \gamma \, \tfrac{\hat{A}}{T_0} \, T_{\sim}\,.
\end{equation}
\end{subequations}

\paragraph{Linear wave equation.}
Reconsidering the equations~\eqref{eq:Rel(i)1}-\eqref{eq:Rel(iv)1} and incorporating only first-order contributions, i.e. terms of the form $\nO(\varepsilon)$, yields 
\begin{equation*}
\partial_{t} \varrho_{\sim} + \varrho_0 \, \partial_{xx} \psi = 0\,, \quad 
\partial_x \, \big(\varrho_0 \, \partial_{t} \psi + p_{\sim}\big) = 0\,, \quad 
p_{\sim} = \tfrac{A}{\varrho_0} \, \varrho_{\sim}\,.
\end{equation*}
Integration with respect to the space variable shows that a solution of the system 
\begin{equation*}
\partial_{t} \varrho_{\sim} = - \, \varrho_0 \, \partial_{xx} \psi\,, \quad 
p_{\sim} = \tfrac{A}{\varrho_0} \, \varrho_{\sim} = - \, \varrho_0 \, \partial_{t} \psi\,,
\end{equation*}
is also a solution of the original system.
The relation for the acoustic pressure implies
\begin{equation*}
\varrho_{\sim} = - \, \tfrac{\varrho_0^2}{A} \, \partial_{t} \psi\,;
\end{equation*}
together with the identity $A = c_0^2 \, \varrho_0$ this leads to a linear wave equation for the acoustic velocity potential 
\begin{equation*}
\partial_{tt} \psi - c_0^2 \, \partial_{xx} \psi = 0\,.
\end{equation*}

\paragraph{Nonlinear damped wave equation.}
The above considerations explain the ansatz 
\begin{equation*}
\varrho_{\sim} = - \, \tfrac{\varrho_0^2}{A} \, \partial_{t} \psi + \varepsilon \, \varrho_0 \, F
\end{equation*}
with space-time-dependent real-valued function~$F$ determined by~\eqref{eq:Rel(i)1}. 
Inserting this representation into~\eqref{eq:Rel(i)1}-\eqref{eq:Rel(iv)1}, neglecting higher-order contributions, employing the identity
\begin{equation*}
\partial_{xt} \psi \, \partial_{t} \psi = \tfrac{1}{2} \, \partial_x \big(\partial_{t} \psi\big)^2\,, 
\end{equation*}
and integrating~\eqref{eq:Rel(ii)1} with respect to space, we arrive at 
\begin{subequations}
\begin{gather}
\label{eq:Rel(i)}
\varepsilon^2 \, \partial_{t} F
= \varepsilon \, \tfrac{\varrho_0}{A} \, \partial_{tt} \psi - \varepsilon \, \partial_{xx} \psi
+ \varepsilon^2 \, \tfrac{\varrho_0}{A} \, \big(\partial_{xt} \psi \, \partial_x \psi + \partial_{xx} \psi \, \partial_{t} \psi\big)\,, \\ 
\label{eq:Rel(ii)}
\varepsilon \, \varrho_0 \, \partial_{t} \psi 
+ \varepsilon \, p_{\sim} 
- \, \varepsilon \, \delta \, \big(\mu_B + \tfrac{4}{3} \, \mu\big) \, \partial_{xx} \psi
+ \varepsilon^2 \, \tfrac{\varrho_0}{2} \, \big(\partial_x \psi\big)^2
- \, \varepsilon^2 \, \tfrac{\varrho_0^2}{2 A} \, \big(\partial_{t} \psi\big)^2 = 0\,, \\
\label{eq:Rel(iii)}
\varepsilon \, \delta \, \tfrac{c_p-c_V}{\alpha_V} \, \varrho_0 \, \partial_{xx} \psi
- \varepsilon \, \lambda \, a \, \partial_{xx} T_{\sim} 
+ \varepsilon \, \delta \, c_V \varrho_0 \, \partial_{t} T_{\sim} = 0\,, \\
\label{eq:Rel(iv)}
\varepsilon \, \varrho_0 \, \partial_{t} \psi + \varepsilon \, p_{\sim} 
= \varepsilon^2 \, A \, F
+ \varepsilon^2 \, \tfrac{B}{A} \, \tfrac{\varrho_0^2}{2 A} \, \big(\partial_{t} \psi\big)^2 
+ \varepsilon \, \gamma \, \tfrac{\hat{A}}{T_0} \, T_{\sim}\,.
\end{gather}
\end{subequations}
On the one hand, we insert~\eqref{eq:Rel(iv)} into~\eqref{eq:Rel(ii)}, differentiate the resulting equation with respect to time, and insert~\eqref{eq:Rel(i)} to obtain 
\begin{equation*}
\begin{split}
&\varepsilon \, \varrho_0 \, \partial_{tt} \psi - \varepsilon \, A \, \partial_{xx} \psi
- \, \varepsilon \, \delta \, \big(\mu_B + \tfrac{4}{3} \, \mu\big) \, \partial_{xxt} \psi
+ \varepsilon \, \gamma \, \tfrac{\hat{A}}{T_0} \, \partial_{t} T_{\sim} \\
&\quad + \varepsilon^2 \, \big(\tfrac{B}{A} - 1\big) \, \tfrac{\varrho_0^2}{A} \, \partial_{tt} \psi \, \partial_{t} \psi 
+ 2 \, \varepsilon^2 \, \varrho_0 \, \partial_{xt} \psi \, \partial_x \psi
+ \varepsilon^2 \, \varrho_0 \, \partial_{xx} \psi \, \partial_{t} \psi = 0\,;
\end{split}
\end{equation*}
replacing the second-order contribution $\varepsilon^2 \, \varrho_0 \, \partial_{xx} \psi \, \partial_{t} \psi$ with 
$\varepsilon^2 \, \tfrac{\varrho_0^2}{A} \, \partial_{tt} \psi \, \partial_{t} \psi + o(\varepsilon^2)$, see~\eqref{eq:Rel(i)}, further yields  
\begin{equation}
\label{eq:Rel1}
\begin{split}
\varepsilon \, \gamma \, \partial_{t} T_{\sim}
= &- \, \varepsilon \, \tfrac{\varrho_0 T_0}{\hat{A}} \, \partial_{tt} \psi + \varepsilon \, \tfrac{A T_0}{\hat{A}} \, \partial_{xx} \psi 
+ \varepsilon \, \delta \, \big(\mu_B + \tfrac{4}{3} \, \mu\big) \, \tfrac{T_0}{\hat{A}} \, \partial_{xxt} \psi \\
&- \, \varepsilon^2 \, \tfrac{B}{A} \, \tfrac{\varrho_0 T_0}{\hat{A}} \, \tfrac{\varrho_0}{2 A} \, \partial_{t} \big(\partial_{t} \psi\big)^2 
- \varepsilon^2 \, \tfrac{\varrho_0 T_0}{\hat{A}} \, \partial_{t} \big(\partial_x \psi\big)^2\,.
\end{split}
\end{equation}
On the other hand, differentiating~\eqref{eq:Rel(iii)} with respect to time, we have 
\begin{equation*}
\varepsilon \, \delta \, \tfrac{c_p-c_V}{\alpha_V} \, \varrho_0 \, \partial_{xxt} \psi
- a \, \partial_{xx} \big(\varepsilon \, \lambda \, \partial_{t} T_{\sim}\big)
+ c_V \varrho_0 \, \partial_{t} \big(\varepsilon \, \delta \, \partial_{t} T_{\sim}\big) = 0\,; 
\end{equation*}
with the help of~\eqref{eq:Rel1}, this yields 
\begin{equation*}
\begin{split}
&\varepsilon \, \partial_{ttt} \psi 
- \Big(\tfrac{\varepsilon \lambda}{\delta} \, \tfrac{a}{c_V \varrho_0} + \varepsilon \, \delta \, \tfrac{1}{\varrho_0} \, \big(\mu_B + \tfrac{4}{3} \, \mu\big)\Big) \, \partial_{xxtt} \psi \\
&\quad + \varepsilon \, \lambda \, \tfrac{a}{c_V \varrho_0} \, \tfrac{1}{\varrho_0} \, \big(\mu_B + \tfrac{4}{3} \, \mu\big) \, \partial_{xxxxt} \psi \\
&\quad - \Big(\varepsilon \, \tfrac{A}{\varrho_0} + \varepsilon \, \gamma \, \tfrac{c_p-c_V}{\alpha_V} \, \tfrac{\hat{A}}{c_V \varrho_0 T_0}\Big) \, \partial_{xxt} \psi \\
&\quad + \tfrac{\varepsilon \lambda}{\delta} \, \tfrac{a}{c_V \varrho_0} \, \tfrac{A}{\varrho_0} \, \partial_{xxxx} \psi \\
&\quad + \varepsilon^2 \, \tfrac{B}{A} \, \tfrac{\varrho_0}{2 A} \, \partial_{tt} \big(\partial_{t} \psi\big)^2 
+ \varepsilon^2 \, \partial_{tt} \big(\partial_x \psi\big)^2 \\
&\quad - \tfrac{\varepsilon^2 \lambda}{\delta} \, \tfrac{B}{A} \, \tfrac{a}{2 A c_V} \, \partial_{xxt} \big(\partial_{t} \psi\big)^2
- \tfrac{\varepsilon^2 \lambda}{\delta} \, \tfrac{a}{c_V \varrho_0} \, \partial_{xxt} \big(\partial_x \psi\big)^2 = 0\,. 
\end{split}
\end{equation*}
With the special scaling 
\begin{equation*}
\delta = \sqrt{\varepsilon}\,, \quad \gamma = \sqrt{\varepsilon} \, \varepsilon\,, \quad \lambda = \varepsilon\,, 
\end{equation*}
we arrive at the relation 
\begin{equation*}
\begin{split}
&\varepsilon \, \partial_{ttt} \psi 
- \varepsilon \, \sqrt{\varepsilon} \, \Big(\tfrac{a}{c_V \varrho_0} + \tfrac{1}{\varrho_0} \, \big(\mu_B + \tfrac{4}{3} \, \mu\big)\Big) \, \partial_{xxtt} \psi \\
&\quad + \varepsilon^2 \, \tfrac{a}{c_V \varrho_0} \, \tfrac{1}{\varrho_0} \, \big(\mu_B + \tfrac{4}{3} \, \mu\big) \, \partial_{xxxxt} \psi \\
&\quad - \Big(\varepsilon \, \tfrac{A}{\varrho_0} + \varepsilon^2 \, \sqrt{\varepsilon} \, \tfrac{c_p-c_V}{\alpha_V} \, \tfrac{\hat{A}}{c_V \varrho_0 T_0}\Big) \, \partial_{xxt} \psi \\
&\quad + \varepsilon \sqrt{\varepsilon} \, \tfrac{a}{c_V \varrho_0} \, \tfrac{A}{\varrho_0} \, \partial_{xxxx} \psi \\
&\quad + \varepsilon^2 \, \tfrac{B}{A} \, \tfrac{\varrho_0}{2 A} \, \partial_{tt} \big(\partial_{t} \psi\big)^2 
+ \varepsilon^2 \, \partial_{tt} \big(\partial_x \psi\big)^2 \\
&\quad - \varepsilon^2 \sqrt{\varepsilon} \, \tfrac{B}{A} \, \tfrac{a}{2 A c_V} \, \partial_{xxt} \big(\partial_{t} \psi\big)^2
- \varepsilon^2 \sqrt{\varepsilon} \, \tfrac{a}{c_V \varrho_0} \, \partial_{xxt} \big(\partial_x \psi\big)^2 = 0\,; 
\end{split}
\end{equation*}
neglecting the higher-order terms 
\begin{equation*}
\varepsilon^2 \, \sqrt{\varepsilon} \, \tfrac{c_p-c_V}{\alpha_V} \, \tfrac{\hat{A}}{c_V \varrho_0 T_0} \, \partial_{xxt} \psi\,, \quad 
\varepsilon^2 \sqrt{\varepsilon} \, \tfrac{B}{A} \, \tfrac{a}{2 A c_V} \, \partial_{xxt} \big(\partial_{t} \psi\big)^2\,, \quad 
\varepsilon^2 \sqrt{\varepsilon} \, \tfrac{a}{c_V \varrho_0} \, \partial_{xxt} \big(\partial_x \psi\big)^2\,, 
\end{equation*}
omitting then $\varepsilon > 0$ and employing the relations
\begin{equation*}
\tfrac{1}{\varrho_0} \, \big(\mu_B + \tfrac{4}{3} \, \mu\big) = \nu \Lambda\,, \quad A = c_0^2 \, \varrho_0\,, \quad 
\tfrac{a}{c_V \varrho_0} = a \, \big(1 + \tfrac{B}{A}\big)\,, 
\end{equation*}
see Table~\ref{Table2}, finally leads to the nonlinear damped wave equation 
\begin{equation*}
\begin{split}
&\partial_{ttt} \psi 
- \Big(a \, \big(1 + \tfrac{B}{A}\big) + \nu \Lambda\Big) \, \partial_{xxtt} \psi 
+ a \, \big(1 + \tfrac{B}{A}\big) \, \nu \Lambda \, \partial_{xxxxt} \psi 
- c_0^2 \, \partial_{xxt} \psi \\
&\quad + a \, \big(1 + \tfrac{B}{A}\big) \, c_0^2 \, \partial_{xxxx} \psi
+ \partial_{tt} \Big(\tfrac{1}{2 c_0^2} \, \tfrac{B}{A} \, \big(\partial_{t} \psi\big)^2 + \big(\partial_x \psi\big)^2\Big) = 0\,, 
\end{split}
\end{equation*}
see also~\eqref{eq:BrunnhuberJordanKuznetsov};
it is remarkable that the differential operator defining the linear contributions factorises as follows 
\begin{equation*}
\Big(\partial_{t} - a \, \big(1 + \tfrac{B}{A}\big) \, \partial_{xx}\Big) \, \big(\partial_{tt} - \nu \Lambda \, \partial_{xxt} - c_0^2 \, \partial_{xx}\big) \, \psi
+ \partial_{tt} \Big(\tfrac{1}{2 c_0^2} \, \tfrac{B}{A} \, \big(\partial_{t} \psi\big)^2 + \big(\partial_x \psi\big)^2\Big) = 0\,.
\end{equation*}
\end{document}